\documentclass[final,11pt,a4paper]{amsart}

\usepackage{graphicx}
\usepackage[all]{xy}
\usepackage{enumitem}
\usepackage{amssymb}
\usepackage{latexsym}
\usepackage{amsmath}
\usepackage{mathrsfs}
\usepackage{array,booktabs}
\usepackage{esvect}

\usepackage[notref, notcite]{showkeys}  
\usepackage{xcolor}
\usepackage{hyperref}
\hypersetup{
    colorlinks,
    linkcolor={red!50!black},
    citecolor={blue!50!black},
    urlcolor={blue!80!black}
}

\newcommand{\harxiv}[1]{\href{http://arxiv.org/abs/#1}{\texttt{arXiv:#1}}}
\newcommand{\hyref}[2]{\hyperref[#2]{#1~\ref*{#2}}}
\newcommand{\hypageref}[1]{\hyperref[#1]{page~\pageref*{#1}}}

\usepackage{fullpage}

\def\clap#1{\hbox to 0pt{\hss#1\hss}}

\theoremstyle{plain}
\newtheorem{theorem}{Theorem}[section]
\newtheorem*{theorem*}{Theorem}

\newtheorem{lemma}[theorem]{Lemma}
\newtheorem{corollary}[theorem]{Corollary}
\newtheorem{proposition}[theorem]{Proposition}

\theoremstyle{definition}
\newtheorem{remark}[theorem]{Remark}
\newtheorem{example}[theorem]{Example}
\newtheorem*{naive-algorithm}{Na\"ive algorithm}
\newtheorem*{refined-algorithm}{Refined algorithm}
\newtheorem{definition}[theorem]{Definition}

\newcommand{\begintabularhammock}{ \smallskip\noindent\hspace{0.05\textwidth}
                                   \begin{tabular}{@{} p{0.15\textwidth} @{} p{0.80\textwidth} @{}} }

\newcommand{\IC}{{\mathbb{C}}}

\newcommand{\IR}{{\mathbb{R}}}

\newcommand{\IZ}{{\mathbb{Z}}}

\newcommand{\IN}{{\mathbb{N}}}

\renewcommand{\mod}[1]{\mathsf{mod}(#1)}
\DeclareMathOperator{\rep}{\mathsf{rep}}

\DeclareMathOperator{\Hom}{\mathrm{Hom}}
\DeclareMathOperator{\Ext}{\mathrm{Ext}}

\renewcommand{\setminus}{\backslash}

\DeclareMathAlphabet{\mathpzc}{OT1}{pzc}{m}{it}

\newcommand{\arrd}{ \ar@{-}[r] \ar@{=}[d] }
\newcommand{\too}{\longrightarrow}
\renewcommand{\iff}{~\Longleftrightarrow~}

\newcommand{\colmat}[2]{{\big(\genfrac{.}{.}{0pt}{1}{#1}{#2}\big) }}
\newcommand{\rowmat}[2]{\big(#1 \ #2\big)}

\renewcommand{\phi}{\varphi}
\renewcommand{\epsilon}{\varepsilon}

\newcommand{\arr}{\ar@{~}[r]}
\newcommand{\arrr}{\ar@{~}[rr]}

\newcommand{\arudd}{\ar[ur] \ar@{.}[dr]}
\newcommand{\aruu}{\ar@{.>}[uuurrr]}
\newcommand{\arurr}{\ar@{.>}[ur] \ar@{.}[rr]}

\newcommand{\bib}[6]{{\bibitem{#2} #3: {\emph{#4},} #5#6.}}

\usepackage{tikz}
\usetikzlibrary{shapes.geometric,positioning}
\usetikzlibrary{arrows,decorations.pathmorphing,decorations.pathreplacing}
\usetikzlibrary{decorations.markings}
\tikzset{vertex/.style={circle,fill=black,inner sep=1pt,outer sep=2pt},
         tinyvertex/.style={font=\scriptsize,minimum size=6pt},
         smallvertex/.style={inner sep=1pt, font=\small},
         >=stealth',
         leadsto/.style={-angle 90,decorate,decoration=snake,very thick},
         cut/.style={decorate,decoration=saw,very thick}}
\tikzset{
    partial ellipse/.style args={#1:#2:#3}{
        insert path={+ (#1:#3) arc (#1:#2:#3)}
    }
}

\usepackage[all]{xy}
\usepackage{tikz-cd}
\usepackage{mathrsfs}
\usepackage{pgfplots}
\pgfplotsset{compat=1.15}
\usepackage{mathrsfs}
\usetikzlibrary{arrows}
\usepackage{arydshln}
\def\temp{&} \catcode`&=\active \let&=\temp

\begin{document}

\title{An isometry theorem for persistent homology of circle-valued functions}

\author{Nathan Broomhead}
\author{Mariam Pirashvili}

\begin{abstract}
This paper explores persistence modules for circle-valued functions, presenting a new extension of the interleaving and bottleneck distances in this setting. We propose a natural generalisation of barcodes in terms of arcs on a geometric model associated to the derived category of quiver representations. The main result is an isometry theorem that establishes an equivalence between the interleaving distance and the bottleneck distance for circle-valued persistence modules.\end{abstract}

\begingroup\renewcommand\thefootnote{}%
%
\addtocounter{footnote}{-2}\endgroup

\maketitle

\vspace{-2ex}

{\small
\setcounter{tocdepth}{1}
\tableofcontents
}

\vspace{-4ex}

\addtocontents{toc}{\protect{\setcounter{tocdepth}{-1}}}  

\section*{Introduction}
\addtocontents{toc}{\protect{\setcounter{tocdepth}{1}}}   

Persistent homology, introduced in \cite{Edel}, is a central concept in topological data analysis, providing a robust framework for understanding the shape and structure of data through the lens of topology. By examining how topological features emerge and persist across different scales, it can capture important patterns in complex, high-dimensional data. It is particularly useful in situations where data is non-linear and so linear approximations are not appropriate. For this reason, it has found widespread use in a variety of disciplines, ranging from biology and medicine to materials science, offering insight into geometric features of data that may not be immediately obvious through traditional statistical methods.
\subsection*{Classical persistence}
In classical persistent homology, a set of topological spaces $\{X_t\}_{t \in\IR}$, is constructed from the data, with the property that there exist maps $X_t \hookrightarrow X_{t'}$ whenever $t\leq t'$. For example, if there is a continuous function $f:X \to \IR$ from some topological space $X$, then we might consider the sublevel filtration where $X_t = f^{-1}(-\infty,t]$ and the map $X_t \hookrightarrow X_{t'}$ is the natural subspace inclusion. This induces a filtration of homology groups in each degree, with coefficients in a field $\mathbf{k}$. Under certain tame conditions, there are only finitely many `singular' values of the real parameter $t$ at which the homology groups change, and so for each $n \in \IZ_{\geq 0}$ there is a finite sequence of vector spaces with linear maps \[ H_n(X_{t_1}, \mathbf{k}) \to  H_n(X_{t_2}, \mathbf{k}) \to  H_n(X_{t_3}, \mathbf{k}) \to \dots \to H_n(X_{t_{k-1}}, \mathbf{k}) \to H_n(X_{t_k}, \mathbf{k}). \]
This is called a \textit{persistence module} and is a representation of the equioriented quiver of type~$A_k$:
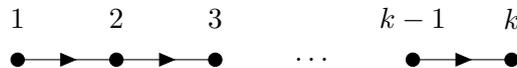
\begin{center}
\definecolor{ttqqqq}{rgb}{0,0,0}
\begin{tikzpicture}[line cap=round,line join=round,>=triangle 45,x=1.3cm,y=1.3cm]

\begin{scope}[ decoration={
    markings,
    mark=at position 0.6 with {\arrow{>}}}
    ] 
\draw[postaction={decorate}] (3,2) -- (4,2);
\draw[postaction={decorate}] (4,2) -- (5,2);
\draw[postaction={decorate}] (7,2) -- (8,2);
\end{scope}

\draw [fill=ttqqqq] (3,2) circle (2.5pt);
\draw[color=ttqqqq] (3.,2.42) node {$1$};
\draw [fill=ttqqqq] (4,2) circle (2.5pt);
\draw[color=ttqqqq] (4.,2.42) node {$2$};
\draw [fill=ttqqqq] (5,2) circle (2.5pt);
\draw[color=ttqqqq] (5.,2.42) node {$3$};
\draw [fill=ttqqqq] (7,2) circle (2.5pt);
\draw[color=ttqqqq] (7.,2.42) node {$k-1$};
\draw [fill=ttqqqq] (8,2) circle (2.5pt);
\draw[color=ttqqqq] (8.,2.42) node {$k$};
\draw [fill=ttqqqq] (6,2) node {$\dots$};
\end{tikzpicture}
\end{center}

An important result due to Gabriel \cite{Gabriel}, shows that such a representation decomposes uniquely as a direct sum of `interval representations'. We can therefore associate to any persistence module a multi-set of intervals of the form $[a,b)$ for some $1 \leq a < b \leq k+1$. The longer intervals correspond to topological features that are said to `persist' longer. This multi-set is referred to as the barcode of the persistence module and it is a computer-friendly summary of the topological structure of a data set. See Figure~\ref{a3 figure} for an example of a barcode in the context of sublevel persistence. A persistence diagram displays the same information in a graphical format. There is a (pseudo)-metric on the space of persistence modules called the interleaving distance $d_I$ introduced in \cite{Chaz} and a metric on the space of barcodes called the bottleneck distance $d_B$ (see \cite{Bottle}). In applications, these allow the topological features of two data sets to be compared in terms of a scalar. There are a number of `algebraic stability' results comparing these distances (see \cite{Chaz, BotLes, Les}) culminating in an isometry theorem which proved that the map $\mathcal{B}$ that takes a persistence module $M$ to its corresponding barcode $\mathcal{B}(M)$ is an isometry. That is, for any persistence modules $M,N$ then \[ d_I(M,N) = d_B(\mathcal{B}(M), \mathcal{B}(N)).\] In practice, this is important, as the bottleneck distance is easier to compute, while the interleaving distance can be considered more fundamental as it has a universality property \cite{Les}.

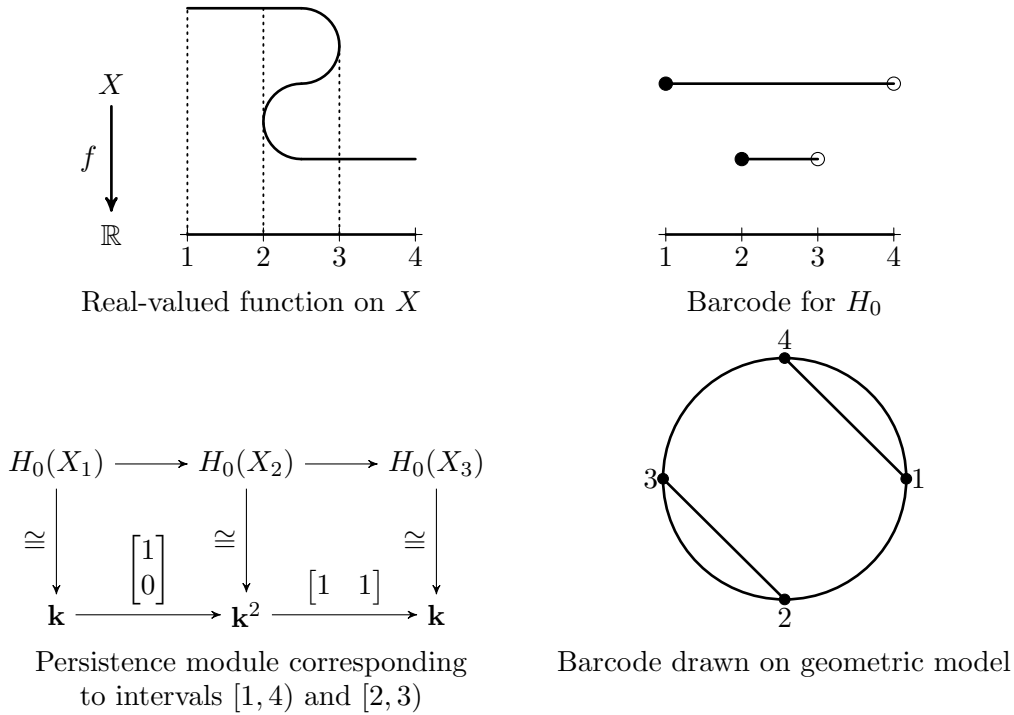
\begin{figure}[ht]
 \begin{center}
 \begin{tabular}{ccc}
 \definecolor{ttqqqq}{rgb}{0,0,0}
\begin{tikzpicture}[line cap=round,line join=round,x=1cm,y=1cm]
\draw [line width=1pt,color=ttqqqq] (2,3)-- (5,3);
\draw [line width=1pt,color=ttqqqq] (2,6)-- (3.5,6);
\draw [shift={(3.5,5.5)},line width=1pt,color=ttqqqq]  plot[domain=-1.5707963267948966:1.5707963267948966,variable=\t]({1*0.5*cos(\t r)+0*0.5*sin(\t r)},{0*0.5*cos(\t r)+1*0.5*sin(\t r)});
\draw [shift={(3.5,4.5)},line width=1pt,color=ttqqqq]  plot[domain=1.5707963267948966:4.71238898038469,variable=\t]({1*0.5*cos(\t r)+0*0.5*sin(\t r)},{0*0.5*cos(\t r)+1*0.5*sin(\t r)});
\draw [line width=1pt,color=ttqqqq] (3.5,4)-- (5,4);
\draw [->,line width=1pt,color=ttqqqq] (1.,4.7) -- (1.,3.3);
\draw [line width=0.7pt,dotted,color=ttqqqq] (2,3)-- (2,6);
\draw [line width=0.7pt,dotted,color=ttqqqq] (3,6)-- (3,3);
\draw [line width=0.7pt,dotted,color=ttqqqq] (4,5.5)-- (4,3);
\draw [color=ttqqqq] (2,3)-- ++(-2.5pt,0 pt) -- ++(5pt,0 pt) ++(-2.5pt,-2.5pt) -- ++(0 pt,5pt);
\draw[color=ttqqqq] (2.,2.7) node {$1$};
\draw [color=ttqqqq] (3,3)-- ++(-2.5pt,0 pt) -- ++(5pt,0 pt) ++(-2.5pt,-2.5pt) -- ++(0 pt,5pt);
\draw[color=ttqqqq] (3.,2.7) node {$2$};
\draw [color=ttqqqq] (4,3)-- ++(-2.5pt,0 pt) -- ++(5pt,0 pt) ++(-2.5pt,-2.5pt) -- ++(0 pt,5pt);
\draw[color=ttqqqq] (4.,2.7) node {$3$};
\draw [color=ttqqqq] (5,3)-- ++(-2.5pt,0 pt) -- ++(5pt,0 pt) ++(-2.5pt,-2.5pt) -- ++(0 pt,5pt);
\draw[color=ttqqqq] (5.,2.7) node {$4$};
\draw[color=ttqqqq] (0.7,4.0) node {$f$};
\draw[color=ttqqqq] (1.0,3.0) node {$\IR$};
\draw[color=ttqqqq] (1.0,5.0) node {$X$};
\end{tikzpicture} & &\begin{tikzpicture}[line cap=round,line join=round,x=1cm,y=1cm]
\draw [line width=1pt] (1,3)-- (4,3);
\draw [line width=1pt] (1,5)-- (4,5);
\draw [line width=1pt] (2,4)-- (3,4);
\draw [color=black] (1,3)-- ++(-2.5pt,0 pt) -- ++(5pt,0 pt) ++(-2.5pt,-2.5pt) -- ++(0 pt,5pt);
\draw[color=black] (1,2.7) node {$1$};
\draw [color=black] (4,3)-- ++(-2.5pt,0 pt) -- ++(5pt,0 pt) ++(-2.5pt,-2.5pt) -- ++(0 pt,5pt);
\draw[color=black] (4,2.7) node {$4$};
\draw [color=black] (2,3)-- ++(-2.5pt,0 pt) -- ++(5pt,0 pt) ++(-2.5pt,-2.5pt) -- ++(0 pt,5pt);
\draw[color=black] (2,2.7) node {$2$};
\draw [color=black] (3,3)-- ++(-2.5pt,0 pt) -- ++(5pt,0 pt) ++(-2.5pt,-2.5pt) -- ++(0 pt,5pt);
\draw[color=black] (3,2.7) node {$3$};
\draw [fill=black] (1,5) circle (2.5pt);
\draw [color=black] (4,5) circle (2.5pt);
\draw [fill=black] (2,4) circle (2.5pt);
\draw [color=black] (3,4) circle (2.5pt);
\end{tikzpicture} \\
 Real-valued function on $X$ &  & Barcode for $H_0$ \\
\begin{tikzpicture}[scale=1]

    \node (A) at (0, 2) {$H_0(X_1)$};
    \node (B) at (2.5, 2) {$H_0(X_2)$};
    \node (C) at (5, 2) {$H_0(X_3)$};
    
    \node (D) at (0, 0) {$\mathbf{k}$};
    \node (E) at (2.5, 0) {$\mathbf{k}^2$};
    \node (F) at (5, 0) {$\mathbf{k}$};
    
    \draw[->] (A) -- (B) node[midway, above] {$\phantom{.}$};
    \draw[->] (B) -- (C) node[midway, above] {$\phantom{.}$};
    \draw[->] (D) -- (E) node[midway, above] {$\begin{bmatrix} 1 \\ 0 \end{bmatrix}$};
    \draw[->] (E) -- (F) node[midway, above] {$\begin{bmatrix} 1 & 1 \end{bmatrix}$};
    
    \draw[->] (A) -- (D) node[midway, left] {$\cong$};
    \draw[->] (B) -- (E) node[midway, left] {$\cong$};
    \draw[->] (C) -- (F) node[midway, left] {$\cong$};
\end{tikzpicture} & & \definecolor{wewdxt}{rgb}{0.,0,0.}
\begin{tikzpicture}[line cap=round,line join=round,>=triangle 45,x=0.8cm,y=0.8cm]
\draw [line width=1pt] (0,0) circle (0.8*2cm);
\draw [line width=1pt] (2,0)-- (0,2);
\draw [line width=1pt] (0,-2)-- (-2,0);
\draw [fill=wewdxt] (0,2) circle (2pt);
\draw[color=wewdxt] (0.,2.3) node {$4$};
\draw [fill=wewdxt] (0,-2) circle (2pt);
\draw[color=wewdxt] (0.,-2.3) node {$2$};
\draw [fill=wewdxt] (-2,0) circle (2pt);
\draw[color=wewdxt] (-2.2,0.) node {$3$};
\draw [fill=wewdxt] (2,0) circle (2pt);
\draw[color=wewdxt] (2.2,0.) node {$1$};
\end{tikzpicture} \\
 Persistence module corresponding  & &  Barcode drawn on geometric model \\
 to intervals $[1,4)$ and $[2,3)$ & &   
 \end{tabular}
 \end{center}
 \caption{Simple example showing sub-level persistent homology, the interval  decomposition and barcodes, for a topological space $X$ with a map to $\IR$}

 \label{a3 figure}
\end{figure}

\subsection*{Zig-zag persistence}
A key generalisation of classical persistent homology is the notion of zig-zag persistence \cite{Car}. Instead of considering a nested sequence of topological spaces, one instead considers a sequence of topological spaces in which the maps alternate in direction. This type of behaviour occurs naturally, for example when taking intersections or unions: 
\[ U \hookleftarrow U \cap V \hookrightarrow V  \qquad \text{or} \qquad U \hookrightarrow U \cup V \hookleftarrow V \]
where $U$ and $V$ are topological spaces.
Under tame conditions, persistence modules constructed in this setting are again representations of a quiver of type~$A_k$, but the quiver now has a `zig-zag' orientation of the arrows. Further generalisations allow for arbitrary orientations.

There are again algebraic stability theorems for zig-zag persistence. These require the notions of the interleaving and bottleneck distances to be defined to this setting. In the literature this has not been done directly, but there are various ways to induce distances using (higher dimensional analogues of) the classical case. For example, in \cite{BotLes}, Botnan and Lesnick introduce a version of these distances for certain multidimension persistence modules and then consider the corresponding distances induced by a fully-faithful functor from zig-zag persistence modules to $2$-d persistence modules.
They then prove an algebraic stability result, which was strengthened to an isometry theorem in \cite{Bjer}. A different approach is taken in \cite{zig}, where the authors use an equivalence between the derived category of representations of any zig-zag $A_n$ quiver and the derived category of representations of the equioriented $A_n$ quiver. This allows a derived version of an isometry theorem to be proved.

A further generalization, which includes zig-zag persistence as a special case, comes from studying ladders -- quivers with commutation relations built from copies of $A_n$  quivers stacked in layers. Such quivers with relations are representation-finite under strong size constraints \cite{EH}. 

 Zig-zag persistence and its generalisations, provide a more flexible approach than classical persistence, capturing the intricacies of topological features in settings where the data cannot be described by a simple filtration. Zig-zag persistence is also more readily generalisable, for example to circle-valued functions $f: X \to S^1$ which are the main focus of this paper.

\subsection*{Circle-valued persistence} Circle-valued functions provide a natural extension of real-valued functions, where instead of measuring values in a linear codomain such as $\IR$, the values lie on a circle. This opens up new possibilities for analysing data in settings where the underlying structure is periodic or has a direction associated to it. Circle-valued functions have found applications in areas, such as the analysis of periodic phenomena in time series data, understanding cyclic structures in biological systems (e.g., circadian rhythms), or understanding implications of wind direction in environmental sciences. 

There are existing results on circle-valued maps in the context of persistent homology. In particular, the work of Burghelea and Dey \cite{Dey} introduces barcodes and Jordan blocks for persistence modules by considering representations of a zig-zag cyclic quiver of type $\tilde{A_n}$. They propose an algorithm that computes these invariants for a given persistence module. Building on this, \cite{Haller} considers a version of the bottleneck distance and proves a geometric version of stability that relates this bottleneck distance to a metric on the space of tame maps from $X$ to $S^1$. 
We now explain this result in slightly more detail:

\noindent For tame scalar-valued functions $f, g: K\rightarrow \mathbb{R}$, the classical stability theorem asserts that for each homological degree $r$, the bottleneck distance between their persistence diagrams is bounded above by the uniform norm:

\[ d_B(Dgm_r(f), Dgm_r(g)) \leq \| f-g\|_\infty.\]
This implies that the map $f \mapsto Dgm_r(f)$ is continuous. It is a version of this statement that is generalised by Burghelea and Haller in \cite{Haller}.
To formulate their stability result, recall that any continuous map $f:K\to S^1$ determines an integral cohomology class $\xi_f\in H^1(K; \mathbb{Z})$ by pulling back a fixed generator of $H^1(S^1; \mathbb{Z})\cong\mathbb{Z}$. Homotopic maps $f_1,f_2:K\to S^1$ determine the same class $\xi_{f_1} = \xi_{f_2}$. In fact, for a simplicial complex, this assignment induces a bijection between the set of homotopy classes of maps $K\to S^1$ and $H^1(K; \mathbb{Z})$. 

Let $C(K, S^1)$ denote the space of continuous maps from $K$ to $S^1$, equipped with the compact open topology.  For a fixed class $\xi \in H^1(K; \mathbb{Z})$, let $C_\xi (K, S^1)$ be the connected component of maps corresponding to $\xi$, and let $C_{\xi,t} (K, S^1) \subseteq C_\xi (K, S^1)$ denote the subspace of tame maps.

Using level-set persistence, each circle-valued map $f$ determines a representation of a zig-zag quiver of type $\tilde{A}$ in each homological degree $r$. To this representation, the authors assign a configuration of points $C_r(f)$ on the complex plane, where each point corresponds to an indecomposable summand of the representation. They prove that the assignment

\[C_{\xi,t} (K, S^1)\ni f\mapsto C_r(f)\]

\noindent is continuous and extends to a continuous map $C_{\xi,t} (K, S^1) \to \mathbb{C}^{n-1}\times (\mathbb{C} \setminus 0)$.

They also reinterpret the classical geometric stability theorem of scalar-valued persistence in the circle-valued context. They first restrict the configurations $C_r(f)$ to consider only mixed (closed-open or open-closed) bars and denote this subset by $C_r^m(f)$. Geometric stability can then be reformulated as the statement that $f\mapsto C_r^m(f)$ is a continuous map.

In this paper, we instead look at the notion of algebraic stability. We extend the classical definition of interleaving distance to apply to persistence modules that are representations of any gentle hereditary algebra. Our new definition encompasses all orientations of the $A_n$ quiver, including both classical and zig-zag persistence, as well as all non-cyclic orientations of the $\tilde{A}_n$ quiver, which covers circle-valued persistence. We use the structure of the Auslander-Reiten quiver and generalise the $\delta$-shift using the Auslander-Reiten translation on the category.

We take a slightly different approach from \cite{Dey} and define barcodes as multisets of arcs and closed curves on the `geometric model' for the derived category.
The theory of geometric models has recently been developed for derived categories of any gentle algebra \cite{OPS} and they have attracted significant attention for their ability to concisely describe indecomposable objects and morphisms in the category. The geometric model for $A_n$ is a disc with $n+1$ marked boundary points, while that of $\tilde{A}_n$ is an annulus with marked points on both boundary components. 
In Example~\ref{ExampleBarcode}, we show that our notion of barcode naturally generalises the classical barcode, in which an interval $[a,b)$ corresponds to an arc between the vertices $a$ and $b$ on the disc (see Figure~\ref{a3 figure}). We then define a bottleneck distance on these barcodes using simple operations on the geometric model. 

Notably, our definitions of interleaving and bottleneck distances can be computed directly from persistence modules and barcodes, respectively, without the need to lift to multidimensional persistence or rely on equivalences to relate the modules to equioriented ones. Using these definitions we prove a strong isometry theorem for persistence modules of type $\tilde{A}$.

\begin{theorem*}[Theorem~\ref{isometry}]  Suppose $M,N$ are persistence modules of type $\tilde{A}$. Then 
\[ d_I(M,N) = d_B(\mathcal{B}(M), \mathcal{B}(N)).\]
    \end{theorem*}

\subsection*{Other related constructions}
Beyond the frameworks discussed above, several other constructions have been developed to generalise classical persistence, and versions of stability have also been considered in many of these settings. For instance, representations of posets that are not discrete -- such as the totally ordered real line $(\IR, \leq)$ -- have been considered, based on foundational work of Crawley-Bowvey \cite{CB}. Related ideas appear in the theory of continuous quivers, introduced in \cite{IRT, HR} and there is an isometry theorem for continuous quivers of type $A$ \cite{O}. Based on this result, an isometry theorem for cyclically oriented continuous quivers of type $\tilde{A}$ was recently presented in \cite{GZ}. The quivers that we consider in this paper are finite and not cyclically oriented, and so our results are complementary.

In a different direction, generalized persistence modules have been formalised as functors from a preordered set to a target category, as in \cite{BSS}. Many notions of persistence, including classical and zig-zag can be considered as special cases of this, with a specific choice of preordered set and target category. There is a general notion of interleaving distance and, using this, an isometry theorem is proved in \cite{MM} for a subclass of examples. This subclass does not include the examples studied in this paper.

\subsection*{Structure of the paper} The first sections cover key material that is necessary to understand the main results and their proof. We aim to make this as accessible as possible to those without a background in representation theory or derived categories. 
Sections~\ref{PrelimQuiver} and \ref{PrelimDerived} introduce definitions and notation from quiver representation theory and derived categories respectively. In particular, they include structural results on the Auslander-Reiten quivers in the $A_n$ and $\tilde{A}_n$ cases. 

Section~\ref{Interleaving} begins by recalling the classical definition of $\delta$-interleaving and then proposes a more generally applicable definition. Proposition~\ref{InterleavingCorresp} shows how this extends the classical case. Proposition~\ref{Interleavingmetric} then shows that the interleaving distance is a (possibly infinite) metric. Section~\ref{GeometricM} gives a short introduction to the geometric models that are used in our definition of barcodes. Operations are defined which are analogous to moving the endpoints of an interval in a classical barcode. Section~\ref{Bottleneck} begins by recalling the classical definition of $\delta$-matching of interval modules. This is then extended to barcodes in the $\tilde{A}_n$ case (Definition~\ref{matching}). Lemma~\ref{bottlemetric} shows that the bottleneck distance is a (possibly infinite) metric. 
Sections~\ref{StabilityNonreg} and \ref{StabilityReg} contain the proof of the main isometry theorem, considering first non regular and then regular persistence modules. 

Section~\ref{BurgheleaBackground} acts as a short dictionary that compares our barcodes with the barcode and Jordan blocks from \cite{Dey}. 
Finally, Section~\ref{Examples} illustrates the theory developed in this paper with some examples.

%

\section{Preliminaries from representation theory}\label{PrelimQuiver}
\noindent
Let $\mathbf{k}$ denote an algebraically closed field. In this paper all vector spaces, algebras and categories are assumed to be $\mathbf{k}$-linear unless otherwise stated. We consider finite dimensional path algebras of quivers $Q$ of type $A$ or type $\tilde{A}$, so the underlying graph is either linear or a cycle respectively. In the cycle case, we assume that the edges of $Q$ are not cyclically oriented, so the path algebra $\mathbf{k}Q$ is finite dimensional over $\mathbf{k}$. A representation $(M, \varphi)$ of a quiver $Q$, consists of a vector space $M_x$ for each vertex $x$ of $Q$ and linear maps $\varphi_\alpha: M_x \to M_y$ for each arrow $\alpha: x \to y$ in $Q$. The definitions and results from this section can be found in many reference texts such as \cite{ASS}.

\subsection{The module category}
We denote by $\mod{\mathbf{k}Q}$ the category of finitely generated right $\mathbf{k}Q$-modules over $\mathbf{k}$, and by  $\rep_{\mathbf{k}}(Q)$ the category of finite dimensional representations of the quiver $Q$. There is a well-known equivalence of categories
\[F: \mod{\mathbf{k}Q} \stackrel{\cong}{\too} \rep_{\mathbf{k}}(Q) \]
and so we will often move between these categories without comment. 

These categories are $\mathbf{k}$-linear, abelian, 
hereditary (so they have global dimension $0$ or $1$) and Krull-Schmidt. This final property means that every object in the category decomposes as a finite direct sum of indecomposable objects, and furthermore, this decomposition is unique up to permutation and isomorphism of the summands. 

\subsection{The Auslander-Reiten theory} 
Auslander-Reiten theory \cite{AR} is a powerful tool which allows certain computations and structures to be seen very explicitly in the categories. In this article, we will state and use some key properties, but we refer the reader to \cite{ARS} for a comprehensive treatment.

\subsubsection{The Auslander-Reiten quiver}
We can draw a (possibly infinite) directed graph which captures a lot of information about $\mod{\mathbf{k}Q}$. This is called the Auslander-Reiten quiver of the category. It has vertices corresponding to isomorphism classes of indecomposable objects and arrows corresponding to irreducible morphisms. More precisely, if $M,N$ are indecomposable objects then there is an arrow $M \too N$, for each element in a basis for the $\mathbf{k}$-vectorspace of irreducible morphisms from $M$ to $N$ (these are non-invertible homomorphisms which don't admit a non-trivial factorisation). 

\subsubsection{The Auslander-Reiten translate}
There is a functor $\tau$ called the Auslander-Reiten translate which will play an important role in our theory. 
$\tau$ commutes with finite direct sums, so we can understand how it acts on an arbitrary module by looking at the action on indecomposable summands.
\begin{proposition}[\cite{ASS}, Proposition~2.10] \label{tauProperties}
Let $M$ and $N$ be indecomposable modules in $\mod{\mathbf{k}Q}$. Then:
\begin{itemize}
    \item $\tau M$ is zero if and only if $M$ is a projective module.
    \item $\tau^{-1} N$ is zero if and only if $N$ is an injective module.
    \item If $M$ is not projective, then $\tau M$ is indecomposable and not injective, and $\tau^{-1}\tau M \cong M$.

    
    \item If $N$ is not injective, then $\tau^{-1} N$ is indecomposable and not projective, and $\tau\tau^{-1} N \cong N$.
    \item If $M,N$ are not projective, then $M \cong N$ if and only if $\tau M \cong \tau N$.
    \item If $M,N$ are not injective, then $M \cong N$ if and only if $\tau^{-1} M \cong \tau^{-1} N$.
\end{itemize}
\end{proposition}
\noindent In particular, the Auslander-Reiten translate induces a bijection between the indecomposable objects that are not projective, and the indecomposable objects that are not injective.

\noindent For any indecomposable module $M$ that is not projective, and any indecomposable module $N$ that is not injective, there are short exact sequences:
\begin{equation}
 0 \too \tau M \stackrel{\scriptsize \colmat{f_1}{f_2}}{\too} E_1 \oplus E_2 \stackrel{\scriptsize \rowmat{g_1}{g_2}}{\too} M \too 0 \quad \text{ and } \quad  0 \too N \stackrel{\scriptsize \colmat{f_3}{f_4}}{\too} F_1 \oplus F_2 \stackrel{\scriptsize \rowmat{g_3}{g_4}}{\too} \tau^{-1} N \too 0 \end{equation}
in $\mod{\mathbf{k}Q}$ which are called `almost split' or `Ausander-Reiten' sequences. The fact that the middle term has at most two non-zero indecomposable summands holds because we are considering cases where $Q$ is of type $A$ or type $\tilde{A}$. Depending on the choice of $M$ and $N$, it is possible that one of these summands is zero. The morphisms $f_i$ and $g_j$ in these sequences are irreducible, and there are commutative squares that can be seen (up to scaling by elements in $\mathbf{k}^*$) in the Auslander-Reiten quiver. Such squares are referred to as a `mesh', and the commutation relations are called `mesh relations' \cite[Section~IV.4]{ASS}.
\begin{figure}[h]
\centering
\definecolor{ududff}{rgb}{0,0,0}
\begin{tikzpicture}[line cap=round,line join=round,>=triangle 45,x=1cm,y=1cm]
\draw [->,line width=1pt] (2.2,2.2) -- (2.8,2.8);
\draw [->,line width=1pt] (3.2,2.8) -- (3.8,2.2);
\draw [->,line width=1pt] (2.2,1.8) -- (2.8,1.2);
\draw [->,line width=1pt] (3.2,1.2) -- (3.8,1.8);
\draw [fill=ududff] (2,2) node {$\tau M$};
\draw [fill=ududff] (3,3) node {$E_1$};
\draw [fill=ududff] (3,1) node {$E_2$};
\draw [fill=ududff] (4,2) node {$M$};
\begin{scriptsize}
\draw[color=black] (2.35,2.7) node {$f_1$};
\draw[color=black] (3.6,2.7) node {$g_1$};
\draw[color=black] (2.35,1.3) node {$f_2$};
\draw[color=black] (3.6,1.3) node {$-g_2$};
\end{scriptsize}

\draw [->,line width=1pt] (6.2,2.2) -- (6.8,2.8);
\draw [->,line width=1pt] (7.2,2.8) -- (7.8,2.2);
\draw [->,line width=1pt] (6.2,1.8) -- (6.8,1.2);
\draw [->,line width=1pt] (7.2,1.2) -- (7.8,1.8);
\draw [fill=ududff] (6,2) node {$N$};
\draw [fill=ududff] (7,3) node {$F_1$};
\draw [fill=ududff] (7,1) node {$F_2$};
\draw [fill=ududff] (8,2) node {$\tau^{-1} N$};
\begin{scriptsize}
\draw[color=black] (6.35,2.7) node {$f_3$};
\draw[color=black] (7.6,2.7) node {$g_3$};
\draw[color=black] (6.35,1.3) node {$f_4$};
\draw[color=black] (7.6,1.3) node {$-g_4$};
\end{scriptsize}

\end{tikzpicture}
\end{figure}

We now look at the type $A$ and type $\tilde{A}$ cases in turn.

\subsection{The module categories of type $A$}
Let $Q$ be a quiver of type $A_n$ with any orientation of the edges. We label the vertices $1, \dots ,n$ such that $Q$ is some orientation of the graph
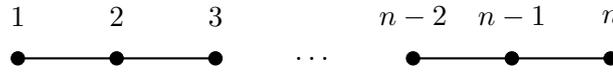
\begin{center}
\definecolor{ttqqqq}{rgb}{0,0,0}
\begin{tikzpicture}[line cap=round,line join=round,>=triangle 45,x=1.3cm,y=1.3cm]

\begin{scope}[ thick,decoration={
    markings,
    mark=at position 0.6 with {\arrow{>}}}
    ] 
\draw[] (3,2) -- (4,2);
\draw[] (4,2) -- (5,2);
\draw[] (7,2) -- (8,2);
\draw[] (8,2) -- (9,2);
\end{scope}

\draw [fill=ttqqqq] (3,2) circle (2.5pt);
\draw[color=ttqqqq] (3.,2.42) node {$1$};
\draw [fill=ttqqqq] (4,2) circle (2.5pt);
\draw[color=ttqqqq] (4.,2.42) node {$2$};
\draw [fill=ttqqqq] (5,2) circle (2.5pt);
\draw[color=ttqqqq] (5.,2.42) node {$3$};
\draw [fill=ttqqqq] (7,2) circle (2.5pt);
\draw[color=ttqqqq] (7.,2.42) node {$n-2$};
\draw [fill=ttqqqq] (8,2) circle (2.5pt);
\draw[color=ttqqqq] (8.,2.42) node {$n-1$};
\draw [fill=ttqqqq] (9,2) circle (2.5pt);
\draw[color=ttqqqq] (9.,2.42) node {$n$};
\draw [fill=ttqqqq] (6,2) node {$\dots$};
\end{tikzpicture}
\end{center}

\noindent For any $ 1 \leq a \leq b \leq n+1$ we define an interval $[a,b) = \{c \in \IZ \mid a \leq c < b \} \subset \{1, \dots, n\}$. 

\begin{definition}
The interval representation $V_{[a,b)}$ is defined to have a one-dimensional vector space at each vertex in the interval $[a,b)$ and zero dimensional vector spaces at all other vertices.
The linear map corresponding to an arrow is an isomorphism if the domain and codomain are both 1-dimensional and zero otherwise.
\end{definition}

\noindent There is a famous theorem due to Gabriel \cite{Gabriel}.
\begin{theorem}[Gabriel '72] \label{Gabriel}
    A representation of $Q$ is indecomposable if and only if it isomorphic to an interval representation. \end{theorem}

\noindent This means that each vertex of the Auslander-Reiten quiver is described by an interval in $[1,n+1)$. 
In particular, the Auslander-Reiten quiver of  $\rep_{\mathbf{k}}(Q)$ is finite and can be `knitted' explicitly starting with the simple projective objects. Figure~\ref{A4ModuleCatExamples} shows two such examples. 

\begin{figure}[t]
\centering
\definecolor{zzttqq}{rgb}{0.6,0.2,0}
\definecolor{ttqqqq}{rgb}{0.2,0,0}
\begin{tikzpicture}[line cap=round,line join=round,>=triangle 45,x=1cm,y=1cm]
\fill[line width=0pt,color=ttqqqq,fill=ttqqqq,fill opacity=0.10000000149011612] (2.5,2) -- (3,1.5) -- (6.5,5) -- (6,5.5) -- cycle;
\fill[line width=0pt,color=ttqqqq,fill=ttqqqq,fill opacity=0.10000000149011612] (11.5,2) -- (12,1.5) -- (13.5,3) -- (12.5,4) -- (13.5,5) -- (13,5.5) -- (11.5,4) -- (12.5,3) -- cycle;
\fill[line width=0pt,color=zzttqq,fill=zzttqq,fill opacity=0.10000000149011612] (6,5.5) -- (5.5,5) -- (9,1.5) -- (9.5,2) -- cycle;
\fill[line width=0pt,color=zzttqq,fill=zzttqq,fill opacity=0.10000000149011612] (15,5.5) -- (14.5,5) -- (15.5,4) -- (14.5,3) -- (16,1.5) -- (16.5,2) -- (15.5,3) -- (16.5,4) -- cycle;
\draw [->,color=ttqqqq] (3.2,2.2) -- (3.8,2.8);
\draw [->, color=ttqqqq] (4.2,3.2) -- (4.8,3.8);
\draw [->, color=ttqqqq] (5.2,4.2) -- (5.8,4.8);
\draw [->, color=ttqqqq] (4.2,2.8) -- (4.8,2.2);
\draw [->, color=ttqqqq] (5.2,2.2) -- (5.8,2.8);
\draw [->, color=ttqqqq] (5.2,3.8) -- (5.8,3.2);
\draw [->, color=ttqqqq] (6.2,2.8) -- (6.8,2.2);
\draw [->, color=ttqqqq] (7.2,2.2) -- (7.8,2.8);
\draw [->, color=ttqqqq] (6.2,3.2) -- (6.8,3.8);
\draw [->, color=ttqqqq] (6.2,4.8) -- (6.8,4.2);
\draw [->, color=ttqqqq] (7.2,3.8) -- (7.8,3.2);
\draw [->, color=ttqqqq] (8.2,2.8) -- (8.8,2.2);
\draw [->, color=ttqqqq] (12.2,2.2) -- (12.8,2.8);
\draw [->, color=ttqqqq] (12.2,3.8) -- (12.8,3.2);
\draw [->, color=ttqqqq] (12.2,4.2) -- (12.8,4.8);
\draw [->, color=ttqqqq] (13.2,4.8) -- (13.8,4.2);
\draw [->, color=ttqqqq] (14.2,4.2) -- (14.8,4.8);
\draw [->, color=ttqqqq] (13.2,3.2) -- (13.8,3.8);
\draw [->, color=ttqqqq] (13.2,2.8) -- (13.8,2.2);
\draw [->, color=ttqqqq] (14.2,2.2) -- (14.8,2.8);
\draw [->, color=ttqqqq] (15.2,3.2) -- (15.8,3.8);
\draw [->, color=ttqqqq] (14.2,3.8) -- (14.8,3.2);
\draw [->, color=ttqqqq] (15.2,2.8) -- (15.8,2.2);
\draw [->, color=ttqqqq] (15.2,4.8) -- (15.8,4.2);

\begin{scope}[ thick,decoration={
    markings,
    mark=at position 0.6 with {\arrow{>}}}
    ] 
    \draw[postaction={decorate}] (7,0.5) -- (9,0.5);
    \draw[postaction={decorate}]  (5,0.5) -- (7,0.5);
    \draw[postaction={decorate}] (3,0.5) -- (5,0.5);
    \draw[postaction={decorate}]  (13,0.5) -- (11,0.5);
    \draw[postaction={decorate}] (13,0.5) -- (15,0.5);
    \draw[postaction={decorate}]  (17,0.5) -- (15,0.5);
\end{scope}

\draw [fill=ttqqqq] (3,2)  node {$[4,5)$};
\draw [fill=ttqqqq] (4,3) node {$[3,5)$};
\draw [fill=ttqqqq] (5,2) node {$[3,4)$};
\draw [fill=ttqqqq] (6,3)  node {$[2,4)$};
\draw [fill=ttqqqq] (7,2)  node {$[2,3)$};
\draw [fill=ttqqqq] (8,3) node {$[1,3)$};
\draw [fill=ttqqqq] (9,2)  node {$[1,2)$};
\draw [fill=ttqqqq] (5,4)  node {$[2,5)$};
\draw [fill=ttqqqq] (7,4)   node {$[1,4)$};
\draw [fill=ttqqqq] (6,5)   node {$[1,5)$};
\draw [fill=ttqqqq] (12,2)  node {$[1,2)$};
\draw [fill=ttqqqq] (12,4)  node {$[3,4)$};
\draw [fill=ttqqqq] (13,3)  node {$[1,4)$};
\draw [fill=ttqqqq] (13,5) node {$[3,5)$};
\draw [fill=ttqqqq] (14,2)  node {$[2,4)$};
\draw [fill=ttqqqq] (14,4)   node {$[1,5)$};
\draw [fill=ttqqqq] (15,3) node {$[2,5)$};
\draw [fill=ttqqqq] (15,5)  node {$[1,3)$};
\draw [fill=ttqqqq] (16,2) node {$[4,5)$};
\draw [fill=ttqqqq] (16,4) node {$[2,3)$};
\draw [fill=ttqqqq] (9,0.5) circle (2.5pt);
\draw [fill=ttqqqq] (7,0.5) circle (2.5pt);
\draw [fill=ttqqqq] (5,0.5) circle (2.5pt);
\draw [fill=ttqqqq] (3,0.5) circle (2.5pt);
\draw [fill=ttqqqq] (13,0.5) circle (2.5pt);
\draw [fill=ttqqqq] (11,0.5) circle (2.5pt);
\draw [fill=ttqqqq] (15,0.5) circle (2.5pt);
\draw [fill=ttqqqq] (17,0.5) circle (2.5pt);
\end{tikzpicture}
\caption{Examples showing the Auslander-Reiten quiver of $\mod{\mathbf{k}Q}$, where $Q$ is a quiver of type $A_4$. The corresponding quiver $Q$ is drawn below each example. The objects are labelled by their supporting intervals, considered as representations of $Q$. The objects in the shaded region on the left of each diagram are the projective objects and the objects in the other shaded regions are injective. Note $[1,5)$ is both projective and injective in the equioriented example on the left, but neither in the example on the right. The quivers have been drawn so that the Auslander-Reiten translate $\tau$ moves an object one step to the left.} \label{A4ModuleCatExamples}
\end{figure}
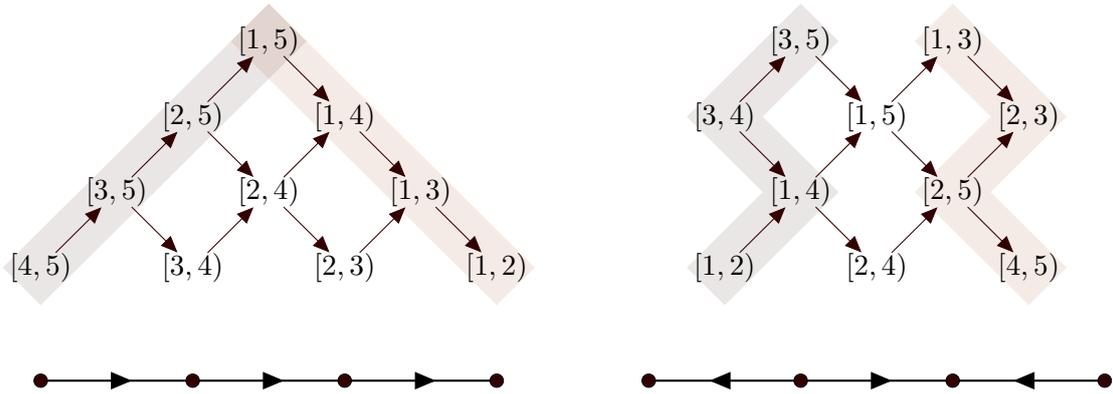

\begin{remark}\label{AtypeARmesh}
    If we consider the \emph{equioriented} quiver of type $A_n$, it can be seen that the squares of the Auslander-Reiten quiver are of the form:
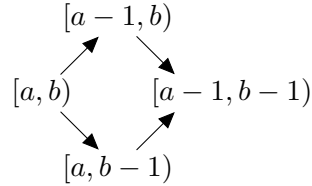
\begin{center}
    
\definecolor{ududff}{rgb}{0,0,0}
\begin{tikzpicture}[line cap=round,line join=round,>=triangle 45,x=1cm,y=1cm]
\draw [->] (2.25,2.25) -- (2.75,2.75);
\draw [->] (3.25,2.75) -- (3.75,2.25);
\draw [->] (2.25,1.75) -- (2.75,1.25);
\draw [->] (3.25,1.25) -- (3.75,1.75);
\draw [fill=ududff] (2,2) node {${[a,b)}$};
\draw [fill=ududff] (3,3) node {${[a-1,b)}$};
\draw [fill=ududff] (3,1) node {${[a,b-1)}$};
\draw [fill=ududff] (4.5,2) node {${[a-1,b-1)}$};
\end{tikzpicture}
\end{center}
where $1<a<b \leq n+1$ and the object $[a,a)$ is considered to be zero. 
The arrows pointing diagonally up are monomorphisms which include a representation supported on a shorter interval into a longer one. Dually, the arrows pointing diagonally down, are epimorphisms, restricting a representation to one supported on a sub-interval. 
\end{remark}
\subsection{The module categories of type $\tilde{A}$}
Let $Q$ be a quiver whose underlying graph is a cycle of order $n$, and which has $p>0$ edges oriented in the clockwise direction and $q>0$ arrows in the anticlockwise direction. We say that $Q$ has type $\tilde{A}_{p,q}$ and note that $p+q=n$. 

The Auslander-Reiten quiver of $\mod{\mathbf{k}Q}$ has three types of connected component called the preprojective component $\mathcal{P}$, the preinjective component $\mathcal{I} $ and the regular components $\mathcal{R}$. It can be represented pictorially as follows:
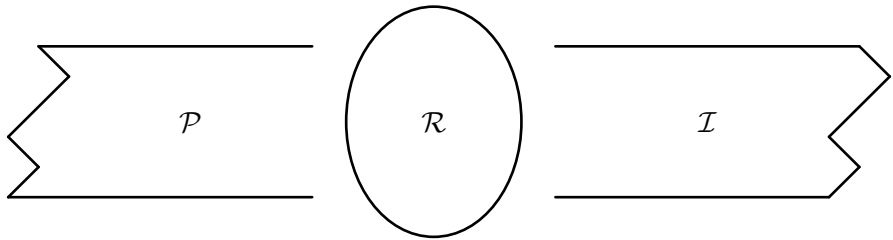
\begin{center}
    \definecolor{ududff}{rgb}{0,0,0}
\begin{tikzpicture}[line cap=round,line join=round,>=triangle 45,x=0.2cm,y=0.2cm]
\draw [line width=1pt] (-10,-6)-- (-8,-4);
\draw [line width=1pt] (-8,-4)-- (-10,-2);
\draw [line width=1pt] (-10,-2)-- (-6,2);
\draw [line width=1pt] (-6,2)-- (-8,4);
\draw [line width=1pt] (-8,4)-- (10,4);
\draw [line width=1pt] (-10,-6)-- (10,-6);
\draw [line width=1pt] (26,4)-- (46,4);
\draw [line width=1pt] (46,4)-- (48,2);
\draw [line width=1pt] (48,2)-- (44,-2);
\draw [line width=1pt] (44,-2)-- (46,-4);
\draw [line width=1pt] (46,-4)-- (44,-6);
\draw [line width=1pt] (44,-6)-- (26,-6);
\draw [rotate around={90:(18,-1)},line width=1pt] (18,-1) ellipse (0.2*7.631161284688725cm and 0.2*5.76494774936704cm);
\draw[color=ududff] (2,-1) node {$\mathcal{P}$};
\draw[color=ududff] (18,-1) node {$\mathcal{R}$};
\draw[color=ududff] (36,-1) node {$\mathcal{I}$};
\end{tikzpicture}
\end{center}

\noindent Morphisms do exist between some of these components, but these morphisms are not irreducible (they factor through in infinitely many indecomposable objects) so they are not captured by the Auslander-Reiten quiver. Non-zero morphisms only exist from left to right in the diagram, so there are no non-zero morphisms from an object in $\mathcal{I}$ to any object in $\mathcal{P}$ or $\mathcal{R}$, and there are no non-zero morphisms from an object in $\mathcal{R}$ to any object in $\mathcal{P}$. All morphisms within a given connected component \textit{are} described by the Auslander-Reiten quiver (\cite{Liu}, Theorem~2.2) and any component with this property is said to be a standard component. 

Each component of the Auslander-Reiten quiver is an example of either an `infinite translation quiver' or a quotient of such a quiver. We define this for any connected acyclic quiver although, in our examples, this will always be either a quiver of type $\tilde{A}_{p,q}$, or the equioriented $A_\infty $ quiver with a source.
\begin{definition}[\cite{ASS}, Chapter~VIII]
    Let $Q=(Q_0,Q_1)$ be a connected and acyclic quiver. The infinite translation quiver $(\IZ Q, \tau)$ has vertex set \[ (\IZ Q)_0 = \IZ \times Q_0 = \{ (n,x) \mid n \in \IZ, \; x \in Q_0 \}\]
    and for each arrow $\alpha : x \too y$ in $Q_1$, there are two arrows 
    \[ (n,\alpha): (n,x) \too (n,y) \quad \text{and} \quad (n,\alpha'): (n+1,y) \too (n,x)\]
    in $(\IZ Q)_1$. Furthermore, all arrows in $(\IZ Q)_1$ are of this form. The translation $\tau$ is defined by
    \[ \tau(n,x) = (n+1, x) \quad \text{for all} \quad (n,x) \in (\IZ Q)_0.\]
    We define $\IN Q$ (and $(-\IN) Q$) to be the full subquiver with vertices $(n,x) \in (\IZ Q)_0$ such that $n \geq 0$ (respectively $n \leq 0$). Furthermore, we define the zero section of $ \IZ Q$ to be the full subquiver with vertices in $\IN Q \cap (-\IN) Q$.
\end{definition}

\subsubsection{The preprojective and preinjective components}
We relate the preprojective and preinjective components of the Auslander-Reiten quiver to certain infinite translation quivers using the following result from \cite[Section~VIII.2]{ASS}.
\begin{proposition}\label{nonregcomponents}
    Suppose $Q$ is a quiver of type $ \tilde{A_n}$. The Auslander-Reiten quiver of $\mod{\mathbf{k}Q}$ contains:
    \begin{enumerate}
        \item a single preprojective component $\mathcal{P}$ which is isomorphic to $(-\IN)Q^{\text{op}}$. The translation corresponds to the Auslander-Reiten translate $\tau$. The full subquiver of all indecomposable projective modules is isomorphic to the zero-section of $(-\IN)Q^{\text{op}}$.
        \item a single preinjective component $\mathcal{I}$ which is isomorphic to $\IN Q^{\text{op}}$. The translation corresponds to the Auslander-Reiten translate $\tau$. The full subquiver of all indecomposable injective modules is isomorphic to the zero-section of $\IN Q^{\text{op}}$.
    \end{enumerate}
\end{proposition}
\noindent An example of a preprojective component of type $\Tilde{A}_{2,2}$ is shown in Figure~\ref{PreprojectiveComponentA4}.
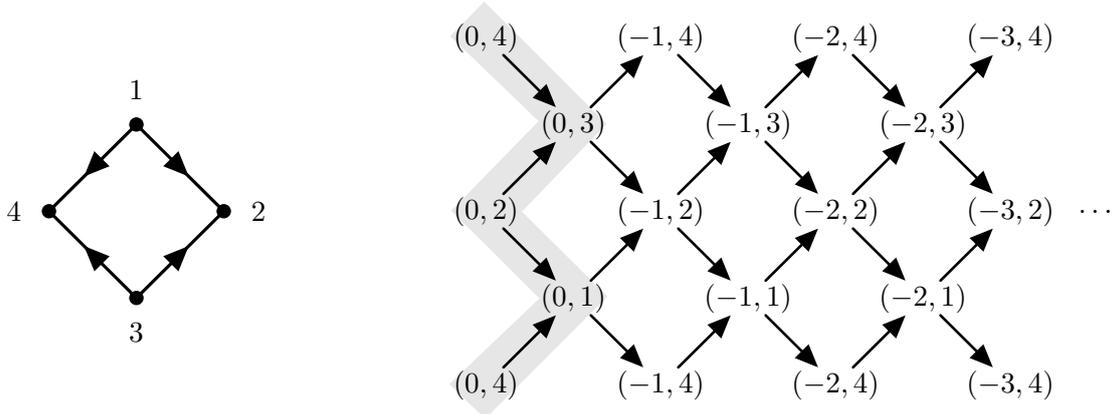
\begin{figure}
\centering
\definecolor{zzttqq}{rgb}{0,0,0}
\definecolor{ududff}{rgb}{0,0,0}
\begin{tikzpicture}[line cap=round,line join=round,>=triangle 45,x=0.23cm,y=0.23cm]
\fill[line width=1pt,color=zzttqq,fill=zzttqq,fill opacity=0.10000000149011612] (24.953292817630093,-11.968862141400432) -- (22.900371589883136,-9.96857991949317) -- (28.058994162170364,-5.020513370564676) -- (23.005649601562464,0.03283119004314617) -- (28.0063551563307,5.1914537623302985) -- (22.9530105957228,10.086881305419126) -- (24.900653811790427,12.087163527326389) -- (31.901641588465953,5.033536744811304) -- (27.05885305121672,-0.01980781579651865) -- (31.95428059430562,-4.915235358885347) -- cycle;

\begin{scope}[very thick,decoration={
    markings,
    mark=at position 0.6 with {\arrow{>}}}
    ] 
\draw[postaction={decorate}] (5,5) -- (10,0);
\draw[postaction={decorate}] (5,5) -- (0,0);
\draw[postaction={decorate}] (5,-5) -- (0,0);
\draw[postaction={decorate}] (5,-5) -- (10,0);
\end{scope}

\draw [->,line width=1pt] (26,9) -- (29,6);
\draw [->,line width=1pt] (26,1) -- (29,4);
\draw [->,line width=1pt] (31,6) -- (34,9);
\draw [->,line width=1pt] (31,4) -- (34,1);
\draw [->,line width=1pt] (36,1) -- (39,4);
\draw [->,line width=1pt] (36,9) -- (39,6);
\draw [->,line width=1pt] (26,-1) -- (29,-4);
\draw [->,line width=1pt] (26,-9) -- (29,-6);
\draw [->,line width=1pt] (31,-6) -- (34,-9);
\draw [->,line width=1pt] (31,-4) -- (34,-1);
\draw [->,line width=1pt] (36,-9) -- (39,-6);
\draw [->,line width=1pt] (36,-1) -- (39,-4);
\draw [->,line width=1pt] (41,-4) -- (44,-1);
\draw [->,line width=1pt] (41,4) -- (44,1);
\draw [->,line width=1pt] (46,1) -- (49,4);
\draw [->,line width=1pt] (46,-1) -- (49,-4);
\draw [->,line width=1pt] (51,-4) -- (54,-1);
\draw [->,line width=1pt] (51,4) -- (54,1);
\draw [->,line width=1pt] (41,6) -- (44,9);
\draw [->,line width=1pt] (46,9) -- (49,6);
\draw [->,line width=1pt] (51,6) -- (54,9);
\draw [->,line width=1pt] (41,-6) -- (44,-9);
\draw [->,line width=1pt] (46,-9) -- (49,-6);
\draw [->,line width=1pt] (51,-6) -- (54,-9);
\draw [fill=ududff] (5,5) circle (2.5pt);
\draw[color=ududff] (5,7) node {$1$};
\draw [fill=ududff] (10,0) circle (2.5pt);
\draw[color=ududff] (12,0) node {$2$};
\draw [fill=ududff] (0,0) circle (2.5pt);
\draw[color=ududff] (-2,0) node {$4$};
\draw [fill=ududff] (5,-5) circle (2.5pt);
\draw[color=ududff] (5,-7) node {$3$};
\draw [fill=ududff] (25,0) node {$(0,2)$};
\draw [fill=ududff] (30,5) node {$(0,3)$};
\draw [fill=ududff] (25,10) node {$(0,4)$};
\draw [fill=ududff] (30,-5) node {$(0,1)$};
\draw [fill=ududff] (25,-10) node {$(0,4)$};
\draw [fill=ududff] (35,-10) node {$(-1,4)$};
\draw [fill=ududff] (35,0) node {$(-1,2)$};
\draw [fill=ududff] (35,10) node {$(-1,4)$};
\draw [fill=ududff] (40,5) node {$(-1,3)$};
\draw [fill=ududff] (40,-5) node {$(-1,1)$};
\draw [fill=ududff] (45,-10) node {$(-2,4)$};
\draw [fill=ududff] (45,0) node {$(-2,2)$};
\draw [fill=ududff] (45,10) node {$(-2,4)$};
\draw [fill=ududff] (50,5) node {$(-2,3)$};
\draw [fill=ududff] (50,-5) node {$(-2,1)$};
\draw [fill=ududff] (55,-10) node {$(-3,4)$};
\draw [fill=ududff] (55,0) node {$(-3,2)$};
\draw [fill=ududff] (55,10) node {$(-3,4)$};
\draw [fill=ududff] (60,0) node {$\dots$};
\end{tikzpicture}
\caption{Example showing $(-\IN)Q^{\text{op}}$ for the quiver $Q$ shown on the left. This is isomorphic to the preprojective component of Auslander-Reiten quiver of $\mod{\mathbf{k}Q}$. The objects in the shaded region on the left of the diagram are in the zero section, and correspond to the indecomposable projective objects in $\mod{\mathbf{k}Q}$. Note that in this diagram, the vertices drawn along the top and bottom of $(-\IN)Q^{\text{op}}$ should be identified.} \label{PreprojectiveComponentA4}
\end{figure}

\subsubsection{The regular components}
The regular part of the Auslander-Reiten quiver of $\mod{\mathbf{k}Q}$ is made up of components called tubes. These are isomorphic to quotients of an infinite translation quiver. Let $\vv{A}_\infty$ be the equioriented quiver of type $A_\infty$:

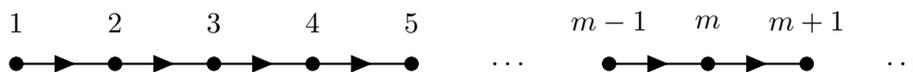
\begin{center}
\definecolor{ttqqqq}{rgb}{0,0,0}
\begin{tikzpicture}[line cap=round,line join=round,>=triangle 45,x=1.3cm,y=1.3cm]

\begin{scope}[thick,decoration={
    markings,
    mark=at position 0.6 with {\arrow{>}}}
    ] 
\draw[postaction={decorate}] (1,2) -- (2,2);
\draw[postaction={decorate}] (2,2) -- (3,2);
\draw[postaction={decorate}] (3,2) -- (4,2);
\draw[postaction={decorate}] (4,2) -- (5,2);
\draw[postaction={decorate}] (7,2) -- (8,2);
\draw[postaction={decorate}] (8,2) -- (9,2);
\end{scope}

\draw [fill=ttqqqq] (1,2) circle (2.5pt);
\draw[color=ttqqqq] (1.,2.42) node {$1$};
\draw [fill=ttqqqq] (2,2) circle (2.5pt);
\draw[color=ttqqqq] (2.,2.42) node {$2$};
\draw [fill=ttqqqq] (3,2) circle (2.5pt);
\draw[color=ttqqqq] (3.,2.42) node {$3$};
\draw [fill=ttqqqq] (4,2) circle (2.5pt);
\draw[color=ttqqqq] (4.,2.42) node {$4$};
\draw [fill=ttqqqq] (5,2) circle (2.5pt);
\draw[color=ttqqqq] (5.,2.42) node {$5$};
\draw [fill=ttqqqq] (7,2) circle (2.5pt);
\draw[color=ttqqqq] (7.,2.42) node {$m-1$};
\draw [fill=ttqqqq] (8,2) circle (2.5pt);
\draw[color=ttqqqq] (8.,2.42) node {$m$};
\draw [fill=ttqqqq] (9,2) circle (2.5pt);
\draw[color=ttqqqq] (9.,2.42) node {$m+1$};
\draw [fill=ttqqqq] (6,2) node {$\dots$};
\draw [fill=ttqqqq] (10,2) node {$\dots$};
\end{tikzpicture}
\end{center}
The translation $\tau$ is an automorphism of the infinite translation quiver $\IZ \vv{A}_\infty$ and so $\tau^r$ is also an automorphism for any $r \in \IZ$. We consider the orbit space of $\IZ \vv{A}_\infty$ under the action of $\tau^r$. 
\begin{definition}
    A component of the Auslander-Reiten quiver which is isomorphic to $\IZ \vv{A}_\infty/(\tau^r)$ for some $r \geq 1$ is called a \emph{stable tube of rank} $r$. A stable tube of rank $r=1$ is called a \emph{homogeneous} tube. \\
The vertices of a rank $r$ tube are of the form $(n,\ell)$ for some $n \in \IZ_r$ and $\ell \geq 1$. If $\ell = 1$, then the vertex is said to be at the \emph{mouth} of the tube.
\end{definition}
We are now in the position to describe $\mathcal{R}$ the regular components.

\begin{proposition}
    The components of $\mathcal{R}$ consist of:
    \begin{enumerate}
        \item a standard stable tube of rank $p$,
        \item a standard stable tube of rank $q$,
        \item a family of standard homogeneous tubes indexed by $\mathbf{k}^*$. 
    \end{enumerate}
    All of the tubes are pairwise orthogonal (so there are no non-zero morphisms between tubes).
\end{proposition}

We can describe the objects in a rank $r$ tube $\mathcal{T}$ explicitly. The full subcategory of $\mod{\mathbf{k}Q}$ with indecomposable objects in $\mathcal{T}$ is hereditary, $\IC$-linear and abelian. There are simple objects along the mouth. Note that these are simple in the tube, but may not be simple in the bigger module category and for this reason they are referred to as quasi-simples. Tubes are uniserial, so every indecomposable object $t = (n,\ell)$ has a unique composition series, which has (composition) length $\ell(t) = \ell$. We denote the socle and top of a module $M$ by $\operatorname{soc}(M)$ and $\operatorname{top}(M)$ respectively.

\begin{lemma} \label{uniserialprops}
    Let $\mathcal{A}$ be a uniserial category and suppose $M,N$ are indecomposable objects in $\mathcal{A}$. Then the following are equivalent:
    \begin{enumerate}
        \item $M \cong N$.
        \item $\operatorname{top}(M) \cong \operatorname{top}(N)$ and $\ell(M) =\ell(N)$.
        \item $\operatorname{soc}(M) \cong \operatorname{soc}(N)$ and $\ell(M) =\ell(N)$.
    \end{enumerate}
\end{lemma}

For convenience, we will sometimes use the language of rays and corays. For a quasi-simple $S$, the ray starting at $S$ consists of all indecomposable objects $M$ such that $\operatorname{soc}(M) \cong S$. By Lemma~\ref{uniserialprops}, there is a unique object $M_\ell$ of a length $\ell$ on the ray up to isomorphism. For any $\ell \geq 1$, there is an indecomposable monomorphism which maps $M_\ell$ to $M_{\ell +1}$ which has cokernel $\operatorname{top}(M_{\ell +1})$ (this is unique up to scaling). Dually, the coray ending at $S$ consists of all indecomposable objects $M$ such that $\operatorname{top}(M) \cong S$ and there is a unique object $M_\ell$ of a length $\ell$ on the ray. For any $\ell \geq 1$, there is an indecomposable epimorphism which maps $M_{\ell +1}$ to $M_\ell$ which has kernel $\operatorname{soc}(M_{\ell +1})$. We say that a morphism between two indecomposable objects on a ray (respectively coray) factors along the ray (respectively coray) if it is a composition of these irreducible monomorphisms (respectively epimorphisms).

\section{Preliminaries on Derived Categories}\label{PrelimDerived}
If $\mathcal{A}$ is an abelian category such as one of the module categories discussed above, then we can construct the bounded derived category $D^b(\mathcal{A})$. This has objects, which are complexes of objects in $\mathcal{A}$, and morphisms given by chain maps up to quasi-isomorphism (see \cite{Keller} for a more complete introduction). The abelian category $\mathcal{A}$ sits naturally as a full subcategory $\mathcal{A}[0] \subseteq D^b(\mathcal{A})$ consisting of complexes of the form 
\[ a[0] =  \dots \to 0 \to 0 \to a \to 0 \to 0 \to \dots\]
that are zero everywhere except in degree zero. This subcategory is often referred to as the  \textit{(standard) heart} of $D^b(\mathcal{A}) $.

In the context of this paper, it may seem that working in the larger category $D^b(\mathcal{A})$ rather than $\mathcal{A}$ adds a potentially unnecessary extra level of complexity. However, in the examples we are studying, computations are not significantly more complicated and in many cases may be simpler. In particular, since the algebras $\mathbf{k}Q$ that we consider are hereditary, all indecomposable objects in $D^b(\mod{\mathbf{k}Q})$ are complexes that are concentrated in one degree. In other words, any indecomposable object in $D^b(\mod{\mathbf{k}Q})$ is in some shift $\mathcal{A}[n]$ of the heart and so of the form $a[n]$, where $a$ is an indecomposable object of $\mod{\mathbf{k}Q}$ and $[n]$ denotes the shift of the complex by $n\in \IZ$. 
For any objects $a,b$ in $\mod{\mathbf{k}Q}$ then  
\[ \Hom_{D^b(\mod{\mathbf{k}Q})}(a,b[n]) = \Ext_{\mod{\mathbf{k}Q}}^n(a,b) \]
and the hereditary property implies that there can only be nonzero morphisms when $n=0$ or $n=1$.

\subsection{Auslander-Reiten theory}
It follows from the previous observation that in our examples, the Auslander-Reiten quiver of the derived category $D^b(\mod{\mathbf{k}Q})$ consists of $\mathbb{Z}$ copies of $\mod{\mathbf{k}Q}$, with some additional morphisms between these which correspond to extensions $\Ext^1$ in the module category. There is a functor $\tau: D^b(\mod{\mathbf{k}Q}) \to D^b(\mod{\mathbf{k}Q})$ which is again called the Auslander-Reiten translate. If $M$ is an indecomposable object in the module category which is not projective, then considering the module category as a subcategory of the derived category as above, the two notions of AR translate coincide. If $M$ is an indecomposable projective object corresponding to a vertex $v$, then $\tau M = I[-1]$ where $I$ is the indecomposable injective module corresponding to the vertex $v$. On the derived category, the Auslander-Reiten translate has nicer properties. For example, it is an autoequivalence with a well-defined inverse functor $\tau^{-1}$. 

In the derived category, instead of Auslander-Reiten sequences, there are distinguished triangles called the Auslander-Reiten triangles for each indecomposable object $M$, 
\begin{equation}
   M \stackrel{\scriptsize \colmat{f_1}{f_2}}{\too} E_1 \oplus E_2 \stackrel{\scriptsize \rowmat{g_1}{g_2}}{\too} \tau^{-1} M \too  M[1] \end{equation}
Again, the fact that the middle term has at most two non-zero indecomposable summands holds because we are considering cases where $Q$ is of type $A$ or type $\tilde{A}$ and depending on the choice of $M$, it is possible that one of these summands is zero. There are again commutative squares that can be seen (up to scaling by elements in $\mathbf{k}^*$) in the Auslander-Reiten quiver.
\begin{figure}[h]
\centering
\definecolor{ududff}{rgb}{0,0,0}
\begin{tikzpicture}[line cap=round,line join=round,>=triangle 45,x=1cm,y=1cm]
\draw [->,line width=1pt] (2.2,2.2) -- (2.8,2.8);
\draw [->,line width=1pt] (3.2,2.8) -- (3.8,2.2);
\draw [->,line width=1pt] (2.2,1.8) -- (2.8,1.2);
\draw [->,line width=1pt] (3.2,1.2) -- (3.8,1.8);
\draw [fill=ududff] (2,2) node {$ M$};
\draw [fill=ududff] (3,3) node {$E_1$};
\draw [fill=ududff] (3,1) node {$E_2$};
\draw [fill=ududff] (4.2,2) node {$\tau^{-1}M$};
\begin{scriptsize}
\draw[color=black] (2.35,2.7) node {$f_1$};
\draw[color=black] (3.6,2.7) node {$g_1$};
\draw[color=black] (2.35,1.3) node {$f_2$};
\draw[color=black] (3.6,1.3) node {$-g_2$};
\end{scriptsize}

\end{tikzpicture}
\end{figure}

\begin{definition} \label{phidefn}
For any indecomposable object $M$, we define the map \[ \Phi_M: M \to \tau^{-1} M\] 
to be the morphism $g_1 \circ f_1$, factoring through the arrows in the AR mesh. This is well defined up to a scaling by an element in $\mathbf{k}^*$.
\end{definition}

\subsection{Type $A$}
If $Q$ is of type $A_n$ with any orientation of the arrows, then Auslander-Reiten quiver of $D^b(\mod{\mathbf{k}Q})$ is isomorphic to the infinite translation quiver $\IZ Q^{op}$. The translation corresponds to the Auslander-Reiten translate. An example with an $A_4$ quiver is shown in Figure~\ref{A4AR}.

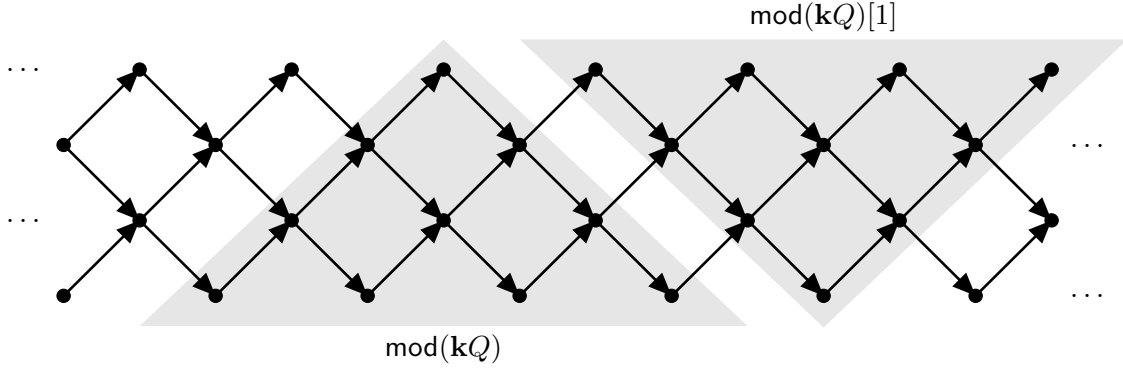
\begin{figure}
\centering
\definecolor{ttqqqq}{rgb}{0,0,0}
\begin{tikzpicture}[line cap=round,line join=round,>=triangle 45,x=1cm,y=1cm]
\fill[line width=1pt,color=ttqqqq,fill=ttqqqq,fill opacity=0.10000000149011612] (-1.,1.4) -- (-5.0,-2.4) -- (3,-2.4) -- cycle;
\fill[line width=1pt,color=ttqqqq,fill=ttqqqq,fill opacity=0.10000000149011612] (0,1.4) -- (4,-2.42) -- (8,1.4) -- cycle;
\draw [->,line width=1pt,color=ttqqqq] (-6,0) -- (-5,-1);
\draw [->,line width=1pt,color=ttqqqq] (-5,-1) -- (-4,-2);
\draw [->,line width=1pt,color=ttqqqq] (-6,-2) -- (-5,-1);
\draw [->,line width=1pt,color=ttqqqq] (-5,-1) -- (-4,0);
\draw [->,line width=1pt,color=ttqqqq] (-6,0) -- (-5,1);
\draw [->,line width=1pt,color=ttqqqq] (-5,1) -- (-4,0);
\draw [->,line width=1pt,color=ttqqqq] (-4,0) -- (-3,-1);
\draw [->,line width=1pt,color=ttqqqq] (-4,-2) -- (-3,-1);
\draw [->,line width=1pt,color=ttqqqq] (-4,0) -- (-3,1);
\draw [->,line width=1pt,color=ttqqqq] (-3,-1) -- (-2,0);
\draw [->,line width=1pt,color=ttqqqq] (-3,1) -- (-2,0);
\draw [->,line width=1pt,color=ttqqqq] (-3,-1) -- (-2,-2);
\draw [->,line width=1pt,color=ttqqqq] (-2,-2) -- (-1,-1);
\draw [->,line width=1pt,color=ttqqqq] (-1,-1) -- (0,0);
\draw [->,line width=1pt,color=ttqqqq] (0,0) -- (1,1);
\draw [->,line width=1pt,color=ttqqqq] (-2,0) -- (-1,1);
\draw [->,line width=1pt,color=ttqqqq] (-2,0) -- (-1,-1);
\draw [->,line width=1pt,color=ttqqqq] (-1,-1) -- (0,-2);
\draw [->,line width=1pt,color=ttqqqq] (0,-2) -- (1,-1);
\draw [->,line width=1pt,color=ttqqqq] (1,-1) -- (2,0);
\draw [->,line width=1pt,color=ttqqqq] (2,0) -- (3,1);
\draw [->,line width=1pt,color=ttqqqq] (-1,1) -- (0,0);
\draw [->,line width=1pt,color=ttqqqq] (0,0) -- (1,-1);
\draw [->,line width=1pt,color=ttqqqq] (1,-1) -- (2,-2);
\draw [->,line width=1pt,color=ttqqqq] (1,1) -- (2,0);
\draw [->,line width=1pt,color=ttqqqq] (2,-2) -- (3,-1);
\draw [->,line width=1pt,color=ttqqqq] (3,-1) -- (4,0);
\draw [->,line width=1pt,color=ttqqqq] (4,0) -- (5,1);
\draw [->,line width=1pt,color=ttqqqq] (2,0) -- (3,-1);
\draw [->,line width=1pt,color=ttqqqq] (3,-1) -- (4,-2);
\draw [->,line width=1pt,color=ttqqqq] (4,-2) -- (5,-1);
\draw [->,line width=1pt,color=ttqqqq] (5,-1) -- (6,0);
\draw [->,line width=1pt,color=ttqqqq] (6,0) -- (7,1);
\draw [->,line width=1pt,color=ttqqqq] (3,1) -- (4,0);
\draw [->,line width=1pt,color=ttqqqq] (4,0) -- (5,-1);
\draw [->,line width=1pt,color=ttqqqq] (5,-1) -- (6,-2);
\draw [->,line width=1pt,color=ttqqqq] (6,-2) -- (7,-1);
\draw [->,line width=1pt,color=ttqqqq] (5,1) -- (6,0);
\draw [->,line width=1pt,color=ttqqqq] (6,0) -- (7,-1);
\draw [color=ttqqqq](7.5,0.0) node {$\dots$};
\draw [color=ttqqqq](7.5,-2) node {$\dots$};
\draw [color=ttqqqq](-6.5,1) node {$\dots$};
\draw [color=ttqqqq](-6.5,-1) node {$\dots$};
\draw [fill=ttqqqq] (-5,1) circle (2.5pt);
\draw [fill=ttqqqq] (-3,1) circle (2.5pt);
\draw [fill=ttqqqq] (-1,1) circle (2.5pt);
\draw [fill=ttqqqq] (1,1) circle (2.5pt);
\draw [fill=ttqqqq] (3,1) circle (2.5pt);
\draw [fill=ttqqqq] (5,1) circle (2.5pt);
\draw [fill=ttqqqq] (7,1) circle (2.5pt);
\draw [fill=ttqqqq] (-4,0) circle (2.5pt);
\draw [fill=ttqqqq] (-2,0) circle (2.5pt);
\draw [fill=ttqqqq] (0,0) circle (2.5pt);
\draw [fill=ttqqqq] (2,0) circle (2.5pt);
\draw [fill=ttqqqq] (4,0) circle (2.5pt);
\draw [fill=ttqqqq] (6,0) circle (2.5pt);
\draw [fill=ttqqqq] (-5,-1) circle (2.5pt);
\draw [fill=ttqqqq] (-3,-1) circle (2.5pt);
\draw [fill=ttqqqq] (-1,-1) circle (2.5pt);
\draw [fill=ttqqqq] (1,-1) circle (2.5pt);
\draw [fill=ttqqqq] (3,-1) circle (2.5pt);
\draw [fill=ttqqqq] (5,-1) circle (2.5pt);
\draw [fill=ttqqqq] (7,-1) circle (2.5pt);
\draw [fill=ttqqqq] (-4,-2) circle (2.5pt);
\draw [fill=ttqqqq] (-2,-2) circle (2.5pt);
\draw [fill=ttqqqq] (0,-2) circle (2.5pt);
\draw [fill=ttqqqq] (2,-2) circle (2.5pt);
\draw [fill=ttqqqq] (-6,-2) circle (2.5pt);
\draw [fill=ttqqqq] (-6,0) circle (2.5pt);
\draw [fill=ttqqqq] (4,-2) circle (2.5pt);
\draw [fill=ttqqqq] (6,-2) circle (2.5pt);
\draw[color=ttqqqq] (-1,-2.7) node {$\mod{\mathbf{k}Q}$};
\draw[color=ttqqqq] (4.,1.7) node {$\mod{\mathbf{k}Q}[1]$};
\end{tikzpicture}
\caption{Auslander-Reiten quiver of $D^b(\mod{\mathbf{k}Q})$, where $Q$ is the equioriented quiver of type $A_4$. The shaded regions indicate the module category $\mod{\mathbf{k}Q}$ and its shift. } \label{A4AR}
\end{figure}

\subsection{Type $\tilde{A}$}
If $Q$ is a quiver of type $ \tilde{A_n}$, then the Auslander-Reiten quiver of the derived category splits into two types of components. As in the $A_n$ case, there are extensions in the module category between injective and projective objects and so, for each integer $n$, $\mathcal{P}[n]$ and $\mathcal{I}[n-1]$ fit together to form one component. This is isomorphic to the infinite translation quiver $\IZ Q^{op}$. There are also regular components $\mathcal{R}[n]$ for each $n \in \IZ$.

\begin{center}
\begin{tikzpicture}[line cap=round,line join=round,>=triangle 45,x=0.5cm,y=0.5cm]
\draw [line width=1pt] (-12,4)-- (-7,4);
\draw [line width=1pt] (-12,2)-- (-7,2);
\draw [line width=1pt] (-9.8,4)-- (-10,3.8);
\draw [line width=1pt] (-10,3.8)-- (-9.8,3.6);
\draw [line width=1pt] (-9.8,3.6)-- (-10,3.4);
\draw [line width=1pt] (-10,3.4)-- (-9.8,3.2);
\draw [line width=1pt] (-9.8,3.2)-- (-10,3);
\draw [line width=1pt] (-10,3)-- (-9.8,2.8);
\draw [line width=1pt] (-9.8,2.8)-- (-10,2.6);
\draw [line width=1pt] (-10,2.6)-- (-9.8,2.4);
\draw [line width=1pt] (-9.8,2.4)-- (-10,2.2);
\draw [line width=1pt] (-10,2.2)-- (-9.8,2);
\draw [rotate around={90:(-5.5,3)},line width=1pt] (-5.5,3) ellipse (0.5*1.2807764064044145cm and 0.5*0.80024259022012cm);
\draw [line width=1pt] (-13,4)-- (-12,4);
\draw [line width=1pt] (-13,2)-- (-12,2);
\draw [line width=1pt] (-4,4)-- (2,4);
\draw [line width=1pt] (-4,2)-- (2,2);
\draw [rotate around={90:(3.5,3)},line width=1pt] (3.5,3) ellipse (0.5*1.2807764064044145cm and 0.5*0.80024259022012cm);
\draw [line width=1pt] (5,4)-- (11,4);
\draw [line width=1pt] (5,2)-- (11,2);
\draw [line width=1pt] (-0.8,4)-- (-1,3.8);
\draw [line width=1pt] (-1,3.8)-- (-0.8,3.6);
\draw [line width=1pt] (-0.8,3.6)-- (-1,3.4);
\draw [line width=1pt] (-1,3.4)-- (-0.8,3.2);
\draw [line width=1pt] (-0.8,3.2)-- (-1,3);
\draw [line width=1pt] (-1,3)-- (-0.8,2.8);
\draw [line width=1pt] (-0.8,2.8)-- (-1,2.6);
\draw [line width=1pt] (-1,2.6)-- (-0.8,2.4);
\draw [line width=1pt] (-0.8,2.4)-- (-1,2.2);
\draw [line width=1pt] (-1,2.2)-- (-0.8,2);
\draw [line width=1pt] (8.2,4)-- (8,3.8);
\draw [line width=1pt] (8,3.8)-- (8.2,3.6);
\draw [line width=1pt] (8.2,3.6)-- (8,3.4);
\draw [line width=1pt] (8,3.4)-- (8.2,3.2);
\draw [line width=1pt] (8.2,3.2)-- (8,3);
\draw [line width=1pt] (8,3)-- (8.2,2.8);
\draw [line width=1pt] (8.2,2.8)-- (8,2.6);
\draw [line width=1pt] (8,2.6)-- (8.2,2.4);
\draw [line width=1pt] (8.2,2.4)-- (8,2.2);
\draw [line width=1pt] (8,2.2)-- (8.2,2);
\draw [rotate around={90:(12.5,3)},line width=1pt] (12.5,3) ellipse (0.5*1.2807764064044103cm and 0.5*0.8002425902201173cm);
\begin{scriptsize}
\draw[color=black] (-2.5,3) node {$\mathcal{I}[0]$};
\draw[color=black] (0.5,3) node {$\mathcal{P}[1]$};
\draw[color=black] (3.5,3) node {$\mathcal{R}[1]$};
\draw[color=black] (-5.5,3) node {$\mathcal{R}[0]$};
\draw[color=black] (-8.5,3) node {$\mathcal{P}[0]$};
\draw[color=black] (-11.5,3) node {$\mathcal{I}[-1]$};
\draw[color=black] (6.5,3) node {$\mathcal{I}[1]$};
\draw[color=black] (9.5,3) node {$\mathcal{P}[2]$};
\draw[color=black] (12.5,3) node {$\mathcal{R}[2]$};
\draw[color=black] (14.5,3) node {$\cdots$};
\draw[color=black] (-14,3) node {$\cdots$};
\end{scriptsize}
\end{tikzpicture}
\end{center}

 \section{Interleaving distance}\label{Interleaving}
In this section we introduce a definition of interleaving distance on persistence modules. We show how this restricts to the usual notion when considering equioriented quivers of type~$A$.
First we recall this usual definition. Consider again the \emph{equioriented} quiver of type $A_n$, oriented such that the arrows go from $x$ to $x+1$ for each $1 \leq x \leq n-1$.

\begin{definition}
    Let $\delta \in \IZ$. The $\delta$-shift of a representation $(M,\phi)$ is defined to be the representation $(M(\delta),\phi(\delta))$, where
\[ M(\delta)_x = \begin{cases} M_{x+\delta}, & 1 \leq x+\delta \leq n\\
0, & \text{otherwise}
    \end{cases}
\qquad \text{  and, }
\]
\[\phi(\delta)_{x,x+1} =  \begin{cases} \phi_{x+\delta, x+1+\delta}, & 1 \leq x+\delta\leq x+1+\delta \leq n\\
0, & \text{otherwise.}
    \end{cases}
\]
\end{definition}
\begin{remark}
    The $\delta$-shift operation commutes with taking (finite) direct sums.  
\end{remark} 

\begin{lemma} \label{deltaistau}
    Let $Q$ be an equioriented quiver of type $A_n$ and let $M$ be an indecomposable representation. If $M$ is not injective then \[M(1) = \tau^{-1} M.\]
\end{lemma}
\begin{proof}
    By Theorem~\ref{Gabriel}, an indecomposable representation $M$ is of the form $V_{[a,b)}$ for some $1 \leq a < b \leq n+1$. Such a representation is injective when $a=1$, so we assume $a>1$. Under this assumption, it follows from the definition of the $\delta$-shift that $V_{[a,b)}(1) = V_{[a-1,b-1)}$. On the other hand, since $V_{[a,b)}$ is not injective, $\tau^{-1}V_{[a,b]}$ is nonzero and it can be read from the Auslander-Reiten quiver that $\tau^{-1}V_{[a,b)} = V_{[a-1,b-1)}$. 
\end{proof}
The lemma is written as a statement about objects in $\rep_{\mathbf{k}}(Q)$. We note however, that the inverse Auslander-Reiten translate $\tau^{-1}M$ in $\rep_{\mathbf{k}}(Q)$ and $D^b(\rep_{\mathbf{k}}(Q))$ coincide when $M$ is not an injective object in $\rep_{\mathbf{k}}(Q)$. We can therefore interpret this as a result about indecomposable objects in the standard heart $(\rep_{\mathbf{k}}(Q))[0] \subseteq D^b(\rep_{\mathbf{k}}(Q))$. In this case, the condition that $M$ is not injective is equivalent to the statement that $\tau^{-1}M \in D^b(\rep_{\mathbf{k}}(Q))$ is an object in the standard heart.
\begin{remark} \label{limitingtoZ}
  It is not the case that $M(\delta) = \tau^{-\delta}  M$ for all indecomposable objects $M$ and $\delta \in \IZ_{\geq 0}$. They differ when the support of $M$ is shifted in such a way that it would no longer be in the interval $[1,n+1)$. Using Proposition~\ref{tauProperties} and Lemma~\ref{deltaistau} it can be seen that $M(\delta) = \tau^{-\delta}  M$ precisely when $\tau^{-\delta} M$ is in the standard heart $(\rep_{\mathbf{k}}(A_n))[0] \subseteq D^b(\rep_{\mathbf{k}}(A_n))$, i.e. it is a complex concentrated in degree zero. For any $ k>0 $ we can consider $\rep_{\mathbf{k}}(A_n)$ as the subcategory of $\rep_{\mathbf{k}}(A_{n+2k})$ consisting of representations supported on the sub interval $[1+k, n+k+1) \subset [1,n+2k+1)$. For a fixed $M$ and $\delta$ it is always possible to find $k$ sufficiently big such that $M(\delta) = \tau^{-\delta}  M$ when considered as objects in $\rep_{\mathbf{k}}(A_{n+2k})$. In other words, the two operations do agree in an appropriately defined limit category $\rep_{\mathbf{k}}(A_\infty)$. 
\end{remark}

We will shortly use $\tau^{-\delta}$ to replace the delta shift in the definition of a $\delta$-interleaving. First, however, we need to define $\delta$-transition morphisms. Again we start by recalling the usual definition. Let $\delta$ be a non-negative integer and consider a representation $(M,\phi)$. 

\begin{definition}
The $\delta$-transition morphism $\phi_M^\delta \colon  M \to M(\delta)$ is the morphism given by the linear maps \[ \phi_{x,x+\delta} :=  \phi_{x+\delta-1,x+\delta}  \circ \dots \circ  \phi_{x+1,x+2} \circ \phi_{x,x+1} \colon  M_x \too M_{x+\delta} = M(\delta)_x,\]
for each $x \in [1,n+1)$.
\end{definition}
\noindent It follows from this definition that $\phi_M^\delta$ factorises as $\phi_M^\delta = \phi_{M(\delta-1)}^1 \circ \dots \circ \phi_{M(1)}^1 \circ \phi_M^1$.

    Furthermore, if $M= \bigoplus_{i=1}^k M_i$ is decomposable, then $\phi_M^\delta$ is diagonal and restricts to the morphism $\phi_{M_i}^\delta \colon M_i \to M_i(\delta)$ on each summand. Therefore for arbitrary $M$ and $\delta \geq 0$, we see that $\phi_M^\delta$ is determined by knowing $\phi_N^1$ for each indecomposable object $N$.

\begin{lemma} \label{transisionsaresame}
Let $Q$ be an equioriented quiver of type $A_n$ and let $N$ be an indecomposable representation, considered as an object in the standard heart $(\rep_{\mathbf{k}}(Q))[0] \subseteq D^b(\rep_{\mathbf{k}}(Q))$. If $N$ is not injective, so $\tau^{-1}N \in (\rep_{\mathbf{k}}(Q))[0]$ then \[\phi_N^1 = \Phi_N\]
where $\Phi_N$ is the map defined in Definition~\ref{phidefn}.
    \end{lemma}
\begin{proof}
    This follows from Remark~\ref{AtypeARmesh}, noting that $\Phi_N$ is the composition of two morphisms in the mesh starting at $N$.
\end{proof}
Motivated by this lemma and the preceding paragraph, we make the following definition.

\begin{definition}
 Let $\delta \in \IZ_{\geq 0}$ and let $N$ be an indecomposable representation, considered as an object in the standard heart $(\rep_{\mathbf{k}}(Q))[0] \subseteq D^b(\rep_{\mathbf{k}}(Q))$. If $\delta=0$, we define $\Phi^0_N$ to be the identity morphism $\operatorname{id}_N: N \to N$. Otherwise we define $\Phi^\delta_N: N \to \tau^{-\delta} N$ to be the composition \[ \Phi^\delta_N = \Phi_{\tau^{-(\delta-1)} N}\circ \dots \circ \Phi_{\tau^{-1} N} \circ \Phi_N.\] 
We can extend this to decomposable objects $N$ such that the restriction of $\Phi^\delta_N$ to any indecomposable summand $N'$ is given by $\Phi^\delta_{N'}$.
\end{definition}


We can then write down the following immediate corollary of Lemma~\ref{transisionsaresame}, showing that in the equioriented $A_n$ case, this generalises the $\delta$-transition morphism, when
resrticted to a certain subset of objects.
\begin{corollary} \label{transsame}
    Let $Q$ be an equioriented quiver of type $A_n$ and let $N$ be a representation, considered as an object in the standard heart $(\rep_{\mathbf{k}}(Q))[0] \subseteq D^b(\rep_{\mathbf{k}}(Q))$. If $\tau^{-k} N$ 
    is also in the standard heart, then \[\phi_N^k = \Phi_N^k.\]
\end{corollary}

We are finally in the position to write down the definition of $\delta$-interleaving that we will use in our more general setting. 

\begin{definition} \label{deltainterleaved} Let $\delta \in \IZ_{\geq 0}$. Two persistence modules $N, M$ are said to be $\delta$-interleaved if there exist morphisms $f: M \to \tau^{-\delta} N$ and $g: N \to \tau^{-\delta} M$, such that  $\tau^{-\delta}g\circ f = \Phi^{2\delta}_M $ and   $\tau^{-\delta}f\circ g = \Phi^{2\delta}_N $. \end{definition}

This generalises the usual definition, in the following sense. Let $Q$ be an equioriented quiver of type $A_n$. As in Remark~\ref{limitingtoZ}, for any $ k\geq 0 $ we can consider $\rep_{\mathbf{k}}(A_n)$ as the subcategory of $\rep_{\mathbf{k}}(A_{n+2k})$ consisting of representations supported on the subinterval $[1+k, n+k+1) \subset [1,n+2k+1)$.

\begin{proposition}\label{InterleavingCorresp}
    Let $M,N$ be persistence modules in $\rep_{\mathbf{k}}(A_n)$. There exists $k_0 \geq 0$ such that for any $k \geq k_0$, the persistence modules $N,M$ considered as objects in the standard heart of $D^b(\rep_{\mathbf{k}}(A_{n+2k}))$ are $\delta$-interleaved in the sense of Definition~\ref{deltainterleaved} if and only if they satisfy the classical definition of $\delta$-interleaved persistence modules.
\end{proposition}
\begin{proof}
     Choose $k_0$ sufficiently big such that for any $k\geq k_0$, $\tau^{-2\delta}  M$ and $\tau^{-2\delta}  N$ are both in the standard heart of $D^b(\rep_{\mathbf{k}}(A_{n+2k}))$. It follows from Lemma~\ref{deltaistau} and Corollary~\ref{transsame} that under these conditions, the $\delta$-shift and the $2\delta$-transition functions $\phi_M^{2\delta}$ and $\phi_N^{2\delta}$ coincide with $\tau^{-\delta}$, $\Phi_M^{2\delta}$ and $\Phi_N^{2\delta}$ respectively. Making these substitutions, we move between the usual definition of $\delta$-interleaved persistence modules and Definition~\ref{deltainterleaved}. 
\end{proof}

\begin{remark}   
The usual definition of $\delta$-interleaved persistence modules is only well-defined when the underlying quiver is equioriented of type $A$. However, Definition~\ref{deltainterleaved} can be applied more widely including to representations of quivers of type $A$ which are not equioriented, for example the zig-zag quiver, and persistence modules of type $\tilde{A}$. In fact, it would make sense to consider this definition for any hereditary gentle algebra.\end{remark} 

We use this definition to extend the notion of interleaving distance.

\begin{definition}
    We define a map $d_I \colon \operatorname{obj}(\mod{\mathbf{k}Q}) \times \operatorname{obj}(\mod{\mathbf{k}Q}) \to \IZ_{\geq 0} \cup \{\infty\} $ called the interleaving distance, by setting
\[ d_I(M,N) = 
   \inf\{ \delta \in \IZ_{\geq 0} \mid M,N \text{ are $\delta$-interleaved\}}
    \]
    where we use the convention that $\inf{\varnothing} = \infty$.
    \end{definition}

\subsection{Interleaving distance in the $\tilde{A}_n$ case} We now consider some properties of this interleaving distance in the $\tilde{A}_n$ case. First we observe that the $\Hom$-orthogonality properties of the category mean that we can look at objects with summands in different components of the Auslander-Reiten quiver separately. Let $M,N$ be any two persistence modules of type $\tilde{A}$. There are decompositions \[M= M_P \oplus M_R \oplus M_I, \quad N= N_P \oplus N_R \oplus N_I\] where the subscripts $P,R$ and $I$ denote summands in the preprojective, regular and preinjective components respectively. Some of these summands may be zero. Such decompositions are well defined by the Krull-Schmidt property.
\begin{lemma} \label{componentwise}
Persistence modules $M$ and $N$ are $\delta$-interleaved if and only if the pairs of summands $M_P$ and $N_P$,  $M_R$ and $N_R$, and $M_I$ and $N_I$ are all $\delta$-interleaved.
\end{lemma}
\begin{proof}
Suppose $M$ and $N$ are $\delta$-interleaved with respect to morphisms $f: M \to \tau^{-\delta} N$ and $g: N \to \tau^{-\delta} M$. Using the decomposition, we can write these as matrices 

\[ f = \begin{pmatrix}
f_{PP} & 0 & 0\\
f_{PR} & f_{RR} & 0\\
f_{PI} & f_{RI} & f_{II}
\end{pmatrix},   \qquad  g = \begin{pmatrix}
g_{PP} & 0 & 0\\
g_{PR} &g_{RR} & 0\\
g_{PI} & g_{RI} & g_{II}
\end{pmatrix}  \]
where, for example, $f_{PI} : M_P \to  \tau^{-\delta} N_I $ denotes the composition of $f$ with the inclusion of  $M_P \hookrightarrow M$ and the projection $ \tau^{-\delta}N \twoheadrightarrow  \tau^{-\delta}N_I$. We observe that the matrix is lower triangular, since there are no nonzero morphisms from an object in the regular component to any object in the preprojective component and from an object in the preinjective component to any other component.

By definition, $\Phi^{2\delta}_M $ restricts to $\Phi^{2\delta}_{M'} $ on any summand $M'$ of $M$, so
\[ \tau^{-\delta}g\circ f = \begin{pmatrix}
\tau^{-\delta}g_{PP}\circ f_{PP} & 0 & 0\\
* & \tau^{-\delta}g_{RR}\circ f_{RR} & 0\\
* &* & \tau^{-\delta}g_{II}\circ f_{II}
\end{pmatrix} = 
\begin{pmatrix}
 \Phi^{2\delta}_{M_P} & 0 & 0\\
0 &  \Phi^{2\delta}_{M_R} & 0\\
0 &0 & \Phi^{2\delta}_{M_{I}}
\end{pmatrix}
 \]
and similarly for $ \tau^{-\delta}f\circ g $.
Comparing diagonal entries, we conclude that $M_P$ and $N_P$ are $\delta$-interleaved with respect to the morphisms $f_{PP}$ and $g_{PP}$ and similarly,  $M_R$ and $N_R$, and $M_I$ and $N_I$ are also $\delta$-interleaved.

Conversely, suppose that $M_P$ and $N_P$, $M_R$ and $N_R$, and $M_I$ and $N_I$ are $\delta$-interleaved with respect to morphisms $f_P$ and $g_P$, $f_R$ and $g_R$, and $f_I$ and $g_I$ respectively. Then $M$ and $N$ are $\delta$-interleaved with respect to morphisms:
\[ f = \begin{pmatrix}
f_{P} & 0 & 0\\
0 & f_{R} & 0\\
0 &0 & f_{I}
\end{pmatrix},   \qquad  g = \begin{pmatrix}
g_{P} & 0 & 0\\
0 &g_{R} & 0\\
0 &0 & g_{I}
\end{pmatrix}.  \]

\end{proof}

\begin{corollary}
    Let $M= M_P \oplus M_R \oplus M_I$ and $N= N_P \oplus N_R \oplus N_I$ be as above,
    then \[ d_I(M,N) = \max\{ d_I(M_P,N_P),d_I(M_R,N_R),d_I(M_I,N_I) \}.\]
\end{corollary}

Using this result, we can restrict ourselves to considering persistence modules whose summands all lie in the same component when calculating interleaving distance.

We now prove a technical lemma, before concluding this section by showing that the interleaving distance, as we have defined it, is a metric.

\begin{lemma} \label{commrel}
    Let $h:M \to N $ be a morphism between two persistence modules in the same component of the Auslander-Reiten quiver, and let $k \in \IZ_{\geq 0}$. Then \[ \tau^{-k} h \circ \Phi_{M}^k = \Phi_{N}^k \circ h.\]
\end{lemma}
\begin{proof}
    First we assume that $M$ and $N$ are indecomposable. If $k=0$ the result holds trivially, so we initially consider the case where $k=1$. Using the stability of the component, the morphism $h$ is a linear combination of basis morphisms and each of these is a finite composition of irreducible morphisms in the component. If $h_0:M \to N $ is an irreducible morphism then there are two possibilities:
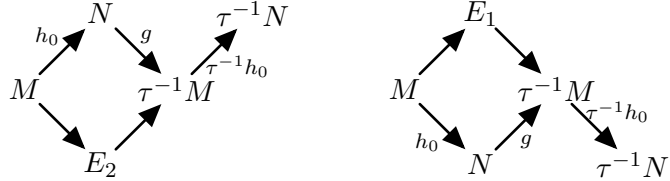
\begin{figure}[h]
\centering
\definecolor{ududff}{rgb}{0,0,0}
\begin{tikzpicture}[line cap=round,line join=round,>=triangle 45,x=1cm,y=1cm]
\draw [->,line width=1pt] (2.2,2.2) -- (2.8,2.8);
\draw [->,line width=1pt] (3.2,2.8) -- (3.8,2.2);
\draw [->,line width=1pt] (2.2,1.8) -- (2.8,1.2);
\draw [->,line width=1pt] (3.2,1.2) -- (3.8,1.8);
\draw [->,line width=1pt] (4.2,2.2) -- (4.8,2.8);
\draw [fill=ududff] (2,2) node {$ M$};
\draw [fill=ududff] (3,3) node {$N$};
\draw [fill=ududff] (3,1) node {$E_2$};
\draw [fill=ududff] (4,2) node {$\tau^{-1}M$};
\draw [fill=ududff] (5,3) node {$\tau^{-1}N$};
\begin{scriptsize}
\draw[color=black] (2.3,2.7) node {$h_0$};
\draw[color=black] (4.8,2.3) node {$\tau^{-1}h_0$};
\draw[color=black] (3.6,2.7) node {$g$};
\end{scriptsize}

\draw [->,line width=1pt] (7.2,2.2) -- (7.8,2.8);
\draw [->,line width=1pt] (8.2,2.8) -- (8.8,2.2);
\draw [->,line width=1pt] (7.2,1.8) -- (7.8,1.2);
\draw [->,line width=1pt] (8.2,1.2) -- (8.8,1.8);
\draw [->,line width=1pt] (9.2,1.8) -- (9.8,1.2);
\draw [fill=ududff] (7,2) node {$M$};
\draw [fill=ududff] (8,3) node {$E_1$};
\draw [fill=ududff] (8,1) node {$N$};
\draw [fill=ududff] (9,2) node {$\tau^{-1}M$};
\draw [fill=ududff] (10,1) node {$\tau^{-1}N$};
\begin{scriptsize}
\draw[color=black] (7.3,1.3) node {$h_0$};
\draw[color=black] (8.6,1.3) node {$g$};
\draw[color=black] (9.8,1.7) node {$\tau^{-1}h_0$};
\end{scriptsize}

\end{tikzpicture}
\end{figure}

 \noindent Using the mesh relations, it is clear that in both cases $\tau^{-1}h_0 \circ \Phi_{M} = \tau^{-1}h_0 \circ g \circ h_0 = \Phi_{N} \circ h_0$. Working inductively, we suppose that $\tau^{-1}h \circ \Phi_{E} = \Phi_{N} \circ h$ for any $h:E \to N$ which is the composition of $n$ irreducible morphisms and that $h' = h\circ h_0$ is a composition of $n+1$ irreducible morphisms where $h_0: M \to E$ is irreducible. Using the induction hypothesis and the functoriality of $\tau$ we see that
 \[ \tau^{-1}h' \circ \Phi_{ M} = \tau^{-1}h\circ \tau^{-1}h_0 \circ \Phi_{M}  = \tau^{-1}h\circ \Phi_{ E} \circ  h_0 =  \Phi_{ N} \circ  h\circ h_0 =\Phi_{N} \circ h'. \]
 In the case when $k=1$ and $M$ and $N$ are indecomposable, the result then follows by induction, and linearity. For any $k \geq 0$ we recall that $\Phi^k_M = \Phi_{\tau^{-k+1}M}\circ \dots \circ \Phi_{ \tau^{-1}M} \circ \Phi_M$. Therefore,
\begin{align*}
\tau^{-k}h \circ \Phi_{ M}^k &= \tau^{-k}h\circ \Phi_{\tau^{-k+1}M} \circ \Phi_{\tau^{-k+2}M} \circ \dots \circ \Phi_{ \tau^{-1}M} \circ \Phi_M\\ 
&=  \Phi_{\tau^{-k+1}N}\circ \tau^{-k+1}h\circ \Phi_{\tau^{-k+2}M} \circ \dots \circ \Phi_{ \tau^{-1}M} \circ \Phi_M \\&= \dots \\
&= \Phi_{\tau^{-k+1}N} \circ \Phi_{\tau^{-k+2}N} \circ \dots \circ \Phi_{ \tau^{-1}N} \circ \Phi_N \circ h  \\
&= \Phi_N^k \circ h
\end{align*} 
Finally, suppose that $M = \bigoplus_{i \in I}M_i$ and $N = \bigoplus_{j \in J}N_j$ are written as a sum of indecomposable objects. Using this decomposition, we can write the map $h:M \to N$ as a matrix with entries $h_{ji}: M_i \to N_j$. A quick calculation, applying the statement for indecomposable objects entry-wise, yields the desired general result.
\end{proof}

\begin{proposition}\label{Interleavingmetric}
    The interleaving distance is a (possibly infinite) metric on $\operatorname{obj}(\mod{\mathbf{k}Q}) $.
\end{proposition}
\begin{proof}
    Let $M$ be any persistence module. It is a consequence of the definition that $M$ is $0$-interleaved with itself, so $d_I(M,M) = 0$. Conversely, if $d_I(M,N) = 0$, then the definition of $0$-interleaved implies that $M$ and $N$ are isomorphic. Therefore, if $M \ncong N$ then $d_I(M,N) > 0$. The definition is symmetric, so it remains to show that the triangle inequality holds. 

    Let $\delta_1$, $\delta_2$ be finite and suppose $M,N$ are $\delta_1$-interleaved with respect to the morphisms \[ f_1: M \to \tau^{-\delta_1}N \text{ and } g_1: N \to \tau^{-\delta_1}M \] and $N,P$ are $\delta_2$-interleaved with respect to \[ f_2: N \to \tau^{-\delta_2}P  \text{  and } g_2: P \to \tau^{-\delta_2}N.\] 
    We show that $M,P$ are $(\delta_1+\delta_2)$-interleaved. By Lemma~\ref{componentwise}
    , we may assume without loss of generality that $M,N, P$ lie in the same component of the Auslander-Reiten quiver, and the morphisms $f_1, f_2, g_1, g_2$ all factor through objects in the same component. Consider the compositions \[ f= \tau^{-\delta_1}f_2 \circ f_1: M \to \tau^{-(\delta_1+\delta_2)}P \; \text{ and } \; g= \tau^{-\delta_2}g_1 \circ g_2: P \to \tau^{-(\delta_1+\delta_2)}M.\] 
    Using the functoriality of $\tau$, and Lemma~\ref{commrel}
\begin{align*}
    \tau^{-(\delta_1+\delta_2)}g \circ f = \tau^{-(\delta_1+\delta_2)}(\tau^{-\delta_2}g_1 \circ g_2)\circ \tau^{-\delta_1}f_2 \circ f_1 &= \tau^{-(\delta_1+2\delta_2)}g_1 \circ \tau^{-(\delta_1+\delta_2)}g_2\circ \tau^{-\delta_1}f_2 \circ f_1 \\
    &= \tau^{-(\delta_1+2\delta_2)}g_1 \circ \Phi_{\tau^{-\delta_1}N}^{2\delta_2} \circ f_1 \\
     &= \tau^{-(\delta_1+2\delta_2)}g_1 \circ \tau^{-2\delta_2}f_1 \circ  \Phi_{M}^{2\delta_2} \\
     &= (\Phi_{\tau^{-2\delta_2}M}^{2\delta_1}) \circ  (\Phi_{M}^{2\delta_2}) = \Phi_{M}^{2(\delta_1 +\delta_2).}
\end{align*}
    A symmetric argument shows that $\tau^{-(\delta_1+\delta_2)}f \circ g  = \Phi_{P}^{2(\delta_1 +\delta_2)}$ and so $M,P$ are $\delta_1+\delta_2$-interleaved. It follows that if $d_I(M,N) = \delta_1$ and $d_I(N,P) = \delta_2$, then
    \[ d_I(M,P) \leq \delta_1 + \delta_2 = d_I(M,N) + d_I(N,P). \]
\end{proof}

\section{Geometric Models}\label{GeometricM}
The algebras ${\mathbf{k}Q}$ of type $A$ or type $\tilde{A}$ are examples of gentle algebras. This is a setting where there is a very good understanding of the bounded derived category $D^b(\mod{\mathbf{k}Q})$ in terms of combinatorial objects called homotopy strings and bands  \cite{ALP, BM, CPS}.  More recently, work by \cite{OPS, HKK, LP} associates to such an algebra a ``geometric model''  which can be used to describe indecomposable objects in $D^b(\mod{\mathbf{k}Q})$ in terms of (graded) arcs and certain closed curves on a surface. Morphisms can also be described in terms of (graded) intersections of these arcs and curves. See also \cite{K} for a related geometric model for $\mod{\mathbf{k}Q}$. We refer the reader to these references for the general case, and restrict ourselves to describing the geometric models in the cases of the path algebras of type $A$ or type $\tilde{A}$.

The data of a geometric model consists of an oriented surface $\Sigma$ with boundary $\delta(\Sigma)$, a finite set $M_A$ of marked points on the boundary, together with a lamination $L_A$. The lamination $L_A$ is a finite collection of non-selfintersecting and pairwise nonintersecting curves on $\Sigma$, considered up to isotopy relative to $M_A$.  If the algebra is of type $A$ or type $\tilde{A}$, then there is a geometric model for  $D^b(\mod{\mathbf{k}Q})$, such that the curves in the lamination satisfy the following properties:
\begin{enumerate} 
\item \label{endbdry} Both endpoints of each curve are in $\delta(\Sigma) \setminus M_A$.
\item No curve in $L_A$ is isotopic to a part of the boundary  $\delta(\Sigma)$ containing no marked points.
\item \label{polygons} The set of curves in $L_A$ divides the disc into polygons, each of which has precisely one marked point in its boundary.
\item \label{hereditary} Any section of the boundary $\delta(\Sigma)$ between two marked points is the end point of at most two curves in $L_A$.
\end{enumerate} 

\begin{remark}
It is shown in \cite[Section 1.1]{OPS} that, given a gentle algebra, $\Sigma$ can be constructed by gluing together polygons (each of which has a single marked point in its boundary). The gluing is controlled by a marked ribbon graph associated to the algebra, and the glued edges become the curves in a lamination. Points (\ref{endbdry})-(\ref{polygons}) above are therefore satisfied by any model constructed in this way. On the other hand, point (\ref{hereditary}) does not hold for the geometric model of an arbitrary gentle algebra, and is a special property of those of types  $A$ and $\tilde{A}$.
\end{remark}

\noindent The quiver $Q$ can be recovered from the geometric model. This comes directly from \cite[Section~1.5]{OPS}, noting that in the hereditary cases that we are considering, there is no ideal of relations: 
\begin{enumerate} 
\item There is a vertex of $Q$ corresponding to each curve in the lamination $L_A$.
\item Suppose two curves in $L_A$ corresponding to vertices $i$ and $j$ both have an endpoint on the same boundary component of $\Sigma$. If the endpoint of $j$ follows that of  $i$ on the boundary in the clockwise order and there are no marked points or other endpoints of curves in $L_A$ between them, then there is an arrow from $i$ to $j$.
\end{enumerate}

\begin{example}\label{AnExample}
    \begin{enumerate} 
\item \emph{Path algebras of type} $A_n$ - the surface $\Sigma$  is a disc $D$, with one boundary circle $\delta(D)$ and $M_A$ consists of $n+1$ marked points on this boundary. In the equioriented case, if we label the vertices $1, \dots, n+1$ clockwise around the boundary, there is a curve in the lamination between the boundary intervals on either side of the marked points $1, \dots, n$. Two examples are shown in Figure~\ref{A4Examples}. 
\item \emph{Path algebras of type} $\tilde A_{p,q}$ - the surface $\Sigma$ is an annulus, which has two boundary circles. $M_A$ consists of $p$ marked points on one boundary and $q$ marked points on the other.
\end{enumerate} 

\end{example}

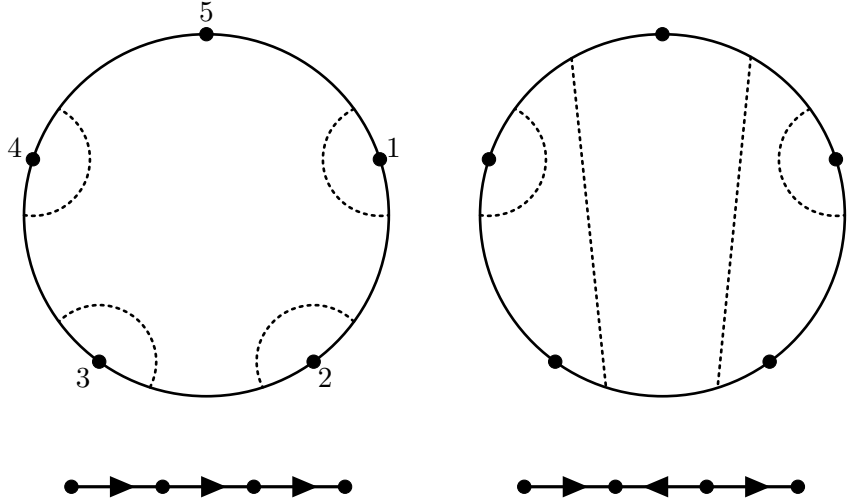
\begin{figure}
\centering
\begin{tikzpicture}[line cap=round,line join=round,>=triangle 45,x=0.6cm,y=0.6cm]
\draw [line width=1pt] (4,4) circle (0.6*4cm);
\draw [line width=1pt] (14,4) circle (0.6*4cm);
\draw [shift={(0.19577393481938543,5.23606797749979)},line width=1pt,dotted]  plot[domain=-1.7278759594743862:1.0954625862214311,variable=\t]({1*1.251475720321847*cos(\t r)+0*1.251475720321847*sin(\t r)},{0*1.251475720321847*cos(\t r)+1*1.251475720321847*sin(\t r)});
\draw [shift={(1.6488589908301075,0.7639320225002102)},line width=1pt,dotted]  plot[domain=-0.47123889803846897:2.350697128230278,variable=\t]({1*1.251475720321847*cos(\t r)+0*1.251475720321847*sin(\t r)},{0*1.251475720321847*cos(\t r)+1*1.251475720321847*sin(\t r)});
\draw [shift={(6.3511410091698925,0.7639320225002102)},line width=1pt,dotted]  plot[domain=0.7853981633974483:3.6126272814771503,variable=\t]({1*1.2514757203218472*cos(\t r)+0*1.2514757203218472*sin(\t r)},{0*1.2514757203218472*cos(\t r)+1*1.2514757203218472*sin(\t r)});
\draw [shift={(7.804226065180615,5.23606797749979)},line width=1pt,dotted]  plot[domain=2.0420352248333655:4.869468613064179,variable=\t]({1*1.2514757203218472*cos(\t r)+0*1.2514757203218472*sin(\t r)},{0*1.2514757203218472*cos(\t r)+1*1.2514757203218472*sin(\t r)});
\draw [line width=1pt,dotted] (12.006864496586166,7.46805577594007)-- (12.762508975527227,0.19623660510418617);
\draw [line width=1pt,dotted] (15.937533763574857,7.4994232260484655)-- (15.209429732284093,0.18722152195184227);
\draw [shift={(10.195773934819385,5.23606797749979)},line width=1pt,dotted]  plot[domain=-1.7278759594743862:1.0972080844543592,variable=\t]({1*1.251475720321846*cos(\t r)+0*1.251475720321846*sin(\t r)},{0*1.251475720321846*cos(\t r)+1*1.251475720321846*sin(\t r)});
\draw [shift={(17.804226065180615,5.23606797749979)},line width=1pt,dotted]  plot[domain=2.0420352248333664:4.869468613064179,variable=\t]({1*1.251475720321846*cos(\t r)+0*1.251475720321846*sin(\t r)},{0*1.251475720321846*cos(\t r)+1*1.251475720321846*sin(\t r)});

\begin{scope}[very thick,decoration={
    markings,
    mark=at position 0.7 with {\arrow{>}}}
    ] 
    \draw[postaction={decorate}] (5.04,-2) -- (7.04,-2);
    \draw[postaction={decorate}]  (3.04,-2) -- (5.04,-2);
    \draw[postaction={decorate}] (1.04,-2) -- (3.04,-2);
    \draw[postaction={decorate}]  (10.97,-2) -- (12.97,-2);
    \draw[postaction={decorate}] (14.97,-2) -- (12.97,-2);
    \draw[postaction={decorate}]   (14.97,-2) -- (16.97,-2);
\end{scope}

\draw [fill=black] (4,8.5) node {$ 5$};
\draw [fill=black] (8.1,5.5) node {$ 1$};
\draw [fill=black] (6.6,0.4) node {$ 2$};
\draw [fill=black] (1.3,0.4) node {$ 3$};
\draw [fill=black] (-0.2,5.5) node {$ 4$};
\draw [fill=black] (4,8) circle (2.5pt);
\draw [fill=black] (7.804226065180615,5.23606797749979) circle (2.5pt);
\draw [fill=black] (6.3511410091698925,0.7639320225002102) circle (2.5pt);
\draw [fill=black] (1.6488589908301075,0.7639320225002102) circle (2.5pt);
\draw [fill=black] (0.19577393481938543,5.23606797749979) circle (2.5pt);
\draw [fill=black] (14,8) circle (2.5pt);
\draw [fill=black] (10.195773934819385,5.23606797749979) circle (2.5pt);
\draw [fill=black] (11.648858990830107,0.7639320225002102) circle (2.5pt);
\draw [fill=black] (16.351141009169893,0.7639320225002093) circle (2.5pt);
\draw [fill=black] (17.804226065180615,5.23606797749979) circle (2.5pt);
\draw [fill=black] (3.04,-2) circle (2.5pt);
\draw [fill=black] (5.04,-2) circle (2.5pt);
\draw [fill=black] (7.04,-2) circle (2.5pt);
\draw [fill=black] (1.04,-2) circle (2.5pt);
\draw [fill=black] (10.97,-2) circle (2.5pt);
\draw [fill=black] (12.97,-2) circle (2.5pt);
\draw [fill=black] (14.97,-2) circle (2.5pt);
\draw [fill=black] (16.97,-2) circle (2.5pt);
\end{tikzpicture}
\caption{Examples of the geometric models for $D^b(\mod{\mathbf{k}Q})$, where $Q$ is a quiver of type $A_4$. The corresponding quiver $Q$ is drawn below each model. The dotted lines indicate the curves in the lamination.} \label{A4Examples}
\end{figure}

An arc on the geometric model is a smooth curve $\gamma:[0,1] \to \Sigma$ such that $\gamma(0),\gamma(1) \in M_A$. We consider arcs up to end-point fixing homotopy. A grading of a curve $\gamma$ is a sequence of integers, indexed by the intersection points of the curve $\gamma$ and curves in the lamination $L_A$. This sequence must satisfy certain conditions given in \cite[Definition~2.3]{OPS}. An arc is always gradable, but in general, the existence of a grading imposes an extra condition on a closed curve. In the $\tilde{A}_n$-case however, the non-trivial closed curves are all gradable. We call a closed curve primitive if it is not homotopy equivalent to $\gamma^n$ for some other closed curve $\gamma$. 

It is shown in \cite{OPS} that there is a bijective correspondence between the set of indecomposable objects in the bounded derived category of the gentle algebra, and the set
\[ \operatorname{Ind} = \operatorname{Arc} \cup (\operatorname{Band} \times \mathbf{k}^* \times \IZ_{>0})\]
where $\operatorname{Arc}$ is the set of non-trivial graded arcs up to end point fixing homotopy, and $\operatorname{Band}$ is the set of grades closed primitive curves up to homotopy.

Given any graded arc, a projective resolution of the corresponding object can be obtained by looking at the sequence of curves in the lamination it crosses. For example, the projective modules, considered as complexes concentrated in degree zero, correspond to the graded arcs that intersect a single curve in the lamination in degree zero. 
\begin{example} \label{ExampleBarcode}
\begin{enumerate} 
\item \emph{Type} $A_n$. Consider again the equioriented $A_n$ case (see Example~\ref{AnExample}). We recall from Theorem~\ref{Gabriel} and the hereditary property, that all indecomposable summands are, up to shift, interval representations. For each $a \in \{1, \dots, n\}$, the indecomposable projective module $P_a$ corresponds to the interval $[a,n+1)$. The indecomposable object $V_{[a,b)}$ corresponding to the interval $[a,b)$ has projective resolution $0 \to P_b \to P_a \to V_{[a,b)}$. A direct calculation then shows that $V_{[a,b)}$ corresponds to the arc between vertices $a$ and $b$. It is graded such that it crosses the curve in the lamination corresponding to $P_a$ in degree zero and the curve in the lamination corresponding to $P_b$ in degree $-1$.
\item \emph{Type} $\tilde A_{p,q}$ - the annulus $\Sigma$ has two boundary circles and so the arcs on $\Sigma$ split into three types. The arcs between the two boundaries correspond to indecomposable objects in the non-regular components, with the grading determining whether they are (shifts of) objects in the preprojective or preinjective component. The arcs with end-points both lying on the boundary circle with $p$ (respectively $q$) marked points, correspond to the indecomposable objects in a rank $p$ (respectively $q$) tube. Objects in the homogeneous tubes correspond to primitive graded closed curves. These tubes come in families indexed by $\mathbf{k}^*$, and a positive integer indexes the length of the object in the tube.  
\end{enumerate} 
    
\end{example}

\subsubsection{Geometric models and the AR-quiver}
We now consider how the AR-quiver relates to the geometric model. Let $(\Sigma, M_A)$ be a geometric model corresponding to $D^b(\mod{\mathbf{k}Q})$ for some gentle quiver $Q$. There is an orientation of each boundary component of $\Sigma$ induced by the orientation of the surface. 
Given a marked point $u$ on the boundary, define $u^+$ to be the next marked point in the positive direction along the boundary, and let $\sigma_u$ be an arc from $u$ to $u^+$ that is homotopic to the part of the boundary $\delta(\Sigma)$ going from $u$ to $u^+$ in a positive direction.

\begin{definition}\label{endpointsdefn} Consider a non-trivial arc $\gamma$ such that $\gamma(0)=u$ and $\gamma(1)= v$. Then:
\begin{itemize} 
    \item  $s(\gamma)$ is the arc that is homotopic to the concatenation $ \gamma \cdot \sigma_u^{-1}$. In other words, $s(\gamma)$ is obtained from $\gamma$ by moving the start of $\gamma$ along the boundary to the next marked point in the positive direction. 
   \item  $t(\gamma) := (s(\gamma^{-1}))^{-1} \simeq \sigma_v \cdot \gamma $, so $t(\gamma)$ is obtained from $\gamma$ by moving the end of $\gamma$ along the boundary to the next marked point in the positive direction. 
    \end{itemize}
\end{definition}

\begin{remark} \label{remarkends} \phantom{drfghj}
\begin{enumerate}
\item If $s(\gamma)$ (respectively $t(\gamma)$) is a trivial arc, we treat this as a well-defined object that corresponds to the zero object in the derived category.
\item For simplicity, we adopt the convention that $s,t$ are defined and fix a trivial arc.  
\item If $s(\gamma)$ and $t(\gamma)$ are both non-trivial, then $ts(\gamma)=st(\gamma)$.
\item  If $\underline{\gamma} =(\gamma, f)$ is a graded arc then there are induced gradings on $s(\gamma)$ and $t(\gamma)$ (see \cite[Section~5]{OPS}) and we denote corresponding graded arcs by $s(\underline{\gamma})$ and $t(\underline{\gamma})$ respectively.
\item These operations are analogous to moving the end points of the intervals in a classic barcode.
    \item \cite[Theorem~5.1]{OPS} shows that these operations can be used to see the squares of the Auslander-Reiten mesh, which are all of the form
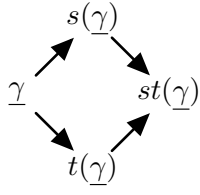
\begin{center}
\definecolor{ududff}{rgb}{0,0,0}
\begin{tikzpicture}[line cap=round,line join=round,>=triangle 45,x=1cm,y=1cm]
\draw [->,line width=1pt] (2.25,2.25) -- (2.75,2.75);
\draw [->,line width=1pt] (3.25,2.75) -- (3.75,2.25);
\draw [->,line width=1pt] (2.25,1.75) -- (2.75,1.25);
\draw [->,line width=1pt] (3.25,1.25) -- (3.75,1.75);
\draw [fill=ududff] (2,2) node {$\underline{\gamma}$};
\draw [fill=ududff] (3,3) node {$s(\underline{\gamma})$};
\draw [fill=ududff] (3,1) node {$t(\underline{\gamma})$};
\draw [fill=ududff] (4,2) node {$st(\underline{\gamma})$};

\end{tikzpicture}
    
\end{center}
\item In particular, if $M$ is an indecomposable object corresponding to a graded arc $\underline{\gamma}=(\gamma, f)$ then $\tau^{-1}M$ corresponds to the graded arc $st(\underline{\gamma})$.
\end{enumerate}
\end{remark}   

We can also define operations $s$ and $t$ on objects corresponding to closed curves. These could also be viewed geometrically as operations that increase (respectively decrease) the winding number of the closed curve. However, since we have taken the convention that the closed curves are primitive but have additional data of a positive integer (indexing the length of the corresponding indecomposable object), we take a more algebraic approach. 
\begin{definition}\label{stclosed} Let $ (\underline{\gamma},  \lambda, l) \in \textbf{Band} \times \mathbf{k}^* \times \IZ_{\geq 0}$, where $\underline{\gamma}$ is a graded primitive closed curve. Then we define operations $s,t: \textbf{Band} \times \mathbf{k}^* \times \IZ_{ \geq 0} \to \textbf{Band} \times \mathbf{k}^* \times \IZ_{>0}$, by
\[ s (\underline{\gamma},  \lambda, l) = \begin{cases}
    (\underline{\gamma},  \lambda, l+1) & \text{if } l\geq 1\\
    (\underline{\gamma},  \lambda, 0)  & \text{if } l=0
\end{cases} \qquad \text{and } \qquad  t(\underline{\gamma},  \lambda, l) = \begin{cases}
    (\underline{\gamma},  \lambda, l-1) & \text{if } l\geq 1\\
    (\underline{\gamma},  \lambda, 0)  & \text{if } l=0
\end{cases} \]
\end{definition}
\noindent The comments from Remark~\ref{remarkends} still hold in this case, where an object $ (\underline{\gamma},  \lambda, l)$ is said to be trivial if $l=0$.

 \section{Bottleneck distance}\label{Bottleneck}
In this section we shall look at a second notion of distance on persistence modules. Once again, we shall recall the definition in the equioriented $A_n$ case and then rewrite this in a way that can be generalised.

\begin{definition}
Given a persistence module $M$, the barcode $\mathcal{B}(M)$ is defined to be a (representation of) the multi-set $(\mathcal{A}_M, m)$ where $\mathcal{A}_M$ is the set of graded arcs or closed curves corresponding to indecomposable summands of $M$ and $m \colon \mathcal{A}_M \to \IZ_{>0}$ is a function encoding their multiplicity.
\end{definition}

We know from Theorem~\ref{Gabriel} that in the equioriented $A_n$ case, all indecomposable objects are interval representations, up to shifting in the derived category. From the previous section, we saw that the interval $[a,b)$ corresponds to the arc between vertices $a$ and $b$ on the disc model. We consider the usual definition of a $\delta$-matching between modules.
\begin{definition}[Equioriented $A_n$-case]
Let $\delta$ be a non-negative integer. For a barcode $\mathcal{B}$, let $\mathcal{B}_\delta$ denote the subset of $\mathcal{B}$ corresponding to intervals $[a,b)$ such that $b-a > \delta$. A $\delta$-matching between barcodes $\mathcal{B}$ and $\mathcal{B}'$ is a matching $\eta : \mathcal{B} \not\to \mathcal{B}'$ such that
    \[ \mathcal{B}_{2\delta} \subseteq \operatorname{Coim}\eta, \qquad \mathcal{B}'_{2\delta} \subseteq \operatorname{Im}\eta \]
    and if $\eta$ maps an interval $[a,b)$ to $[a',b')$, then 
    \[ [a,b) \subset [a'-\delta, b'+\delta) \;\; \text{  and  } \;\; [a',b') \subset [a-\delta, b+\delta).\]
    We say that $\mathcal{B}$ and $\mathcal{B}'$ are $\delta$-matched if there is an $\delta$-matching between $\mathcal{B}$ and $\mathcal{B}'$.
\end{definition}

Since we would like to consider cases in which indecomposable representations do not correspond to intervals, we rephrase this in terms of the geometric model. We start by considering the operations on the end points of arcs from Definition~\ref{endpointsdefn} and the corresponding operations on intervals. 


\begin{lemma}
\label{equiorMatching}
    Suppose $1 \leq a < b \leq n+1$. The operations $s,t$ on arcs induce two operations on intervals as follows:
\[ \tilde{s}[a,b) =  [a-1, b) , \qquad  \tilde{t}[a,b) = [a, b-1).
    \]
where the empty set corresponds to the zero object in the category and $\tilde{s}$ is not defined if $a=1$.

\end{lemma}
\begin{proof}
    This follows from Remark~\ref{remarkends}.
    \end{proof}
 \begin{remark}
     If $(\gamma,f)$ is a graded arc corresponding to an interval $[1,b)$ for some $b$ then $s(\gamma,f)$ is a well-defined object in the derived category, but it is no longer a module concentrated in degree zero, and so doesn't correspond to an interval. 
 \end{remark}

\begin{lemma}
  Let $\delta$ be a non-negative integer and suppose $[a,b), {[a',b')} \subset [1,n+1)$ are nonempty intervals.
  \begin{enumerate}
      \item $b-a \leq \delta$ if and only if $ \tilde{t}^{k}  [a,b) = \varnothing$ for some integer $0 \leq k \leq \delta$.
      \item $[a,b) \subset [a'-\delta, b'+\delta)$ and $[a',b') \subset [a-\delta, b+\delta)$ if and only if \[ \tilde{t}^{m_1} \tilde{s}^{n_1}[a,b) = \tilde{t}^{m_2} \tilde{s}^{n_2}[a',b') \neq \varnothing\] for some integers $0 \leq n_1, n_2, m_1, m_2 \leq \delta$.
      \end{enumerate}
\end{lemma}

\begin{proof}
For the forwards direction of $(1)$, let $k = b-a$. If $k =b-a \leq \delta $ then by definition \[ \tilde{t}^{k} [a,b) = \tilde{t}^{b-a} [a,b) = [a,b-(b-a)) = \varnothing.\] Conversely if there exists $0 \leq k \leq \delta$ such that $\tilde{t}^{k} [a,b)= [a, b-k)$ is well defined and empty, then $ b-a = k \leq \delta$.

For the forwards direction of $(2)$, we consider $n_1 = \operatorname{max}(0, a-a')$ and $n_2 = \operatorname{max}(0, a'-a)$. From the interval assumptions we have $a'-\delta \leq a$ and $a-\delta \leq a'$ so $0 \leq n_1,n_2 \leq \delta$. Then clearly $\tilde{s}^{n_1}[a,b) = [\operatorname{min}(a,a'),b)$ and $\tilde{s}^{n_2}[a',b') =  [\operatorname{min}(a,a'),b')$ noting that $\operatorname{min}(a,a')>0$ so these are well defined intervals in $[1,n+1)$. Similarly, if we take $m_1 = \operatorname{max}(0, b-b')$ and $m_2 = \operatorname{max}(0, b'-b)$, then $0 \leq m_1,m_2 \leq \delta$, and $\tilde{t}^{m_1} \tilde{s}^{n_1}[a,b] = [\operatorname{min}(a,a'),\operatorname{min}(b,b')] = \tilde{t}^{m_2} \tilde{s}^{n_1}[a',b'] $ observing that $\operatorname{min}(a,a') < \operatorname{min}(b,b')$ so this is again a non-empty interval. Conversely, if $\tilde{t}^{m_1} \tilde{s}^{n_1}[a,b) = \tilde{t}^{m_2} \tilde{s}^{n_2}[a',b') \neq \varnothing$, then it follows from the definitions that $a-n_1 = a' -n_2$. Therefore, $a = a' - (n_2-n_1) \geq a' - \delta $ and $a' = a - (n_1-n_2) \geq a - \delta $. The inequalities for $b$ and $b'$ follow similarly.
\end{proof}

We can consider these properties in a more general setting.
\begin{definition}
    Suppose $\underline{\gamma}= (\gamma,f)$ and $\underline{\gamma'}=  (\gamma',f')$ are graded arcs and let $\delta \in \IZ_{\geq 0}$. 
    \begin{itemize}
        \item $\underline{\gamma}$ is called $\delta$-short if there exists $0 \leq k \leq \delta$ such that either $ {t}^{k}\underline{\gamma}$ or $ {s}^{k}\underline{\gamma}$ is trivial. 
        \item We define the length of $\underline{\gamma}$ to be $\ell(\underline{\gamma}) = \inf\{\delta \in \IZ_{\geq 0} \mid \underline{\gamma} \text{ is $\delta$-short}\}$.
        \item $\underline{\gamma}$ and $ \underline{\gamma}'$ are called $\delta$-equivalent if there exist integers $0 \leq n_1, n_2, m_1, m_2 \leq \delta$ such that \[ {t}^{m_1} {s}^{m_2}\underline{\gamma} = {t}^{n_1} {s}^{n_2}\underline{\gamma}' \]
        is non-trivial.
    \end{itemize}
\end{definition}

\begin{lemma} \label{stmovements}
    Suppose $\Sigma$ is a geometric model of type $\tilde{A}$ and let $\underline{\gamma}= (\gamma,f)$ be a graded arc or closed curve of finite length $\ell(\underline{\gamma})$. After a suitable choice of orientation of $\gamma$, for any $k_1, k_2 \in \IZ_{\geq 0}$, such that $k_1-k_2 < \ell(\underline{\gamma})$, then
    \[\ell({t}^{k_1} {s}^{k_2} \underline{\gamma}) = \ell(\underline{\gamma})-k_1+k_2\]
    If $k_1-k_2 \geq \ell(\underline{\gamma})$ then ${t}^{k_1} {s}^{k_2} \underline{\gamma}$ is trivial.
\end{lemma}
\begin{proof}
    The graded closed curve case follows straight from Definition\ref{stclosed}. Now suppose an arc $\gamma$ has finite length, so both endpoints must lie on the same boundary component $B$. We consider the universal cover of the annulus $\Sigma$ which is homeomorphic to an infinite strip $\IR \times [0,1]$. Without loss of generality we may assume that $B$ lifts to $\tilde{B} \simeq\IR \times \{0\}$ with the natural orientation, and that the marked points on $B$ lift to the integral points $\IZ \times \{0\}$. Note that any arc $\tilde{\alpha}$ on the universal cover is completely determined by its end points $(\tilde{\alpha}(0),\tilde{\alpha}(1))$. An arc $\alpha$ on $\Sigma$ is trivial if and only if $\tilde{\alpha}(0)=\tilde{\alpha}(1)$ for any lift $\tilde{\alpha}$. The definitions of the operations $s$ and $t$ also lift naturally to the universal cover where for any non-trivial arc $\tilde{\alpha}$ then $\tilde{s}(\tilde{\alpha})$ corresponds to the pair $(\tilde{\alpha}(0)+1,\tilde{\alpha}(1))$ and $\tilde{t}(\tilde{\alpha})$ corresponds to the pair $(\tilde{\alpha}(0),\tilde{\alpha}(1)+1)$. 
    
     Suppose $\ell(\underline{\gamma})=l$, so $l$ is minimal such that either $ {t}^{l}  \underline{\gamma}$ or ${s}^{l} \underline{\gamma}$ is trivial. Up to a choice of orientation, we may assume that $ {t}^{l}  \underline{\gamma}$ is trivial. Then $ \tilde{\gamma}(1)=\tilde{\gamma}(0)-l$. Now consider ${t}^{k_1} {s}^{k_2} \tilde{\underline{\gamma}}$, where $k_1-k_2 < l$. By construction ${t}^{k_1} {s}^{k_2} \tilde{\underline{\gamma}}(0) = \tilde{\underline{\gamma}}(0) +k_2$ and ${t}^{k_1} {s}^{k_2} \tilde{\underline{\gamma}}(1) = \tilde{\underline{\gamma}}(1) +k_1 = \tilde{\underline{\gamma}}(0) -l +k_1$. Then \[{s}^{m}({t}^{k_1} {s}^{k_2} \tilde{\underline{\gamma}})(1) = \tilde{\underline{\gamma}}(0) -l +k_1 <  \tilde{\underline{\gamma}}(0)+k_2 +m = {s}^{m}({t}^{k_1} {s}^{k_2} \tilde{\underline{\gamma}})(0)\] so ${s}^{m}({t}^{k_1} {s}^{k_2} {\underline{\gamma}})$ is not trivial for any $m\geq 0$. However \[{t}^{m}({t}^{k_1} {s}^{k_2} \tilde{\underline{\gamma}})(1) = \tilde{\underline{\gamma}}(0) -l +k_1 +m =  \tilde{\underline{\gamma}}(0)+k_2 = {t}^{m}({t}^{k_1} {s}^{k_2} \tilde{\underline{\gamma}})(0)\] holds precisely when $m=l-k_1+k_2$.  Therefore, $\ell({t}^{k_1} {s}^{k_2} \underline{\gamma})= m = l-k_1+k_2$.

\end{proof}
\begin{remark}\label{choice}
   From now on, we will consistently orient all finite length $\underline{\gamma}$ arcs such that $t^{\ell(\underline{\gamma})} \underline{\gamma}$ is trival. Note that Lemma~\ref{stmovements} then implies that $s^k\underline{\gamma}$ is not trivial for any $k \geq 0$. 
\end{remark}

We give some sufficient conditions for the operations $s$ and $t$ to commute.

\begin{lemma} \label{partcommutivity}
    Suppose $\Sigma$ is a geometric model of type $\tilde{A}$ and let $\underline{\gamma}= (\gamma,f)$ be a graded arc. 
    If $\ell({t}^{k_1} \underline{\gamma})$ is strictly positive or infinite, then ${t}^{k_1} {s}^{k_2} \underline{\gamma}= {s}^{k_2}{t}^{k_1} \underline{\gamma}$.
\end{lemma}
\begin{proof}
  Since $s(\gamma) = \gamma \cdot \sigma_u^{-1}$ is nontrivial by Remark~\ref{choice}, then $ts(\gamma) = \sigma_v \cdot (\gamma \cdot \sigma_u^{-1})$. If $\ell(t\gamma) >0$ then $t(\gamma) = \sigma_v \cdot \gamma$ is also nontrivial, so $st(\gamma) = (\sigma_v \cdot \gamma) \cdot \sigma_u^{-1}$. The concatenation of arcs is an associative operation, so it follows that $st(\gamma)= ts(\gamma)$. Applying this iteratively, using Lemma~\ref{stmovements}, the result follows.
\end{proof}

\subsection{Bottleneck distance}
We are now in a position to define the bottleneck distance between two barcodes in cases of type $\tilde{A}$.

\begin{definition}[$\tilde A_n$-case]\label{matching}
Let $\delta$ be a non-negative integer. For a barcode $\mathcal{B}$, let $\mathcal{B}_\delta$ denote the subset of $\mathcal{B}$ consisting of objects of length $>\delta$. A \emph{$\delta$-matching} between barcodes $\mathcal{B}$ and $\mathcal{B}'$ is a matching $\eta : \mathcal{B} \not\to \mathcal{B}'$ such that
    \[ \mathcal{B}_{2\delta} \subseteq \operatorname{Coim}\eta, \qquad \mathcal{B}'_{2\delta} \subseteq \operatorname{Im}\eta \]
    and if $\eta(N) = N'$, then $N$ and $N'$ are $\delta$-equivalent. \\
    We say that $\mathcal{B}$ and $\mathcal{B}'$ are \emph{$\delta$-matched} if there is a $\delta$-matching between $\mathcal{B}$ and $\mathcal{B}'$.
\end{definition}

\begin{remark} \label{SymmetryB}
   By considering the inverse matching $\eta^{-1}: \mathcal{B}' \not\to \mathcal{B}$, we observe that $\mathcal{B}$ and $\mathcal{B}'$ are $\delta$-matched if and only if $\mathcal{B}'$ and $\mathcal{B}$ are $\delta$-matched. 
\end{remark}

\begin{definition}
    We define the bottleneck distance between two barcodes $\mathcal{B}$ and $\mathcal{B}'$ by
\[ d_B(\mathcal{B},\mathcal{B}') = \inf\{\delta \in \IZ_{\geq 0} \mid \mathcal{B}\; \text{and} \; \mathcal{B}' \; \text{are $\delta$-matched}\}. \]
Again we use the convention that $\inf \varnothing = \infty$.
    \end{definition}

\begin{lemma} \label{bottlemetric}
    This is a well-defined (possibly infinite) metric on the space of all barcodes.
\end{lemma}
\begin{proof}
If $\mathcal{B}$ is a barcode, then the identity map induces a bijection $\eta : \mathcal{B} \not\to \mathcal{B}$ such that 
 \[ \mathcal{B}_{0} = \mathcal{B} = \operatorname{Coim}\eta = \operatorname{Im}\eta. \]
 Any arc in $\mathcal{B}$ is $0$-equivalent to itself, so $d_B(\mathcal{B},\mathcal{B})= 0$.
    If $\mathcal{B}$ and $\mathcal{B}'$ are any two barcodes, then clearly $d_B(\mathcal{B},\mathcal{B}')\geq 0$. If $d_B(\mathcal{B},\mathcal{B}')= 0$ then there exists $\eta : \mathcal{B} \not\to \mathcal{B}'$ such that 
 \[ \mathcal{B} = \operatorname{Coim}\eta, \; \text{and} \; \mathcal{B'} = \operatorname{Im}\eta, \]
so $\eta$ induces a bijection between the elements of $\mathcal{B}$ and $\mathcal{B}'$. Furthermore, under this bijection,
if $\eta(\underline{\gamma}) = \underline{\gamma}' $, then $\underline{\gamma}$ and $\underline{\gamma}'$ are $0$-equivalent and so $\underline{\gamma} \cong \underline{\gamma}'$. Therefore $\mathcal{B}=\mathcal{B}'$. Symmetry follows from Remark~\ref{SymmetryB}.

To prove the triangle inequality, suppose $d_1= d_B(\mathcal{B},\mathcal{B}'), d_2 =d_B(\mathcal{B}',\mathcal{B}'') < \infty$. Since $\IZ$ is discrete, the infima are obtained and there exists a $d_1$-matching $\eta_1 \colon \mathcal{B} \not\to \mathcal{B}'$ and a $d_2$-matching $\eta_2 \colon \mathcal{B}' \not\to \mathcal{B}''$. Composing these, we obtain a matching $\eta_2 \circ \eta_1 \colon \mathcal{B} \not\to \mathcal{B}''$. 

Suppose $\underline{\gamma} \in \mathcal{B}_{2(d_1+d_2)}$. Then $\underline{\gamma} \in \mathcal{B}_{2 d_1} \subseteq \operatorname{Coim}\eta_1$ since $d_1+d_2 \geq d_1$. If $\underline{\gamma}'=\eta_1(\underline{\gamma})$, then $\underline{\gamma}$ and $\underline{\gamma}'$ are $d_1$-equivalent. Therefore, there exist integers $0 \leq n_1, n_2, m_1, m_2 \leq d_1$ such that  ${t}^{m_1} {s}^{m_2}\underline{\gamma} = {t}^{n_1} {s}^{n_2}\underline{\gamma}' $ is non-trivial. By Lemma~\ref{stmovements}, non-triviality implies that $m_1-m_2<\ell(\underline{\gamma})$ and $n_1-n_2<\ell(\underline{\gamma}')$.
Suppose to the contrary that $\underline{\gamma}'$ is $2d_2$-short. Again using Lemma~\ref{stmovements} 
\begin{align*}  
 \ell(\underline{\gamma}) = \ell({t}^{m_1} {s}^{m_2}\underline{\gamma} )+m_1-m_2 &= \ell({t}^{n_1} {s}^{n_2}\underline{\gamma}')+m_1-m_2\\ &= \ell(\underline{\gamma}')+m_1-m_2-n_1+n_2 \\ &\leq 2d_2 +2d_1.
\end{align*}
        However, this would contradict the assumption that $\underline{\gamma} \in \mathcal{B}_{2(d_1+d_2)}$. Therefore $\underline{\gamma}' \in \mathcal{B}_{2d_2} \subseteq \operatorname{Coim}\eta_2$. It follows that $\underline{\gamma} \in \operatorname{Coim}( \eta_2 \circ \eta_1)$ and so $\mathcal{B}_{2(d_1+d_2)} \subseteq \operatorname{Coim}( \eta_2 \circ \eta_1)$. A similar argument shows that $\mathcal{B}_{2(d_1+d_2)}'' \subseteq \operatorname{Im}( \eta_2 \circ \eta_1)$. 

Now suppose that $\eta_2 \circ \eta_1(\underline{\gamma}) =\underline{\gamma}''$ and let $\underline{\gamma}' = \eta_1(\underline{\gamma})$. Then by assumption $\underline{\gamma}, \underline{\gamma}'$ are $d_1$-equivalent and $\underline{\gamma}',\underline{\gamma}''$ are $d_2$-equivalent. Therefore, there exist integers $0 \leq n_1, n_2, m_1, m_2 \leq d_1$ and $0 \leq n_1', n_2', m_1', m_2' \leq d_2$ such that  ${t}^{m_1} {s}^{m_2}\underline{\gamma} = {t}^{n_1} {s}^{n_2}\underline{\gamma}' $ and ${t}^{m_1'} {s}^{m_2'}\underline{\gamma}' = {t}^{n_1'} {s}^{n_2'}\underline{\gamma}'' $ are both non-trivial arcs. In particular, this implies that $\ell({t}^{m_1} {s}^{m_2}\underline{\gamma})>0$  and $\ell({t}^{m_1'} {s}^{m_2'}\underline{\gamma'})>0 $. Then
\begin{align*}
    {t}^{m_1+m_1'} {s}^{m_2+m_2'}\underline{\gamma} = {t}^{m_1'} {s}^{m_2'}{t}^{m_1} {s}^{m_2}\underline{\gamma} &= {t}^{m_1'} {s}^{m_2'}{t}^{n_1} {s}^{n_2}\underline{\gamma}'\\ &= {t}^{n_1} {s}^{n_2}{t}^{m_1'} {s}^{m_2'}\underline{\gamma}' = {t}^{n_1} {s}^{n_2}{t}^{n_1'} {s}^{n_2'}\underline{\gamma}'' = {t}^{n_1+n_1'} {s}^{n_2+n_2'}\underline{\gamma}''\end{align*}
where the first and third equalities hold by Lemma~\ref{partcommutivity}. Since $0 \leq m_1+m_1', m_2+m_2', n_1+n_1', n_2+n_2'\leq d_1 +d_2$, it follows that $\underline{\gamma}$ and $\underline{\gamma}''$ are $d_1 +d_2$-equivalent and so $\mathcal{B}'$ and $\mathcal{B}''$ are $d_1 +d_2$-matched. Therefore $d_B(\mathcal{B},\mathcal{B}'') \leq d_1 +d_2 = d_B(\mathcal{B},\mathcal{B}') + d_B(\mathcal{B'},\mathcal{B}'')$ as required.

\end{proof}

In Section~\ref{Interleaving}, we showed that we could calculate the interleaving distance between two objects by restricting to the summands lying in the different components of the Auslander Reiten quiver separately. We can do something similar for the bottleneck distance.

For a barcode $\mathcal{B}$, let $\mathcal{B}_\mathcal{P}$ ($\mathcal{B}_\mathcal{I}$ or $\mathcal{B}_\mathcal{R}$) denote the subset of $\mathcal{B}$ consisting of objects which are preprojective (preinjective or regular respectively), so,
\[\mathcal{B}  = \mathcal{B}_\mathcal{P} \sqcup \mathcal{B}_\mathcal{R} \sqcup \mathcal{B}_\mathcal{I}.\]

\begin{proposition}
    Let $\mathcal{B}$ and $\mathcal{B}'$ be barcodes. Then
    \[ d_B(\mathcal{B}, \mathcal{B}') = \max\{ d_B(\mathcal{B}_\mathcal{P},\mathcal{B}'_\mathcal{P}),d_B(\mathcal{B}_\mathcal{R},\mathcal{B}'_\mathcal{R}),d_B(\mathcal{B}_\mathcal{I},\mathcal{B}'_\mathcal{I}). \}\]
\end{proposition}
\begin{proof}
Suppose $\mathcal{B}$ and $\mathcal{B}'$ are $\delta$-matched with matching $\eta$ and let $\mathcal{C} \in \{ \mathcal{P},\mathcal{I},\mathcal{R} \}$ be a component. Consider $\eta'$ to be the restriction of $\eta$ to $\mathcal{B_\mathcal{C}}$. For any $N \in \operatorname{Coim}\eta$, then $N$ and $\eta(N)$ must be in the same component, since they are $\delta$-equivalent. Therefore $\eta' \colon \mathcal{B_\mathcal{C}} \not\to \mathcal{B'_\mathcal{C}}$. Furthermore, $N$ and $\eta'(N)$ inherit $\delta$-equivalence and 
\begin{align*}   
(\mathcal{B_\mathcal{C}})_{2\delta} &= \mathcal{B}_{2\delta} \cap \mathcal{C}  \subseteq \operatorname{Coim}\eta \cap \mathcal{C} = \operatorname{Coim}\eta', \\
(\mathcal{B_\mathcal{C}})_{2\delta}' &= \mathcal{B}_{2\delta}' \cap \mathcal{C}  \subseteq \operatorname{Im}\eta \cap \mathcal{C} = \operatorname{Im}\eta'.
\end{align*} 
Therefore, $\mathcal{B_\mathcal{C}}$ and $\mathcal{B'_\mathcal{C}}$ are also $\delta$-matched. It follows that \[ d_B(\mathcal{B}, \mathcal{B}') \geq \max\{ d_B(\mathcal{B}_\mathcal{P},\mathcal{B}'_\mathcal{P}), d_B(\mathcal{B}_\mathcal{R},\mathcal{B}'_\mathcal{R}), d_B(\mathcal{B}_\mathcal{I}, \mathcal{B}'_\mathcal{I}) \}.\]
Conversely, suppose $\eta_\mathcal{P} \colon \mathcal{B_\mathcal{P}} \not\to \mathcal{B'_\mathcal{P}}$, $\eta_\mathcal{R} \colon \mathcal{B_\mathcal{R}} \not\to \mathcal{B'_\mathcal{R}}$ and $\eta_\mathcal{I} \colon \mathcal{B_\mathcal{I}} \not\to \mathcal{B'_\mathcal{I}}$ are $\delta$-matchings. Then $\eta = \eta_\mathcal{P} \sqcup \eta_\mathcal{R} \sqcup \eta_\mathcal{I} \colon \mathcal{B} \not\to \mathcal{B}'$ is a matching with the property that if $\eta(N) =N'$ then $N$ and $N'$ are $\delta$-equivalent. Furthermore,
\begin{align*}   
\mathcal{B}_{2\delta} &= (\mathcal{B}_\mathcal{P})_{2\delta} \sqcup (\mathcal{B}_\mathcal{R})_{2\delta} \sqcup (\mathcal{B}_\mathcal{I})_{2\delta}  \subseteq \operatorname{Coim}\eta_\mathcal{P} \sqcup \operatorname{Coim}\eta_\mathcal{R} \sqcup \operatorname{Coim}\eta_\mathcal{I} = \operatorname{Coim}\eta, \\
\mathcal{B}_{2\delta}' &= (\mathcal{B}'_\mathcal{P})_{2\delta} \sqcup (\mathcal{B}'_\mathcal{R})_{2\delta} \sqcup (\mathcal{B}'_\mathcal{I})_{2\delta}  \subseteq \operatorname{Im}\eta_\mathcal{P} \sqcup \operatorname{Im}\eta_\mathcal{R} \sqcup \operatorname{Im}\eta_\mathcal{I} = \operatorname{Im}\eta.
\end{align*} 
The result follows.
\end{proof}

\section{Stability for non-regular persistence modules}\label{StabilityNonreg}
We consider persistence modules that are either in the preprojective or the preinjective components in type $\tilde A$. There is a relation on the indecomposable objects in the preprojective or the preinjective components, where for any indecomposable objects $M,N$, 
\[ M \leq N \qquad \iff \Hom(M,N) \neq 0.\]

\begin{lemma}
    This is a well-defined partial order.
\end{lemma}
\begin{proof}
   Reflexivity follows from the existence of the identity map for any indecomposable object. Antisymmetry follows from the structure of the Auslander-Reiten quiver (see Proposition~\ref{nonregcomponents}). Transitivity holds, since the composition of basis morphisms is non-zero.
\end{proof}
We can also consider the corresponding partial order on arcs in the geometric model. For any indecomposable objects $M,N$ in either the preprojective or the preinjective component, the arcs $\gamma_M$ and $\gamma_N$ connect the two boundary circles of the annulus. Then $\Hom(M,N)\neq 0$ if and only if there is a graded intersection point between $\underline{\gamma}_M$ and $\underline{\gamma}_N$. Lifting $\underline{\gamma}_M$ to the universal cover, there is such an intersection point if and only if there is a graded intersection point between the lift $\tilde{\underline{\gamma}}_M$ and some lift $\tilde{\underline{\gamma}}_N$ of $\underline{\gamma}_N$. It can be seen that this happens precisely when $\underline{\gamma}_N = t^{n_1}s^{n_2} \underline{\gamma}_M$ for some $n_1, n_2 \in \IZ_{\geq 0}$. This is illustrated in the example below, where the arrows indicate the positive orientation of the boundary components.
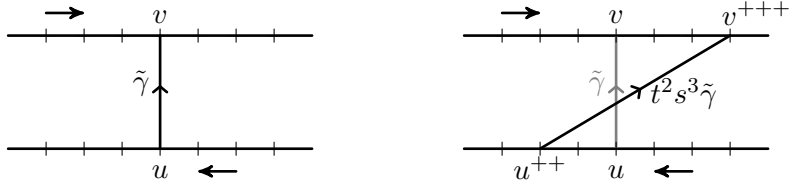
\begin{center}
    
\definecolor{yqyqyq}{rgb}{0.5019607843137255,0.5019607843137255,0.5019607843137255}
\begin{tikzpicture}[line cap=round,line join=round,x=0.5cm,y=0.5cm]
\draw [line width=1pt] (-12,3)-- (-4,3);
\draw [line width=1pt] (-12,0)-- (-4,0);
\draw [line width=1pt] (-8,0)-- (-8,3);
\draw [line width=1pt] (-8,1.65) -- (-7.82,1.5);
\draw [line width=1pt] (-8,1.65) -- (-8.18,1.5);
\draw [line width=1pt] (0,0)-- (8,0);
\draw [line width=1pt] (0,3)-- (8,3);
\draw [line width=1pt,color=yqyqyq] (4,0)-- (4,3);
\draw [line width=1pt,color=yqyqyq] (4,1.65) -- (4.18,1.5);
\draw [line width=1pt,color=yqyqyq] (4,1.65) -- (3.82,1.5);
\draw [line width=1pt] (2,0)-- (7,3);
\draw [line width=1pt] (4.628623938856879,1.5771743633141295) -- (4.592609235976952,1.345651273371742);
\draw [line width=1pt] (4.628623938856879,1.5771743633141295) -- (4.407390764023042,1.654348726628258);
\draw [->,line width=1pt] (-6,-0.6) -- (-7,-0.6);
\draw [->,line width=1pt] (-11,3.6) -- (-10,3.6);
\draw [->,line width=1pt] (6,-0.6) -- (5,-0.6);
\draw [->,line width=1pt] (1,3.6) -- (2,3.6);
\draw [color=black] (-8,0)-- ++(-2.5pt,0 pt) -- ++(5pt,0 pt) ++(-2.5pt,-2.5pt) -- ++(0 pt,5pt);
\draw[color=black] (-8,-0.6) node {$u$};
\draw [color=black] (-8,3)-- ++(-2.5pt,0 pt) -- ++(5pt,0 pt) ++(-2.5pt,-2.5pt) -- ++(0 pt,5pt);
\draw[color=black] (-8,3.5) node {$v$};
\draw[color=black] (-8.5,1.73) node {$\tilde \gamma$};
\draw [color=black] (4,0)-- ++(-2.5pt,0 pt) -- ++(5pt,0 pt) ++(-2.5pt,-2.5pt) -- ++(0 pt,5pt);
\draw[color=black] (4,-0.6) node {$u$};
\draw [color=black] (4,3)-- ++(-2.5pt,0 pt) -- ++(5pt,0 pt) ++(-2.5pt,-2.5pt) -- ++(0 pt,5pt);
\draw[color=black] (4,3.5) node {$v$};
\draw[color=black] (7.7,3.6) node {$v^{+++}$};
\draw[color=yqyqyq] (3.5,1.73) node {$\tilde\gamma$};
\draw [color=black] (2,0)-- ++(-2.5pt,0 pt) -- ++(5pt,0 pt) ++(-2.5pt,-2.5pt) -- ++(0 pt,5pt);
\draw[color=black] (2.,-0.5) node {$u^{++}$};
\draw [color=black] (7,3)-- ++(-2.5pt,0 pt) -- ++(5pt,0 pt) ++(-2.5pt,-2.5pt) -- ++(0 pt,5pt);
\draw[color=black] (5.78,1.47) node {$t^2 s^3 \tilde \gamma$};
\draw [color=black] (3,0)-- ++(-2.5pt,0 pt) -- ++(5pt,0 pt) ++(-2.5pt,-2.5pt) -- ++(0 pt,5pt);
\draw [color=black] (5,3)-- ++(-2.5pt,0 pt) -- ++(5pt,0 pt) ++(-2.5pt,-2.5pt) -- ++(0 pt,5pt);
\draw [color=black] (6,3)-- ++(-2.5pt,0 pt) -- ++(5pt,0 pt) ++(-2.5pt,-2.5pt) -- ++(0 pt,5pt);
\draw [color=black] (-11,3)-- ++(-2.5pt,0 pt) -- ++(5pt,0 pt) ++(-2.5pt,-2.5pt) -- ++(0 pt,5pt);
\draw [color=black] (-10,3)-- ++(-2.5pt,0 pt) -- ++(5pt,0 pt) ++(-2.5pt,-2.5pt) -- ++(0 pt,5pt);
\draw [color=black] (-9,3)-- ++(-2.5pt,0 pt) -- ++(5pt,0 pt) ++(-2.5pt,-2.5pt) -- ++(0 pt,5pt);
\draw [color=black] (-7,3)-- ++(-2.5pt,0 pt) -- ++(5pt,0 pt) ++(-2.5pt,-2.5pt) -- ++(0 pt,5pt);
\draw [color=black] (-6,3)-- ++(-2.5pt,0 pt) -- ++(5pt,0 pt) ++(-2.5pt,-2.5pt) -- ++(0 pt,5pt);
\draw [color=black] (-5,3)-- ++(-2.5pt,0 pt) -- ++(5pt,0 pt) ++(-2.5pt,-2.5pt) -- ++(0 pt,5pt);
\draw [color=black] (-11,0)-- ++(-2.5pt,0 pt) -- ++(5pt,0 pt) ++(-2.5pt,-2.5pt) -- ++(0 pt,5pt);
\draw [color=black] (-10,0)-- ++(-2.5pt,0 pt) -- ++(5pt,0 pt) ++(-2.5pt,-2.5pt) -- ++(0 pt,5pt);
\draw [color=black] (-9,0)-- ++(-2.5pt,0 pt) -- ++(5pt,0 pt) ++(-2.5pt,-2.5pt) -- ++(0 pt,5pt);
\draw [color=black] (-7,0)-- ++(-2.5pt,0 pt) -- ++(5pt,0 pt) ++(-2.5pt,-2.5pt) -- ++(0 pt,5pt);
\draw [color=black] (-6,0)-- ++(-2.5pt,0 pt) -- ++(5pt,0 pt) ++(-2.5pt,-2.5pt) -- ++(0 pt,5pt);
\draw [color=black] (-5,0)-- ++(-2.5pt,0 pt) -- ++(5pt,0 pt) ++(-2.5pt,-2.5pt) -- ++(0 pt,5pt);
\draw [color=black] (1,3)-- ++(-2.5pt,0 pt) -- ++(5pt,0 pt) ++(-2.5pt,-2.5pt) -- ++(0 pt,5pt);
\draw [color=black] (2,3)-- ++(-2.5pt,0 pt) -- ++(5pt,0 pt) ++(-2.5pt,-2.5pt) -- ++(0 pt,5pt);
\draw [color=black] (3,3)-- ++(-2.5pt,0 pt) -- ++(5pt,0 pt) ++(-2.5pt,-2.5pt) -- ++(0 pt,5pt);
\draw [color=black] (1,0)-- ++(-2.5pt,0 pt) -- ++(5pt,0 pt) ++(-2.5pt,-2.5pt) -- ++(0 pt,5pt);
\draw [color=black] (5,0)-- ++(-2.5pt,0 pt) -- ++(5pt,0 pt) ++(-2.5pt,-2.5pt) -- ++(0 pt,5pt);
\draw [color=black] (6,0)-- ++(-2.5pt,0 pt) -- ++(5pt,0 pt) ++(-2.5pt,-2.5pt) -- ++(0 pt,5pt);
\draw [color=black] (7,0)-- ++(-2.5pt,0 pt) -- ++(5pt,0 pt) ++(-2.5pt,-2.5pt) -- ++(0 pt,5pt);
\end{tikzpicture}
\end{center}

 We can define a bipartite graph $G(M,N)$ which has vertex set $V= V_M \sqcup V_N$ with parts 
\begin{align*} V_M &= \{M_i \mid M_i  \: \text{is an indecomposable summand of}\:  M\} \\ 
 V_N &= \{N_i \mid N_i  \: \text{is an indecomposable summand of}\:  N\} 
\end{align*}
and the set of edges is given by
\[ E = \{ (M_i, N_j) \in V_M \times V_N \mid M_i \leq \tau^{-\delta}N_j \leq \tau^{-2\delta}M_i \}. \]
Note that $G(M,N)= G(N,M)$ since the functors $\tau^{\pm 1}$ are order preserving.
\begin{lemma}\label{Halls}
Suppose $M,N$ are $\delta$-interleaved persistence modules which are both in either the preprojective component or the preinjective component. Then there is a perfect matching on the graph $G(M,N)$.
\end{lemma}
\begin{proof}
We will use Hall's Marriage theorem to prove this result. To show that there is a matching which covers $V_M$, it suffices to show that for every subset of the vertices $I \subseteq V_M$, then
\[ |I| \leq |J|,\]
where $J=\operatorname{Nbd}_G(I) \subseteq V_N$ is the neighbourhood of $I$ in $G(M,N)$. Let $I \subseteq V_M$ be any subset of the vertices and denote by 
\[ M_I = \bigoplus_{M_i \in I} M_i \; \text{and }  N_J = \bigoplus_{N_j \in J} N_j \]
the corresponding summands of $M$ and $N$ respectively.
We start by showing that the map \[\Phi^{2\delta}_{M_I}: M_I \to \tau^{-2\delta}M_I\]
must factor through $\tau^{-\delta}N_J$. 
By the definition of $\delta$-interleaving, the morphism $\Phi^{2\delta}_{M_I}: M_I \to \tau^{-2\delta}M_I$ factors through $\tau^{-\delta}N$.  If $\tau^{-\delta}N_j$ is any summand of $\tau^{-\delta}N$ through which $\Phi^{2\delta}_{M_I}$ factors non-trivially, then certainly  $\Hom( M_i , \tau^{-\delta}N_j) \neq 0$ and  $\Hom( \tau^{-\delta}N_j, \tau^{-2\delta}M_i) \neq 0$, so $ M_i \leq \tau^{-\delta}N_j \leq \tau^{-2\delta}M_i$. Therefore, by definition of $G(M,N)$ we have  $(M_i, N_j) \in E$ and so $ N_j \in J$. 

Consider the full subcategory $\mathcal{S}$ of persistence modules whose summands all lie in the component. There is a functor from $\mathcal{S}$ to the category of finite dimensional vector spaces $\operatorname{vect}_{\IC}$ which takes an object $M = \oplus_{i \in I}M_i$ to $\mathcal{F}M = \oplus_{i \in I}\IC$. Any non-zero morphism $f \in \Hom(M_i,N_j)$ can be written as a linear combination of `basis' morphisms $\{f_m \mid {m=1, \dots, {\dim(\Hom(M_i,N_j)})}\}$ each of which factors as a finite composition of the irreducible morphisms in the component. By construction, these basis morphisms are compatible with composition in the sense that the composition of basis morphisms is again a (nonzero) basis morphism. Then $\mathcal{F}$ acts on a morphism between indecomposable objects by $\mathcal{F}(f) =\mathcal{F}(\sum_{m} \alpha_mf_m) = \sum_{m} \alpha_m \in \IC$, and this extends linearly to morphisms between arbitrary objects.

Applying the functor $\mathcal{F}$, 
we see that
\[ \mathcal{F}( \Phi^{2\delta}_{M_I}):  \mathcal{F}( M_I) \stackrel{\scriptstyle  \mathcal{F}(f) }{\too}  \mathcal{F}(  \tau^{-\delta}N_J)  \stackrel{\scriptstyle  \mathcal{F}(\tau^{-\delta}g) }{\too}   \mathcal{F}( \tau^{-2\delta}M_I)\]
where $\mathcal{F}(f)$ and $\mathcal{F}(\tau^{-\delta}g)$ are respectively $|J| \times |I|$ and $|I| \times |J|$  matrices with complex entries. In particular, since $\mathcal{F}( \Phi^{2\delta}_{M_I})$ is a rank $|I|$ diagonal matrix, it follows that $ |I| \leq |J|$ as required. Therefore, by Hall's Marriage theorem, there is a matching of $G(M,N)$ which covers $V_M$. This implies that $|V_M| \leq |V_N|$.

A symmetric argument shows that there is a matching of $G(M,N)$ which covers $V_N$, so $|V_M| \geq |V_N|$. Therefore, $|V_M| = |V_N|$ and so the matching which covers $V_M$ (or $V_N$) is in fact a perfect matching.
\end{proof}

\begin{proposition}
        Suppose $M,N$ are $\delta$-interleaved persistence modules which are both in either the preprojective component or the preinjective component. Then the corresponding barcodes $\mathcal{B}(M)$ and $ \mathcal{B}(N)$ are $\delta$-matched and so \[ d_I(M,N) \geq d_B(\mathcal{B}(M), \mathcal{B}(N)).\]
\end{proposition}
\begin{proof}
    We consider the perfect matching defined in Lemma~\ref{Halls}. This induces a perfect matching $\eta: \mathcal{B}(M) \too \mathcal{B}(N)$. Since this matching is perfect, the conditions on $\operatorname{Coim}\eta$ and $\operatorname{Im}\eta$ are automatically satisfied. It remains to show that if $\eta(\underline{\gamma}_{M'}) = \underline{\gamma}_{N'}$, then $\underline{\gamma}_{M'}$ and $\underline{\gamma}_{N'}$ are $\delta$-equivalent. By construction, if $\eta(\underline{\gamma}_{M'}) = \underline{\gamma}_{N'}$ then there is an edge between $M'$ and $N'$ in the graph $G(M,N)$. Therefore,
    \[  M' \leq \tau^{-\delta}N' \leq \tau^{-2\delta}M' \]
In terms of the partial order on the corresponding arcs, this implies that there exist $n_1, n_2 \in \IZ_{\geq 0}$ such that $\underline{\gamma}_{\tau^{-\delta}N'} = t^{n_1}s^{n_2} \underline{\gamma}_{M'}$ and $m_1, m_2 \in \IZ_{\geq 0}$ such that $\underline{\gamma}_{\tau^{-2\delta}M'} = t^{m_1}s^{m_2} \underline{\gamma}_{\tau^{-\delta}N'}$. By definition, we also have that $\underline{\gamma}_{\tau^{-2\delta}M'} = t^{2\delta}s^{2\delta} \underline{\gamma}_{M'}$. It follows that $n_1+m_1=n_2+m_2=2\delta$, so $n_1,n_2, m_1,m_2 \leq 2\delta$.
For $i\in \{1,2\}$, let $k_i = \min\{n_i, \delta \}$. Then $0 \leq \delta- k_i\leq \delta$ and $0\leq n_i-k_i \leq 2\delta-\delta = \delta$. Since $t^{\delta}s^{\delta}\underline{\gamma}_{N'} = \underline{\gamma}_{\tau^{-\delta}N'} =t^{n_1}s^{n_2} \underline{\gamma}_{M'}$ it follows that 
\[t^{\delta-k_1}s^{\delta-k_2}\underline{\gamma}_{N'} = t^{n_1-k_1}s^{n_2-k_2} \underline{\gamma}_{M'}\]
noting that since the end points on the arcs are on different boundary components, applying the operations $s$ and $t$ never results in a trivial arc. In particular, the operations have well defined inverses. The result follows.
\end{proof}

\begin{proposition} \label{nonregisom}
        Suppose $M,N$ are persistence modules which are both in either the preprojective component or the preinjective component. If $\mathcal{B}(M)$ and $ \mathcal{B}(N)$ are $\delta$-matched, then $M$ and $N$ are $\delta$-interleaved and so \[ d_I(M,N) \leq d_B(\mathcal{B}(M), \mathcal{B}(N)).\]
\end{proposition}
\begin{proof}
     By definition there exists a matching $\eta : \mathcal{B}_M \not\to \mathcal{B}_N$ such that
    \[ (\mathcal{B}_M)_{2\delta} \subseteq \operatorname{Coim}\eta, \qquad (\mathcal{B}_N)_{2\delta} \subseteq \operatorname{Im}\eta. \]
    Since all arcs corresponding to indecomposable objects in any non-regular component have infinite length, it follows that $\eta$ is a perfect matching.
    Let $M_i$ and $N_i$ be summands of $M$ and $N$ respectively such that $\eta(\underline{\gamma}_{M_i}) = \underline{\gamma}_{N_i}$. By assumption $\underline{\gamma}_{M_i}$ and $\underline{\gamma}_{N_i}$ are $\delta$-equivalent, so there exist integers $0 \leq n_1, n_2, m_1, m_2 \leq \delta$ such that ${t}^{m_1} {s}^{m_2}\underline{\gamma}_{M_i} = {t}^{n_1} {s}^{n_2}\underline{\gamma}_{N_i} $ is non-trivial. Using Lemma~\ref{partcommutivity} it follows that \[{t}^{\delta-n_1 +m_1} {s}^{\delta -n_2 +m_2}\underline{\gamma}_{M_i} = {t}^{\delta} {s}^{\delta}\underline{\gamma}_{N_i} = \underline{\gamma}_{\tau^{-\delta}N_i}. \] For each $j = 0, \dots , \delta -n_2 +m_2 $, let $K_j$ be the indecomposable object corresponding to the arc ${s}^{j}\underline{\gamma}_{M_i} $. It follows from Remark~\ref{remarkends} that for each $j = 0, \dots , \delta -n_2 +m_2-1$, then $K_j$ and $K_{j+1}$ are objects in a square of the AR-mesh, and there is a (unique up to scaling) irreducible morphism $ u_j : K_j \to K_{j+1}$. For each $j = \delta -n_2 +m_2+1, \dots , 2\delta -(n_1 +n_2)+(m_1+m_2) $, let $K_j$ be the indecomposable object corresponding to the arc ${t}^{j-(\delta -n_2 +m_2)}{s}^{\delta -n_2 +m_2}\underline{\gamma}_{M_i} $ and for each $j= \delta -n_2 +m_2, \dots , 2\delta -(n_1 +n_2)+(m_1+m_2)-1$ let $v_j : K_j \to K_{j+1}$ be the corresponding nonzero irreducible morphism from the AR-mesh. Consider the morphism \[ f_i = v_{2\delta -(n_1 +n_2)+(m_1+m_2)-1} \circ \cdots \circ v_{\delta -n_2 +m_2} \circ u_{\delta -n_2 +m_2-1} \circ \cdots \circ u_1 \circ u_0: M_i \to \tau^{-\delta}N_i\]
    which is a composition of irreducible morphisms (arrows) in the AR quiver which goes $\delta -n_2 +m_2\geq 0$ steps diagonally up and then $\delta -n_1 +m_1\geq 0$ steps diagonally down.

    By a symmetric argument, we can construct a morphism $g_i$ from $N_i$ to $\tau^{-\delta}M_i$ which goes $\delta -m_2 +n_2\geq 0$ steps diagonally up and then $\delta -m_1 +n_1\geq 0$ steps diagonally down. Then using the mesh relations it follows that $\tau^{-\delta}g_i \circ f_i = \Phi_{M_i}^{2\delta}$, the morphism that goes $2\delta$ steps diagonally up and $2\delta$ steps diagonally down in the AR quiver. We can then define morphisms $f: M \to N$ and $g: N \to M$ as diagonal matrices with the $f_i$ (respectively $g_i$) on the diagonal. It follows that $M$ and $N$ are $\delta$-interleaved.

    \end{proof}
Putting together the two previous results, we have proved the isometry result when restricted to non-regular persistence modules.

\section{Stability for regular persistence modules}\label{StabilityReg}

Now we look at persistence modules whose summands lie in the regular component. The regular component consists of a family of standard stable rank one tubes parametriesed by $\mathbf{k}^*$ together with two standard stable tubes of ranks $p$ and $q$ respectively ($p,q \geq 1$). Because distinct tubes are orthogonal to each other (there are no morphisms between them), the argument from the proof of Lemma~\ref{componentwise} shows that $M,N$ are $\delta$-interleaved if and only if $M_\mathcal{T}$ and $N_\mathcal{T}$ are $\delta$-interleaved for each tube $\mathcal{T}$, where $M_\mathcal{T}$ and $N_\mathcal{T}$ are the summands of $M$ and $N$ respectively that lie in $\mathcal{T}$, or the zero object if there is no summand in $\mathcal{T}$. Therefore, we assume without loss of generality that $M,N$ are $\delta$-interleaved persistence modules whose indecomposable summands all lie in a tube $\mathcal{T}$.

We recall that each tube is a hereditary $\mathbf{k}$-linear abelian category, with quasi-simple objects along the mouth (these are simple in the tube, but may not be simple in the bigger module category). The tubes are uniserial, so every indecomposable object $M$ of the tube has a unique composition series, and a well-defined (composition) length $\ell(M)$. The following lemma justifies the slight abuse of the notation $\ell$.

\begin{lemma}
    For any indecomposable object $M$ in a tube corresponding to an arc $\gamma_M$, the composition length $\ell(M)$ and the length $\ell(\gamma_M)$ coincide.
\end{lemma}
\begin{proof}
    The arc corresponding to a quasi-simple $S$ is an arc of the form $\gamma_S =\sigma_v^{-1}$ for some vertex $v$ on the boundary. Therefore $t(\gamma_S)$ is trivial and so $\ell(S)=1=\ell(\gamma_S)$. An indecomposable object $M$ with composition length $2$ is the extension of $2$ quasi-simples, so there is a short exact sequence $0 \to S' \to M \to S \to 0$. It follows that $M$ is the cone of the (unique up to scaling) morphism $S[-1] \to S'$ in the derived category and, using the description of cones in the geometric model, that $\gamma_M \simeq \gamma_{S'} \cdot \gamma_{S}$, wher the dot denotes the concatenation of arcs. Therefore $t(\gamma_M)= \gamma_S$ which is non-trivial, but $t^2(\gamma_M)$ is trivial. A straightforward inductive argument then completes the proof.
\end{proof}

In this section we will use the notation $\hom(A,B)$ to denote the dimension of $\Hom(A,B)$.
We can use the uniseriality property of the tubes to calculate the dimension of $\Hom$ spaces. In particular, for any indecomposable object $M$, then $\hom(S,M) = 1$ if $S$ is the quasi-simple socle of $M$, and $\hom(S,M) = 0$ if $S$ is any of the other quasi-simples.

\begin{lemma} \label{lemma-counting-summands}Let $\mathcal{T}$ be a rank $p$ standard stable tube with quasi-simples $S_1, \dots , S_p$, and let $\sigma = \bigoplus_{i=1}^p S_i$. Suppose $M$ is any object in the tube $\mathcal{T}$. Then:
\begin{enumerate}
\item The number of indecomposable summands of $M$ with socle $S$ is equal to $\hom(S, M)$. 
\item The number of indecomposable summands of $M$ with top $S$ is equal to $\hom(M,S)$. 
\item The number of indecomposable summands of $M$ is equal to $\hom(\sigma, M)= \hom(M, \sigma)$.
\end{enumerate}
\end{lemma}
\begin{proof}
 For any object $M$ in the tube, we decompose $M =  \bigoplus_{i=1}^a M_i$ into indecomposable objects. Then
 \[ \hom(s, M) =  \sum_{j=1}^a \hom(s, M_j) \]
 counts the number of indecomposable summands which have socle $s$. The argument for the summands with top $s$ is dual. For the total number of indecomposable summands, 
\[ \hom(\sigma, M) =  \sum_{j=1}^a\sum_{i=1}^p \hom( s_i, M_j) = \sum_{j=1}^a 1 =a.\]
\end{proof}

\subsection{Canonical Injections} \label{caninj}
Following \cite{BauerLesnick} we consider canonical injections as follows. \\
Suppose $M,N$ are persistence modules and $\hom(S, M) \leq \hom(S, N)$ for each quasi-simple $S$. We order the indecomposable summands with socle $S$ by length, so \[ \ell(M_1) \geq \ell(M_2)\geq \dots \qquad \text{ and  } \qquad\ell(N_1) \geq \ell(N_2)\geq \dots\] Then there exist canonical injections 
\[ \{ \text{summands of $M$ with socle $S$}\} \hookrightarrow \{ \text{ summands of $N$ with socle $S$}\}\]
where each indecomposable summand $M_i$ maps to $N_i$. This induces an injection for all indecomposable summands,
\[ \{ \text{indecomposable summands of $M$ }\} \hookrightarrow \{ \text{indecomposable summands of $N$}\},\] which maps the $i$th longest summand of $M$ with socle $S$ to the $i$th longest summand of $N$ with socle $S$, for each quasi-simple $S$ and each $1 \leq i \leq \hom(S, M)$.

If $M,N$ are persistence modules and $\hom(M, S) \geq \hom(N, S)$ for each quasi-simple $S$, then we can define canonical injections dually. 

\begin{theorem}[Structure theorem for regular modules] \label{StructureThm} Let $M$ and $N$ be objects in a tube $\mathcal{T}$:
\begin{enumerate} 
\item If $M \hookrightarrow N$ is an injective morphism in $\mathcal{T}$, then for each quasi-simple $S \in \mathcal{T}$,
 \[ \hom(S, M) \leq \hom(S, N).\]
 If $M_i$ is an indecomposable summand of $M$ such that $\operatorname{soc}(M_i)=S$, then
 the canonical injection maps $M_i$ to a summand $N_j$ of $N$,  such that $\operatorname{soc}(N_j)=S$ and $\ell(M_i) \leq \ell(N_j)$.  
\item If $M \twoheadrightarrow N$ is an surjective morphism in $\mathcal{T}$, then for each quasi-simple $S \in \mathcal{T}$,
 \[  \hom(M,S) \geq \hom(N,S). \] 
 If $N_i$ is an indecomposable summand of $N$ such that $\operatorname{top}(N_i)=S$, then
 the canonical injection maps $N_i$ to a summand $M_j$, such that $\operatorname{top}(N_j)=S$ and $\ell(N_i) \leq \ell(M_j)$.  
\end{enumerate}
\end{theorem}
\begin{proof}
 We prove the first statement, the other one is dual. By Lemma~\ref{lemma-counting-summands}, the number of indecomposable summands of $M$ with socle $S$ is equal to $\hom(S, M)$. Since $\Hom(S,-)$ is a left exact functor and $M \hookrightarrow N$ is an injection, it follows that $\Hom(S,M) \hookrightarrow \Hom(S,N)$ is also an injection and so $\hom(S,M) \leq \hom(S,N)$.

 Now let $a_k \in \mathcal{T}$ be the indecomposable object with socle $S$ and length $k$ and let $r \in \Hom(S,a_k)$ be the unique non-zero morphism (up to scaling). An indecomposable object $x$ with socle $S$ has length $\ell(x) \geq k$ precisely when the (unique up to scaling) non-zero morphism $\phi \in \Hom(S,x)$ factors through the object $a_k$, so $\phi= \tilde{\phi} \circ r = r^* (\tilde{\phi})$ for some $\tilde{\phi} \in \Hom(a_k,x)$. In particular, the dimension of \[ r^*(\Hom(a_k,x)) \subseteq \Hom(S,x) \] is one when $x$ has socle $S$ and length $\ell(x) \geq k$, and is zero otherwise.

  Therefore, if we decompose $M =  \bigoplus_{i=1}^a M_i$ into indecomposable summands, the dimension of $r^*(\Hom(a_k, M))$ counts the number of indecomposable summands that have a socle $S$ and length $\ell(x) \geq k$.
 Since $\Hom(S,-)$ is a left exact functor and $\iota \colon M \hookrightarrow N$ is an injection, there is a commutative diagram with exact rows:

\begin{center}
  \begin{tikzcd} 0 \arrow[r,] &\Hom(a_k,M)\arrow[r,"\iota_*"]\arrow[d,"r^*"] &\Hom(a_k,N)\arrow[d,"r^*"]\\ 
  0 \arrow[r,] &\Hom(S,M)\arrow[r,"\iota_*"] &\Hom(S,N) \end{tikzcd}  
\end{center}
which induces an injective map $r^*(\Hom(a_k, M)) \hookrightarrow r^*(\Hom(a_k, N))$. Therefore, for any $k\geq 0$, there are at least as many summands of $N$ with socle $S$ of length $\geq k$, as there are summands of $M$ with socle $S$ of length $\geq k$. It follows that the canonical injection maps any summand of $M$ length $k$ to a summand of $N$ of length at least $k$. 
 \end{proof}

\subsection{The induced matching associated to a morphism}
Following the general structure of the argument in \cite{BauerLesnick}, we show that given any morphism $f:M \to N$ of modules where all summands are regular, there is an induced matching $\eta_f \colon \mathcal{B}(M) \not\to \mathcal{B}(N)$. First, we factorise $f: M \twoheadrightarrow \operatorname{Im}f \hookrightarrow N$ as a surjection followed by an injection.
Using Theorem~\ref{StructureThm}, we see that these two maps induce two canonical injections which, in turn, induce the matching $\eta_f$
\begin{center} 
\begin{tikzcd} & \mathcal{B}({\operatorname{Im}f})\arrow[dl, hook]\arrow[dr, hook]&\\ \mathcal{B}(M)\arrow[rr,"/"{anchor=center,sloped}, "\eta_f"]& &\mathcal{B}(N). \end{tikzcd}
\end{center}
Since regular modules are neither injective nor projective, the AR-translate $\tau$ has a well defined inverse and induces a perfect matching $\theta_\delta : \mathcal{B}(N) \to \mathcal{B}(\tau^{\delta}N)$.

\begin{theorem}
    Suppose $M$ and $N$ are $\delta$-interleaved persistence modules with respect to the morphisms \[ f: M \to \tau^{-\delta}N \text{ and } g: N \to \tau^{-\delta}M. \]
    Then $\theta_\delta \circ\eta_f:\mathcal{B}_{M} \to \mathcal{B}_{N}$ is a $\delta$-matching and so \[ d_I(M,N) \geq d_B(\mathcal{B}_{M}, \mathcal{B}_{N}).\]
\end{theorem}

\begin{proof}
The composition $\Phi^{2 \delta}_M = \tau^{-\delta}g\circ f$ induces a commutative diagram with exact rows:
\begin{equation}\label{diagram}  \begin{tikzcd} 0 \arrow[r,] &\operatorname{Ker}{f} \arrow[r,""]\arrow[d,""] & M \arrow[d,"="]\arrow[r,"q"]&\operatorname{Im}{f} \arrow[d,"\psi"] \arrow[r,""]& 0\\ 
  0 \arrow[r,] &\operatorname{Ker}{\Phi^{2 \delta}_M}\arrow[r,""] &M \arrow[r,"\overline{q}"] &\operatorname{Im}\Phi^{2 \delta}_M \arrow[r,""]& 0\end{tikzcd}  
\end{equation}
Since $\Phi^{2 \delta}_M$ is by definition diagonal, the lower short exact sequence splits into the direct sum of sequences of the form:
\begin{center}
  \begin{tikzcd}  0 \arrow[r,] &\operatorname{Ker}{\Phi^{2 \delta}_{M_i}}\arrow[r,""] &M_i \arrow[r,"\overline{q}_i"] &\operatorname{Im}\Phi^{2 \delta}_{M_i} \arrow[r,""]& 0\end{tikzcd}  
\end{center}
for each indecomposable summand $M_i$ of $M$. An explicit calculation shows that if $\ell(M_i) \leq 2 \delta$, then $\operatorname{Im}\Phi^{2 \delta}_{M_i}=0$. If $\ell(M_i) > 2 \delta$ then the module $\operatorname{Im}\Phi^{2 \delta}_{M_i}$ is the indecomposable object on the same coray as $M_i$ such that $\ell(\operatorname{Im}\Phi^{2 \delta}_{M_i}) = \ell(M_i)-2 \delta$, and the morphism $\overline{q}_i$ factors down the coray between $M_i$ and $\operatorname{Im}\Phi^{2 \delta}_{M_i}$. In particular, the number of indecomposable summands of $\operatorname{Im}\Phi^{2 \delta}_{M}$ is the same as the number of indecomposable summands of $M$ which have length greater than $2 \delta$.

Let $\mathsf{C}$ denote the coray ending at a quasi-simple $S'$ and let $S=\tau^{-1}S'$. For any module $X$, considered as a complex concentrated at degree zero in the derived category, Serre duality implies $\hom(X,S') = \hom(S[-1],X)$. In particular, $\hom(S[-1],X)$ also counts the number of indecomposable summands of $X$ on $\mathsf{C} $. Applying $\Hom(S[-1],-)$ to the triangles induced by the commutative diagram~(\ref{diagram}), and using the fact that the module category is hereditary, we obtain the following commutative square in which all maps are surjective. 
\begin{center}
  \begin{tikzcd} \Hom{(S[-1],M)} \arrow[r,"q_*"]\arrow[d,"="] & \Hom{(S[-1],\operatorname{Im}{f})} \arrow[d,"\psi_*"]\\ 
 \Hom{(S[-1],M)}\arrow[r,""] &\Hom{(S[-1],\operatorname{Im}{\Phi^{2 \delta}_{M}})} \end{tikzcd}  
\end{center}
Considering the dimensions of these spaces we can conclude that the number of indecomposable summands of $\operatorname{Im}{f}$ on $\mathsf{C}$ is greater than or equal to the number of indecomposable summands of $M$ on $\mathsf{C}$ of length $> 2\delta$. Therefore any summand of $M$ which is not matched to a summand of $\operatorname{Im}{f}$ under the induced matching must be $2 \delta$-trivial. In other words, $ (\mathcal{B}_M)_{2\delta} \subseteq \operatorname{Coim}\eta_f$. A dual argument shows that $(\mathcal{B}({\tau^{-\delta}N}))_{2\delta} \subseteq \operatorname{Im}\eta_f$, so $(\mathcal{B}({N}))_{2\delta} \subseteq \operatorname{Im}(\theta_\delta \circ\eta_f)$.

It remains to prove that if $M_i$ and $N_i$ are summands of $M$ and $N$ respectively, such that $\eta_f(\gamma_{M_i}) = \gamma_{\tau^{-\delta}N_i}$, then $\gamma_{M_i}$ and $ \gamma_{N_i}$ are $\delta$-equivalent. Suppose $M_0$ is a $2 \delta$-trivial indecomposable summand of $M$ which is in $\operatorname{Coim}\eta_f$. Let $C_0$ denote the summand of $\operatorname{Im}f$ to which $M_0$ is matched. It follows from Theorem~\ref{StructureThm} that $0<\ell(C_0) \leq \ell(M_0) \leq 2 \delta$ so $\underline{\gamma}_{C_0} = t^{k} \underline{\gamma}_{M_0}$ for $k = \ell(M_0) -\ell(C_0) \leq 2\delta$. 

Now consider the indecomposable summands $M_i$ of $M$ on $\mathsf{C}$ which have length $\ell(M_i)>2 \delta$. We order them so $\ell(M_1) \geq \ell(M_2)\geq \dots \geq \ell(M_m)>2\delta$. Similarly, we order the indecomposable summands $C_i$ of $\operatorname{Im}{f}$ so that $\ell(C_1) \geq \ell(C_2)\geq \dots \geq \ell(C_k)$. We note that $m \leq k$ and by definition $M_i$ is matched with $C_i$ for each $i=1, \dots m$. We say that the composition $\overline{q_i}: M_i \hookrightarrow M {\to} \operatorname{Im}\Phi^{2 \delta}_M {\twoheadrightarrow} \operatorname{Im}\Phi^{2 \delta}_{M_i}$ maps non-trivially via  some $C_j$ if the composition
\[M_i \hookrightarrow M \stackrel{q}{\to} \operatorname{Im}f \twoheadrightarrow C_j \hookrightarrow \operatorname{Im}f  \stackrel{\psi}{\to} \operatorname{Im}\Phi^{2 \delta}_M {\twoheadrightarrow} \operatorname{Im}\Phi^{2 \delta}_{M_i} \] is non-zero, where the monomorphisms are inclusions of summands and the epimorphisms are projections onto summands. If this holds then, since $\overline{q_i}$ factors along the coray $\mathsf{C}$, it follows that $C_j$ lies between $M_i$ and $\operatorname{Im}\Phi^{2 \delta}_{M_i}$ on $\mathsf{C}$. Therefore $\ell(M_i)-\ell(C_j) \leq 2\delta$ and $\underline{\gamma}_{C_j} = t^{k} \underline{\gamma}_{M_i}$ for $k = \ell(M_i) -\ell(C_j) \leq 2\delta$. 
We extend this idea. For any $n \in \{1, \dots , m\}$, we define \[ w_n = \max\{ j \in \{1, \dots , k \} \mid \bigoplus_{i=1}^n M_i \hookrightarrow M \stackrel{\overline{q}}{\to} \operatorname{Im}\Phi^{2 \delta}_M \twoheadrightarrow \bigoplus_{i=1}^n \operatorname{Im}\Phi^{2 \delta}_{M_i} \text{  maps non-trivially via }  C_j\}.\]
We claim this is well defined and that $w_n \geq n$ for all $n \in \{1, \dots , m\}$. We prove this by induction. It follows from diagram~(\ref{diagram}) that $\overline{q_1}$ factors via $\operatorname{Im}{f}$ and therefore maps via at least one of its summands. Therefore $w_1$ is well defined and $w_1 \geq 1$. Now suppose that $w_n \geq n$ for some $n \in \{1, \dots , m-1\}$. It is clear from the definition that $w_{n+1}\geq w_n \geq n$. It is therefore sufficient to show that $w_{n+1} \neq n$. Let $M'= \bigoplus_{i=1}^{n+1} M_i$. If $w_{n+1} =n$ this would imply that the map $M' \hookrightarrow M \to \operatorname{Im}\Phi^{2 \delta}_M \twoheadrightarrow \operatorname{Im}\Phi^{2 \delta}_{M'}$ factors via $C'= \bigoplus_{i=1}^{n} C_i$. It follows that the map $C' \to \operatorname{Im}\Phi^{2 \delta}_{M'}$ is an epimorphism and therefore, using the hereditary property, there is a surjective map 
\[\Hom{(s[-1],C')} \twoheadrightarrow \Hom{(s[-1],\operatorname{Im}{\Phi^{2 \delta}_{M'}})}.\]
Comparing the dimensions of these two spaces, this would imply that $n>n+1$ which would be a contradiction.

Given any $M_i$, then by construction $\ell(M_i) \geq \ell(C_{k_i}) \geq \ell(\operatorname{Im}{\Phi^{2 \delta}_{M_i}})$. Since $w_i \geq i$ then $\ell(C_{w_i}) \leq \ell(C_i)$ and it follows from Theorem~\ref{StructureThm} that $\ell(M_i) \geq \ell(C_{i})$. Therefore, $C_i$ lies between $M_i$ and $\operatorname{Im}\Phi^{2 \delta}_{M_i}$ on $\mathsf{C}$, so $\underline{\gamma}_{C_i} = t^{k} \underline{\gamma}_{M_i}$ for $k = \ell(M_i) -\ell(C_i) \leq 2\delta$.

A dual argument shows that for any summand $C_i$ of $\operatorname{Im}f$ then $\underline{\gamma}_{\tau^{-\delta}N_i} = s^{k} \underline{\gamma}_{C_i}$ for $k = \ell(N_i) -\ell(C_i) \leq 2\delta$. Putting these parts together we see that $\underline{\gamma}_{\tau^{-\delta}N_i} = s^{k_2} \underline{\gamma}_{C_i} = s^{k_2} t^{k_1} \underline{\gamma}_{M_i} =  t^{k_1}s^{k_2} \underline{\gamma}_{M_i}$ for $0 \leq k_1, k_2  \leq 2\delta$.
For $j\in \{1,2\}$, let $n_j = \min \{ k_j, \delta \}$. Then $0 \leq \delta-n_j, k_j-n_j \leq \delta$ and $ t^{\delta -n_1} s^{\delta-n_2} \underline{\gamma}_{N_i} \simeq t^{k_1-n_1}s^{k_2-n_2} \underline{\gamma}_{M_i}$ is nontrivial since $k_1< \ell(M_i)$.

\end{proof}

\begin{proposition}
        Suppose $M,N$ are persistence modules which are both in the same regular component. If $\mathcal{B}(M)$ and $ \mathcal{B}(N)$ are $\delta$-matched, then $M$ and $N$ are $\delta$-interleaved and so \[ d_I(M,N) \leq d_B(\mathcal{B}(M), \mathcal{B}(N)).\]
\end{proposition}
\begin{proof}
     By definition there exists a matching $\eta : \mathcal{B}_M \not\to \mathcal{B}_N$ such that
    \[ (\mathcal{B}_M)_{2\delta} \subseteq \operatorname{Coim}\eta, \qquad (\mathcal{B}_N)_{2\delta} \subseteq \operatorname{Im}\eta. \]
    Let $M_i$ and $N_i$ be summands of $M$ and $N$ respectively such that $\eta(\underline{\gamma}_{M_i}) = \underline{\gamma}_{N_i}$. By assumption $\underline{\gamma}_{M_i}$ and $\underline{\gamma}_{N_i}$ are $\delta$-equivalent, so there exist integers $0 \leq n_1, n_2, m_1, m_2 \leq \delta$ such that ${t}^{m_1} {s}^{m_2}\underline{\gamma}_{M_i} = {t}^{n_1} {s}^{n_2}\underline{\gamma}_{N_i} $ is non-trivial. Using Lemma~\ref{partcommutivity} it follows that \[{t}^{\delta-n_1 +m_1} {s}^{\delta -n_2 +m_2}\underline{\gamma}_{M_i} = {t}^{\delta} {s}^{\delta}\underline{\gamma}_{N_i} = \underline{\gamma}_{\tau^{-\delta}N_i}. \] 
    As in the proof of Proposition~\ref{nonregisom}, we can define a morphism $f_i: M_i \to \tau^{-\delta}N_i$ which is a composition of irreducible morphisms (arrows) in the AR quiver which goes $\delta -n_2 +m_2\geq 0$ steps diagonally up and then $\delta -n_1 +m_1\geq 0$ steps diagonally down. This morphism is zero if $ \delta-n_1 +m_1 \geq \ell(M_i)$ and non-zero otherwise.
    By a symmetric argument, we can construct a morphism $g_i:N_i \to \tau^{-\delta}M_i$ which goes $\delta -m_2 +n_2\geq 0$ steps diagonally up and then $\delta -m_1 +n_1\geq 0$ steps diagonally down. Then using the mesh relations it follows that $\tau^{-\delta}g_i \circ f_i$ is the morphism obtained by composing irreducible morphisms that go $2\delta$ steps diagonally up and $2\delta$ steps diagonally down in the AR quiver. By definition and the mesh relations, this is precisely $\Phi_{M_i}^{2\delta}$.     
    Similarly,  $\tau^{-\delta}f_i \circ g_i = \Phi_{N_i}^{2\delta}$. 
    Note that for any summand $M_0$ of $M$ such that $\underline{\gamma}_{M_0} \notin \operatorname{Coim}\eta$ then $\ell(M_0) \leq 2\delta$ and so $\Phi_{M_0}^{2\delta} = 0$. For any summand $N_0$ of $N$ such that $\underline{\gamma}_{N_0} \notin \operatorname{Im}\eta$ then $\ell(N_0) \leq 2\delta$ and so $\Phi_{N_0}^{2\delta} = 0$. We can then define morphisms $f: M \to N$ and $g: N \to M$ as 
\[
\begin{pmatrix}
\begin{matrix}
f_1 & 0 & \cdots & 0 \\
0 & f_2 & \cdots & 0 \\
\vdots & \vdots & \ddots & \vdots \\
0 & 0 & \cdots & f_n
\end{matrix} 
& \vline \;
\mathbf{\underline{0}} \\
\hline
\mathbf{\underline{0}} & \vline \;
\mathbf{\underline{0}} \\
\end{pmatrix} \qquad  \text{and }\qquad
\begin{pmatrix}
\begin{matrix}
g_1 & 0 & \cdots & 0 \\
0 & g_2 & \cdots & 0 \\
\vdots & \vdots & \ddots & \vdots \\
0 & 0 & \cdots & g_n
\end{matrix} 
& \vline \;
\mathbf{\underline{0}} \\
\hline
\mathbf{\underline{0}} & \vline \;
\mathbf{\underline{0}} \\
\end{pmatrix}
\]
    where $\mathbf{\underline{0}}$ denotes a zero matrix. A short calculation then shows that $M$ and $N$ are $\delta$-interleaved with respect to $f$ and $g$.

    \end{proof}

Putting together the previous results, we have proved the following theorem.
\begin{theorem}[Isometry Theorem] \label{isometry} Suppose $M,N$ are persistence modules of type $\tilde{A}$. Then 
\[ d_I(M,N) = d_B(\mathcal{B}(M), \mathcal{B}(N)).\]
    
\end{theorem}

\section{Comparison with Prior Results on Circle-Valued Persistence}\label{BurgheleaBackground}

In this section, we briefly review the level set construction of circle-valued persistence modules from \cite{Haller} and provide a comparative dictionary relating the barcodes and Jordan blocks defined there with those introduced in this paper.

We consider circle-valued maps on topological spaces \( X \), where \( X \) is assumed to be a compact absolute neighborhood retract (ANR), a class that includes simplicial complexes. We impose conditions on the map \( f \), analogous to those required of scalar-valued Morse functions. Specifically, we say that a map \( f: X \to S^1 \) is \emph{tame} if:
\begin{itemize}
    \item every fiber \( X_\theta = f^{-1}(\theta) \) is a deformation retract of some open neighborhood; and
    \item away from a finite set of angles \( \Sigma = \{\theta_1, \theta_2, \dots, \theta_r\} \subset S^1 \), the restriction of $ f$ to ${X \setminus f^{-1}(\Sigma)} $ is a fibration.
\end{itemize}

Given a tame, continuous circle-valued map \( f: X \to S^1 \), we consider its infinite cyclic cover \( \tilde{f}: \tilde{X} \to \mathbb{R} \). This lifted map behaves analogously to a real-valued Morse function, allowing for the definition of critical points, indices, and gradient flow in the usual way.

Under the tameness assumptions, there exists a finite set of angles \( 0 < \theta_1 \leq \dots \leq \theta_r \leq 2\pi \) at which the homotopy type of the fibers changes. These are the \emph{critical values} of \( f \). We choose regular values \( 0 < t_1 < t_2 < \dots < t_r \), such that \( t_1 < \theta_1 \) and \( \theta_{i-1} < t_i < \theta_i \) for each \( i = 2, \dots, r \). The corresponding \emph{singular fibers} are \( X_i = f^{-1}(\theta_i) \), and the \emph{regular fibers} are \( R_i = f^{-1}(t_i) \). There are continuous maps between these spaces that are well-defined up to homotopy equivalence.

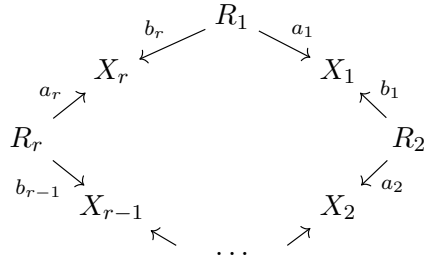
\[
\begin{tikzcd}[cramped, sep=tiny]
	&&& R_1 \\
	& X_r &&&& X_1 \\
	\\
	R_r &&&&&& R_2 \\
	\\
	& X_{r-1} &&&& X_2 \\
	&& {} & \dots & {}
	\arrow["{b_r}"', from=1-4, to=2-2]
	\arrow["{a_1}", from=1-4, to=2-6]
	\arrow["{a_r}", from=4-1, to=2-2]
	\arrow["{b_{r-1}}"', from=4-1, to=6-2]
	\arrow["{b_1}"', from=4-7, to=2-6]
	\arrow["{a_2}", from=4-7, to=6-6]
	\arrow[from=7-3, to=6-2]
	\arrow[from=7-5, to=6-6]
\end{tikzcd}
\]

Taking \( k \)-dimensional homology (for any \( k \geq 0 \)) yields a representation of a quiver \( Q \), specifically an \( \tilde{A}_{2r} \) quiver with a zigzag orientation. The linear maps in the representation are induced by the continuous maps \( a_i \) and \( b_i \).

\vspace{0.5em}
\noindent In \cite{Haller}, the authors define two main types of invariants for circle-valued maps:

\begin{enumerate}
    \item \textbf{Barcodes}: These are finite intervals \( I \subset \mathbb{R} \), classified into four types:
    \begin{itemize}
        \item Closed intervals \( [a,b] \), with \( a \leq b \) ;
        \item Open intervals \( (a,b) \), with \( a < b \);
        \item Left-open, right-closed intervals \( (a,b] \), with \( a < b \);
        \item Left-closed, right-open intervals \( [a,b) \), with \( a < b \).
    \end{itemize}
    
    \item \textbf{Jordan blocks}: These are pairs \( (\lambda, k) \), where \( \lambda \in \mathbf{k} \setminus \{0\} \) and \( k \in \mathbb{Z}_{>0} \). Each Jordan block corresponds to the matrix:
    \[
    \begin{pmatrix}
    \lambda & 1 & 0 & \cdots & 0\\
    0 & \lambda & 1 & \cdots & 0\\
    \vdots & \ddots & \ddots & \ddots & \vdots \\
    0 & \cdots & 0 & \lambda & 1\\
    0 & \cdots & 0 & 0 & \lambda
    \end{pmatrix}
    \]
\end{enumerate}

\noindent These invariants correspond to the indecomposable representations of $Q$, the \( \tilde{A}_{2r} \) type quiver, in the following way. Let \( \tilde{Q} \) be the doubly infinite zigzag quiver of type \( A \) which is the universal cover of \( Q \). We consider representations of \( \tilde{Q} \) that are either \emph{periodic} or have \emph{finite support}, referred to as {`good' representations} in \cite{Haller}. These induce representations of \( Q \).

An indecomposable representation of \( Q \) is of \emph{interval type} if it arises from a finitely supported \( \tilde{Q} \)-representation. It corresponds to a:
\begin{itemize} 
    \item \emph{closed interval} if its support lies between two singular vertices;
    \item \emph{open interval} if its support lies between two regular vertices;
    \item \emph{half-open interval} if its support lies between a regular and a singular vertex.
\end{itemize}
Periodic representations of \( \tilde{Q} \) descend to regular representations of \( Q \) and correspond to Jordan blocks.

These correspond to arcs and closed curves on the geometric model as follows.

\begin{table}[h]
\centering
\begin{tabular}{l|l|l}
\textbf{Object from \cite{Haller}} & \textbf{Location in AR Quiver} & \textbf{Object in Geometric Model} \\ \hline
Closed interval                    & Preprojective component         & Arc between different boundary components \\
Open interval                      & Preinjective component          & Arc between different boundary components \\
Half-open interval                 & Rank \( p \) or \( q \) tube     & Arc between the same boundary component \\
Jordan block \( (\lambda, l) \)    & Homogeneous tube                & \( (\underline{\gamma}, \lambda, l) \) for closed curve \( \gamma \)
\end{tabular}
\end{table}
\noindent Note that the arcs corresponding to objects in the preprojective and preinjective components are distinguished by their gradings.

\section{Examples} \label{Examples}

In this section, we present some simple examples that illustrate the theory that we have developed. We highlight some examples of circle-valued maps that are finite distances apart and give some that are infinitely far apart.
\subsection{Winding number}

Let us consider the following two maps $X \to S^1$ to the circle. On the left, the space $X$ consists of two circles, while on the right, the space $X$ consists of one circle, but the map has winding number $2$. Neither of the maps has a singular value, but we will treat $s$ as a singular value, and $t$ as a regular value. Thus, in both cases, we will be dealing with representations of the Kronecker quiver.

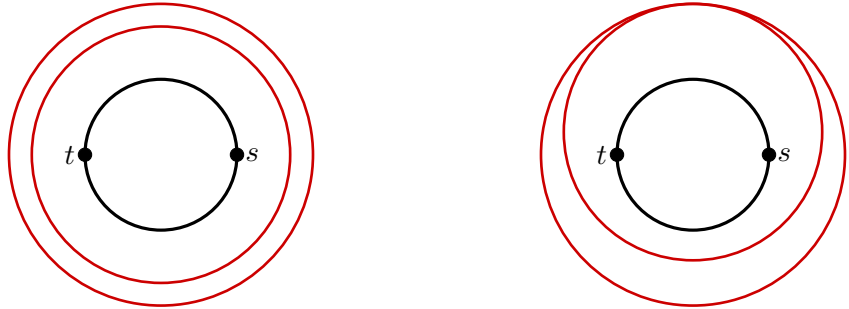
\definecolor{ccqqqq}{rgb}{0.8,0,0}
\begin{center}
\begin{tikzpicture}

\draw [line width=1.2pt] (-7,0) circle (1cm);
\draw [line width=1pt,color=ccqqqq] (-7,0) circle (1.7cm);
\draw [line width=1pt,color=ccqqqq] (-7,0) circle (2cm);

\draw [fill] (-8,0) circle (2.5pt);
\draw (-8.2,0) node {$t$};
\draw [fill] (-6, 0) circle (2.5pt);
\draw (-5.8,0) node {$s$};

\draw [line width=1.2pt] (0,0) circle (1cm);
\draw [line width=1pt,color=ccqqqq] (0,0.3) circle (1.7cm);
\draw [line width=1pt,color=ccqqqq] (0,0) circle (2cm);

\draw [fill] (-1,0) circle (2.5pt);
\draw (-1.2,0) node {$t$};
\draw [fill] (1, 0) circle (2.5pt);
\draw (1.2,0) node {$s$};

\end{tikzpicture}
\end{center}

We write down the representations $M_L$ and $M_R$ over $\mathbf{k}$ that are the zero-dimensional circle-valued persistence modules in the two examples:

\begin{minipage}{0.5\textwidth}
\[
\begin{tikzcd}
    \mathbf{k}^2 \arrow[r, shift left=1ex, "\begin{bmatrix} 1 \temp 0 \\ 0 \temp 1 \end{bmatrix}"{above}]
        \arrow[r, shift right=1ex, "\begin{bmatrix} 1 \temp 0 \\ 0 \temp 1 \end{bmatrix}"{below}]
    & \mathbf{k}^2
\end{tikzcd}
\]
\end{minipage}%
\begin{minipage}{0.5\textwidth}
\[
\begin{tikzcd}
    \mathbf{k}^2 \arrow[r, shift left=1ex, "\begin{bmatrix} 0 \temp 1 \\ 1 \temp 0 \end{bmatrix}"{above}]
        \arrow[r, shift right=1ex, "\begin{bmatrix} 1 \temp 0 \\ 0 \temp 1 \end{bmatrix}"{below}]
    & \mathbf{k}^2
\end{tikzcd}
\]
\end{minipage}
Both representations decompose into the sum of two indecomposable summands:

\begin{minipage}{0.5\textwidth}
\[\left( \begin{tikzcd}
    \mathbf{k} \arrow[r, shift left=1ex, "1"{above}]
        \arrow[r, shift right=1ex, "1"{below}]
    & \mathbf{k}
\end{tikzcd} \right) \bigoplus \left( \begin{tikzcd}
    \mathbf{k} \arrow[r, shift left=1ex, "1"{above}]
        \arrow[r, shift right=1ex, "1"{below}]
    & \mathbf{k}
\end{tikzcd} \right) \]
\end{minipage}%
\begin{minipage}{0.5\textwidth}
\[
\left( \begin{tikzcd}
    \mathbf{k} \arrow[r, shift left=1ex, "1"{above}]
        \arrow[r, shift right=1ex, "1"{below}]
    & \mathbf{k}
\end{tikzcd} \right) \bigoplus \left( \begin{tikzcd}
    \mathbf{k} \arrow[r, shift left=1ex, "-1"{above}]
        \arrow[r, shift right=1ex, "1"{below}]
    & \mathbf{k}
\end{tikzcd} \right)
\]
\end{minipage}

\noindent All four summands are regular representations, and they all have length $1$ and sit at the mouths of rank $1$ tubes. The corresponding barcodes both contain band objects; in the left-hand example this is two copies of $(\underline{\gamma},1,1)$, while the right-hand example has $(\underline{\gamma},1,1)$ and $(\underline{\gamma},-1,1)$, where $\gamma$ is the closed curve on the annulus. It is clear that $M_L$ and $M_R$ are not isomorphic since their second summands are not - they lie in different tubes. Therefore $M_L$ and $M_R$ are not $0$-interleaved and their corresponding barcodes are not $0$-matched.

Since all summands of $M_L$ and $M_R$ are at the mouths of tubes, $\Phi_{M_L}^2=0=\Phi_{M_R}^2$. It follows that $M_L$ and $M_R$ are $1$-interleaved with respect to the zero morphisms $M_L \stackrel{0}{\to} \tau^{-1}M_R$ and $M_R \stackrel{0}{\to} \tau^{-1}M_L$. Therefore $ d_I(M_L,M_R) =1$. 

We can also calculate the bottleneck distance. We check that there is a $1$-matching. Since no objects in either barcode have length $>2$ it follows that the conditions on $\operatorname{Coim} \eta$ and $\operatorname{Im} \eta$ are empty. Therefore, one possible choice of matching $\eta$ is the zero map. The other condition is then also empty, and so a $1$-matching trivially exists. It follows that $ d_B(\mathcal{B}(M_L),\mathcal{B}(M_R)) =1$ as expected. Since the first summands of $M_L$ and $M_R$ are isomorphic, we could also have chosen the $\eta$ that matches these summands, which are certainly $1$-equivalent. 

Thus, we are able to compare persistence modules resulting from maps coming from two different spaces and find them to have a finite distance.

\subsection{Homotopy class}

Let $X$ be a figure-of-eight and let us consider two maps from $X$ to the circle. The first map has both loops wound around the circle, while the second only has one. Looking in homological degree $0$, we get the following representations:
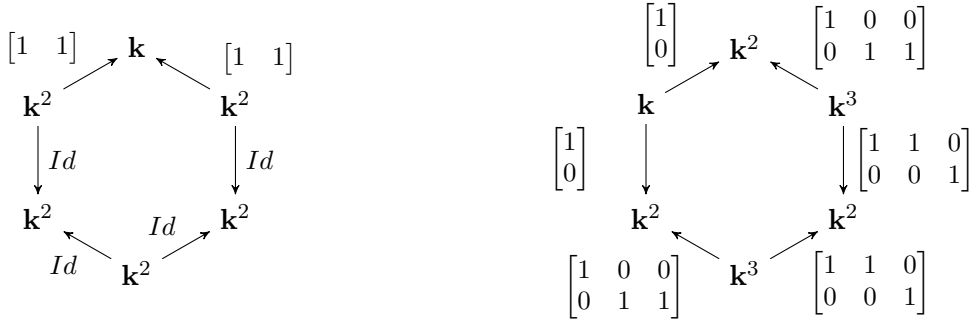
\[
\begin{tikzpicture}
    \def\radius{1.5cm}

    \node[xshift=-4cm] (s1) at (90:\radius)   {$\mathbf{k}$};
    \node[xshift=-4cm] (t1) at (30:\radius)    {$\mathbf{k}^2$};
    \node[xshift=-4cm] (s2) at (-30:\radius)  {$\mathbf{k}^2$};
    \node[xshift=-4cm] (t2) at (-90:\radius) {$\mathbf{k}^2$};
    \node[xshift=-4cm] (s3) at (-150:\radius)  {$\mathbf{k}^2$};
    \node[xshift=-4cm] (t3) at (150:\radius)  {$\mathbf{k}^2$};
    
    \node[xshift=4cm] (s1b) at (90:\radius)   {$\mathbf{k}^2$};
    \node[xshift=4cm] (t1b) at (30:\radius)    {$\mathbf{k}^3$};
    \node[xshift=4cm] (s2b) at (-30:\radius)  {$\mathbf{k}^2$};
    \node[xshift=4cm] (t2b) at (-90:\radius) {$\mathbf{k}^3$};
    \node[xshift=4cm] (s3b) at (-150:\radius)  {$\mathbf{k}^2$};
    \node[xshift=4cm] (t3b) at (150:\radius)  {$\mathbf{k}$};
    \path[->,font=\small]
        (t1) edge node[right of=t1, yshift = 0.2cm] {$\begin{bmatrix} 1 \temp 1 \end{bmatrix}$} (s1)
        (t1) edge node[auto] {$Id$} (s2)
        (t2) edge node[auto] {$Id$} (s2)
        (t2) edge node[auto] {$Id$} (s3)
        (t3) edge node[auto] {$Id$} (s3)
        (t3) edge node[auto] {$\begin{bmatrix} 1 \temp 1 \end{bmatrix}$} (s1);

    \path[->,font=\small]
        (t1b) edge node[auto, right of=t1b, yshift = 0.5cm] {$\begin{bmatrix} 1 \temp 0 \temp 0 \\
                                                0\temp 1 \temp 1\end{bmatrix}$} (s1b)
        (t1b) edge node[auto] {$\begin{bmatrix} 1 \temp 1 \temp 0 \\
                                                0\temp 0 \temp 1 \end{bmatrix}$} (s2b)
        (t2b) edge node[auto, right of=t2b, yshift = -0.5cm] {$\begin{bmatrix} 1 \temp 1 \temp 0 \\
                                                0\temp 0 \temp 1 \end{bmatrix}$} (s2b)
        (t2b) edge node[auto] {$\begin{bmatrix} 1 \temp 0 \temp 0 \\
                                                0\temp 1 \temp 1\end{bmatrix}$} (s3b)
        (t3b) edge node[auto, left of=t3b] {$\begin{bmatrix} 1 \\ 0 \end{bmatrix}$} (s3b)
        (t3b) edge node[auto] {$\begin{bmatrix} 1 \\ 0 \end{bmatrix}$} (s1b);
\end{tikzpicture}
\]
Using GAP \cite{GAP}, we found that the left persistence module decomposes into indecomposable representations with the following dimensions:

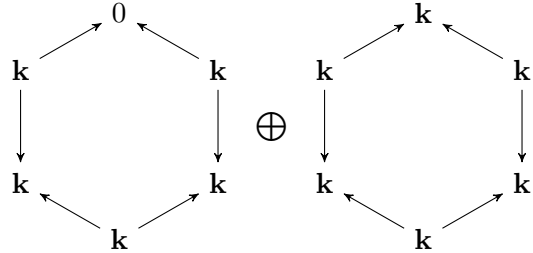
\[
\begin{tikzpicture}
    \def\radius{1.5cm}  
    
    \node[xshift=-2cm] (s1) at (90:\radius)   {$0$};
    \node[xshift=-2cm] (t1) at (30:\radius)    {$\mathbf{k}$};
    \node[xshift=-2cm] (s2) at (-30:\radius)  {$\mathbf{k}$};
    \node[xshift=-2cm] (t2) at (-90:\radius) {$\mathbf{k}$};
    \node[xshift=-2cm] (s3) at (-150:\radius)  {$\mathbf{k}$};
    \node[xshift=-2cm] (t3) at (150:\radius)  {$\mathbf{k}$};
    
    \node[xshift=2cm] (s1b) at (90:\radius)   {$\mathbf{k}$};
    \node[xshift=2cm] (t1b) at (30:\radius)    {$\mathbf{k}$};
    \node[xshift=2cm] (s2b) at (-30:\radius)  {$\mathbf{k}$};
    \node[xshift=2cm] (t2b) at (-90:\radius) {$\mathbf{k}$};
    \node[xshift=2cm] (s3b) at (-150:\radius)  {$\mathbf{k}$};
    \node[xshift=2cm] (t3b) at (150:\radius)  {$\mathbf{k}$};

    \node[xshift=-1.5cm] (sum) at (0:\radius) {$\bigoplus$};
    
    \path[->,font=\small]
        (t1) edge node[right] {} (s1)
        (t1) edge node[auto] {} (s2)
        (t2) edge node[auto] {} (s2)
        (t2) edge node[auto] {} (s3)
        (t3) edge node[auto] {} (s3)
        (t3) edge node[auto] {} (s1);
    
    \path[->,font=\small]
        (t1b) edge node[auto] {} (s1b)
        (t1b) edge node[auto] {} (s2b)
        (t2b) edge node[auto] {} (s2b)
        (t2b) edge node[auto] {} (s3b)
        (t3b) edge node[auto] {} (s3b)
        (t3b) edge node[auto] {} (s1b);
\end{tikzpicture}
\]
while the right persistence module decomposes into indecomposable representations with the following dimensions:

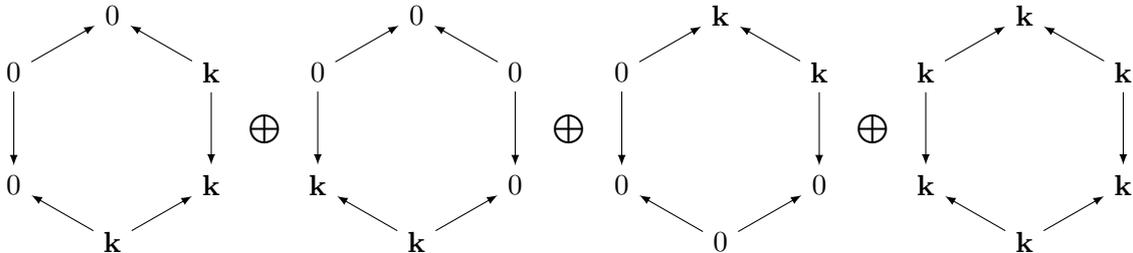
\[
\begin{tikzpicture}[>=latex]
    \def\radius{1.5cm}  
    
    \node[xshift=-2cm] (s1) at (90:\radius)   {$0$};
    \node[xshift=-2cm] (t1) at (30:\radius)    {$0$};
    \node[xshift=-2cm] (s2) at (-30:\radius)  {$0$};
    \node[xshift=-2cm] (t2) at (-90:\radius) {$\mathbf{k}$};
    \node[xshift=-2cm] (s3) at (-150:\radius)  {$\mathbf{k}$};
    \node[xshift=-2cm] (t3) at (150:\radius)  {$0$};
    
    \node[xshift=2cm] (s1b) at (90:\radius)   {$\mathbf{k}$};
    \node[xshift=2cm] (t1b) at (30:\radius)    {$\mathbf{k}$};
    \node[xshift=2cm] (s2b) at (-30:\radius)  {$0$};
    \node[xshift=2cm] (t2b) at (-90:\radius) {$0$};
    \node[xshift=2cm] (s3b) at (-150:\radius)  {$0$};
    \node[xshift=2cm] (t3b) at (150:\radius)  {$0$};

    \node[xshift=-1.5cm] (sum1) at (0:\radius) {$\bigoplus$};

    \node[xshift=-6cm] (s1c) at (90:\radius)   {$0$};
    \node[xshift=-6cm] (t1c) at (30:\radius)    {$\mathbf{k}$};
    \node[xshift=-6cm] (s2c) at (-30:\radius)  {$\mathbf{k}$};
    \node[xshift=-6cm] (t2c) at (-90:\radius) {$\mathbf{k}$};
    \node[xshift=-6cm] (s3c) at (-150:\radius)  {$0$};
    \node[xshift=-6cm] (t3c) at (150:\radius)  {$0$};

    \node[xshift=-5.5cm] (sum2) at (0:\radius) {$\bigoplus$};

    \node[xshift=6cm] (s1d) at (90:\radius)   {$\mathbf{k}$};
    \node[xshift=6cm] (t1d) at (30:\radius)    {$\mathbf{k}$};
    \node[xshift=6cm] (s2d) at (-30:\radius)  {$\mathbf{k}$};
    \node[xshift=6cm] (t2d) at (-90:\radius) {$\mathbf{k}$};
    \node[xshift=6cm] (s3d) at (-150:\radius)  {$\mathbf{k}$};
    \node[xshift=6cm] (t3d) at (150:\radius)  {$\mathbf{k}$};

    \node[xshift=2.5cm] (sum3) at (0:\radius) {$\bigoplus$};
    
    \path[->,font=\small]
        (t1) edge node[right] {} (s1)
        (t1) edge node[auto] {} (s2)
        (t2) edge node[auto] {} (s2)
        (t2) edge node[auto] {} (s3)
        (t3) edge node[auto] {} (s3)
        (t3) edge node[auto] {} (s1);
    
    \path[->,font=\small]
        (t1b) edge node[auto] {} (s1b)
        (t1b) edge node[auto] {} (s2b)
        (t2b) edge node[auto] {} (s2b)
        (t2b) edge node[auto] {} (s3b)
        (t3b) edge node[auto] {} (s3b)
        (t3b) edge node[auto] {} (s1b);

        \path[->,font=\small]
        (t1c) edge node[auto] {} (s1c)
        (t1c) edge node[auto] {} (s2c)
        (t2c) edge node[auto] {} (s2c)
        (t2c) edge node[auto] {} (s3c)
        (t3c) edge node[auto] {} (s3c)
        (t3c) edge node[auto] {} (s1c);

        \path[->,font=\small]
        (t1d) edge node[auto] {} (s1d)
        (t1d) edge node[auto] {} (s2d)
        (t2d) edge node[auto] {} (s2d)
        (t2d) edge node[auto] {} (s3d)
        (t3d) edge node[auto] {} (s3d)
        (t3d) edge node[auto] {} (s1d);
\end{tikzpicture}
\]

An explicit calculation of the dimension vectors of indecomposable representations shows that the leftmost components in each case are preinjective and can be matched to each other. The remaining representations are all in the regular components. In particular, they all have finite length and so can potentially be matched to zero. As a consequence, we deduce that the interleaving distance is finite.
Thus, it is possible that two maps from the same space $X$ with different cohomology classes can induce persistence modules which are a finite distance apart.

\subsection{An example of an infinite interleaving distance}
Let us compare the identity map from the circle to the circle with the following map, with its associated persistence module:

\
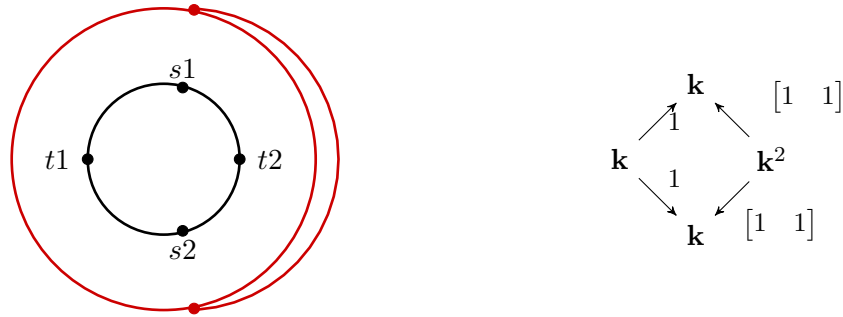
\[
\begin{tikzpicture}
\draw [line width=1pt, color=ccqqqq] (0,0) circle (2cm);
\draw [shift={(0.3,0)}, line width=1pt, color=ccqqqq]  plot[domain=-1.5:1.5,variable=\t]({2*cos(\t r)},{2*sin(\t r)});
\draw [line width=1pt] (0,0) circle (1cm);

\draw [fill] (-1,0) circle (2pt);
\draw (-1.4,0) node {$t1$};
\draw [fill] (1,0) circle (2pt);
\draw (1.4,0) node {$t2$};

\draw [fill] (0.25,0.95) circle (2pt);
\draw (0.25,1.2) node {$s1$};
\draw [fill] (0.25,-0.95) circle (2pt);
\draw (0.25,-1.2) node {$s2$};

\draw [fill, color=ccqqqq] (0.4,1.98) circle (2pt);
\draw [fill, color=ccqqqq] (0.4,-1.98) circle (2pt);

  \def\radius{1cm}  
    
    \node[xshift=7cm] (s1) at (90:\radius)   {$\mathbf{k}$};
    \node[xshift=7cm] (t2) at (0:\radius)    {$\mathbf{k}^2$};
    \node[xshift=7cm] (s2) at (-90:\radius)  {$\mathbf{k}$};
    \node[xshift=7cm] (t1) at (180:\radius) {$\mathbf{k}$};
  
    \path[->,font=\small]
        (t1) edge node[right] {$1$} (s1)
        (t1) edge node[auto] {$1$} (s2)
        (t2) edge node[auto] {$\begin{bmatrix} 1 \temp 1 \end{bmatrix}$} (s2)
        (t2) edge node[auto, right of=t2, yshift = 0.3cm] {$\begin{bmatrix} 1 \temp 1 \end{bmatrix}$} (s1);

\end{tikzpicture}
\]
This persistence module decomposes into the following summands:

\
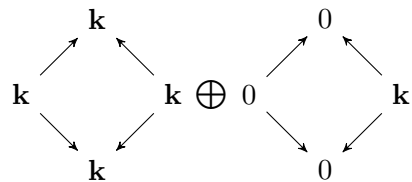
\[
\begin{tikzpicture}

  \def\radius{1cm}  
    
    \node[xshift=-1.5cm] (s1) at (90:\radius)   {$\mathbf{k}$};
    \node[xshift=-1.5cm] (t2) at (0:\radius)    {$\mathbf{k}$};
    \node[xshift=-1.5cm] (s2) at (-90:\radius)  {$\mathbf{k}$};
    \node[xshift=-1.5cm] (t1) at (180:\radius) {$\mathbf{k}$};

    \node[xshift=-1cm] (sum) at (0:\radius) {$\bigoplus$};

    \node[xshift=1.5cm] (s1a) at (90:\radius)   {$0$};
    \node[xshift=1.5cm] (t2a) at (0:\radius)    {$\mathbf{k}$};
    \node[xshift=1.5cm] (s2a) at (-90:\radius)  {$0$};
    \node[xshift=1.5cm] (t1a) at (180:\radius) {$0$};
  
    \path[->,font=\small]
        (t1) edge node[right] {} (s1)
        (t1) edge node[auto] {} (s2)
        (t2) edge node[auto] {} (s2)
        (t2) edge node[auto] {} (s1);

    \path[->,font=\small]
        (t1a) edge node[right] {} (s1a)
        (t1a) edge node[auto] {} (s2a)
        (t2a) edge node[auto] {} (s2a)
        (t2a) edge node[auto] {} (s1a);

\end{tikzpicture}
\]
While the first summand would match up with the representation of the identity map, the second summand is an injective module, and hence infinitely far away from the zero module. This can also be seen from the corresponding barcode, as the arc corresponding to this module has endpoints on both boundary components and therefore cannot be trivalised by moving these end points along the boundaries. Thus, the persistence modules corresponding to these two maps are infinitely far from each other. Note that the spaces involved are different.

\addtocontents{toc}{\protect{\setcounter{tocdepth}{-1}}}

\enlargethispage{3ex}
\bigskip
\noindent

\smallskip
\noindent
\resizebox{\textwidth}{!}{%
\begin{tabular}{@{}l@{}}
Nathan Broomhead, Mathematical Sciences, University of Plymouth, Drake Circus, Plymouth, PL4 8AA, United Kingdom \\
Mariam Pirashvili, Mathematical Sciences, University of Plymouth, Drake Circus, Plymouth, PL4 8AA, United Kingdom 
\end{tabular}
}
\end{document}